\documentclass[10pt]{article}
\usepackage{amssymb,amsfonts,amsmath, amsthm, amsbsy}
\usepackage{color,graphicx, subcaption, placeins, array, mathtools, url, color, rotating} 
\usepackage[text={6.75in,9.5in},centering,letterpaper]{geometry}
\usepackage[hypertexnames=false,colorlinks=true,urlcolor=blue,linkcolor=blue,citecolor=blue]{hyperref}
\usepackage{tikz}
\usepackage[shortlabels]{enumitem}

\everymath={\displaystyle}
 \numberwithin{equation}{section}

\newtheorem{Theorem}{Theorem}[section]

\newtheorem{Lemma}[Theorem]{Lemma}
\newtheorem{Proposition}[Theorem]{Proposition}

\newtheorem{Remark}[Theorem]{Remark}

\def\Re{\mathop{\mathrm{Re}}}

\title{ Large amplitude radially symmetric spots and gaps in a dryland ecosystem model}
\author{Eleanor Byrnes\thanks{Department of Applied Mathematics, University of Washington, Seattle, WA, USA}\and Paul Carter\thanks{Department of Mathematics, University of California, Irvine, CA, USA}\and Arjen Doelman\thanks{Mathematisch Instituut, Universiteit Leiden, Leiden, NL} \and Lily Liu\thanks{University of Chicago, Chicago, IL, USA }}

\begin{document}
\maketitle

\begin{abstract}
We construct far-from-onset radially symmetric spot and gap solutions in a two-component dryland ecosystem model of vegetation pattern formation on flat terrain, using spatial dynamics and geometric singular perturbation theory. We draw connections between the geometry of the spot and gap solutions with that of traveling and stationary front solutions in the same model. In particular, we demonstrate the instability of spots of large radius by deriving an asymptotic relationship between a critical eigenvalue associated with the spot and a coefficient which encodes the sideband instability of a nearby stationary front. Furthermore, we demonstrate that spots are unstable to a range of perturbations of intermediate wavelength in the angular direction, provided the spot radius is not too small. Our results are accompanied by numerical simulations and spectral computations.
\end{abstract}

\section{Introduction}
The phenomenon of vegetation pattern formation in dryland ecosystems has attracted attention in the last several decades as a mechanism of ecological resilience. While frequently considered to be an early warning sign for desertification~\cite{gowda2018signatures,gowda2014transitions,may1977thresholds, noy1975stability, rietkerk2004self, rietkerk2008regular, rietkerk1997site}, the formation of localized vegetation patches or patterns has also been viewed as a means of evading such critical transitions~\cite{rietkerk2021evasion}. The interaction of infiltration feedback mechanisms and competition for water resources results in the formation of vegetation patches~\cite{macfadyen1950vegetation,rietkerk2002self, schlesinger1990biological,von2001diversity,wilcox2003ecohydrology}. On sloped terrain, one observes vegetation stripes, or bands, aligned perpendicular to the slope, while on flat ground spots, gaps, or disorganized labyrinth patterns are prevalent~\cite{barbier2014case,deblauwe2012determinants, deblauwe2011environmental, gandhi2018influence, ludwig2005vegetation,valentin1999soil}. Spots, gaps, rings or other radially symmetric patterns (sometimes called ``fairy circles") have been observed extensively in drylands in Australia and Africa~\cite{getzin2016discovery, meron2018patterns, ravi2017ecohydrological} (and other ecosystems, such as submarine seascapes~\cite{ruiz2017fairy}), and have served as the focus of many studies of self-organization in ecosystems.

Vegetation pattern formation is frequently modeled by multi-component reaction diffusion systems. In such models, there is a well-developed theory of spot and stripe pattern formation near the onset of Turing instabilities~\cite{gowda2014transitions, gowda2016assessing, siero2015striped}. However, far less is known analytically concerning large amplitude or far-from-onset planar vegetation patterns. A number of studies have considered existence and stability properties of banded vegetation~\cite{BCD,CD,doelman2002homoclinic, sewalt2017spatially}, as well as desertification fronts~\cite{CDLOR,fernandez2019front}, but less is known concerning far-from-onset radially symmetric vegetation patches. Prior work has considered small amplitude radial solutions~\cite{hill2021existence} and $1$D simplifications~\cite{jaibi2020existence}. However, to our knowledge, no rigorous studies exist concerning large amplitude radially symmetric vegetation patches in a dryland ecosystem model. 

We consider the model introduced in~\cite{BCD} 
\begin{align}
\begin{split}\label{eq:modifiedKlausmeier}
U_t &= \Delta U+a-U-UV^2\\
V_t &= \delta^2\Delta V -mV+UV^2(1-bV),
\end{split}
\end{align}
which is a modification of the dryland ecosystem model originally proposed by Klausmeier~\cite{klausmeier1999regular}. This model also coincides with that studied in~\cite{eigentler2021species}, in the case of a single species. Here $U$ represents water availability, $V$ represents vegetation density, and the parameters $a,m,b$ are positive and represent rainfall, mortality, and inverse of soil carrying capacity, respectively. The diffusion coefficient $\delta^2$ represents the ratio of timescales of diffusion of water vs. vegetation. Here, we assume that water diffuses much faster than vegetation, so that $0<\delta\ll1$ is a small parameter \cite{rietkerk2008regular}.  {We note that Klausmeier's model originally had $b=0$. While a singular perturbation analysis of this case is possible (see, e.g.~\cite{CD}), this limit is highly singular, requiring several rescalings and blow up techniques to account for passage near a non-hyperbolic slow manifold. Here, we therefore focus on the case $b>0$.}

We are interested in the formation of radially symmetric vegetation spots (localized vegetation patches surrounded by bare soil) and gaps (localized regions of bare soil in an otherwise uniformly vegetated state) in~\eqref{eq:modifiedKlausmeier}. Exploiting the small parameter $\delta\ll1$, we will use geometric singular perturbation methods to construct radially symmetric solutions through a spatial dynamics approach in the radial coordinate. Our approach is similar to that in~\cite{vHS}, in which the authors construct far-from-onset spot solutions in a 3-component FitzHugh--Nagumo model. However, the nonlinearities in~\eqref{eq:modifiedKlausmeier} introduce complications in the analysis due to the fact that the reduced flow on the resulting slow manifolds is no longer linear as in the case in~\cite{vHS}. However, we will show that radial solutions can still be constructed in~\eqref{eq:modifiedKlausmeier}. 

\begin{Remark}
\label{r:spikesandmore}
In the setting of classical models like Gray-Scott, Gierer-Meinhardt and Schnakenberg, the existence, stability and interactions of radially symmetric localized `spikes' has been studied -- see for instance ~\cite{chen2011stability,kolokolnikov2009spot,wei2013mathematical} and the references therein. However, these patterns differ from the ones considered here since they have a homoclinic nature: unlike the present gaps and spots, the regions in which these spikes are not close to the background states are asymptotically small. The patterns considered here can be seen as two-dimensional (radially symmetric) versions of the one-dimensional `mesa patterns' studied in \cite{kolokolnikov2007self} and in ~\cite{jaibi2020existence} in the setting of vegetation patterns. Like in the present model, and unlike in \cite{vHS,vHS2}, the `slow flows' (for the spatial dynamics) considered in these papers are nonlinear -- which is typically the case for ecosystem models \cite{doelman2022slow}.   
\end{Remark}

The system~\eqref{eq:modifiedKlausmeier} admits (up to) three homogeneous steady states: the desert state $(U,V)=(U_0,V_0):=(a,0)$ and if $\frac{a}{m}>2(b+\sqrt{1+b^2})$, there are two additional vegetated steady states $(U,V)=(U_{1,2},V_{1,2})$ where
\begin{align}\label{eq:steadystates}
U_{1,2}&=m\left(\frac{a}{m}-\frac{V_{1,2}}{1-bV_{1,2}}\right), \qquad V_{1,2}=\frac{\frac{a}{m}\mp \sqrt{\left(\frac{a}{m}\right)^2-4\left(1+\frac{a}{m}b\right)}}{2\left(1+\frac{a}{m}b\right)}
\end{align}
which coincide at the critical value $\frac{a}{m}=2(b+\sqrt{1+b^2})$.

We search for stationary solutions, which are radially symmetric, so that $(U,V)(x,y,t) = (u,v)(r), r\in[0,\infty)$, and $\Delta = \partial_r^2+r^{-1}\partial_r$ under planar radial symmetry, where $r=(x^2+y^2)^{1/2}$. Such solutions satisfy the ordinary differential equation
\begin{align}
\begin{split}\label{eq:klaus_stationary}
0&= u_{rr}+\frac{1}{r}u_r+a-u-uv^2\\
0&= \delta^2\left(v_{rr}+\frac{1}{r}v_r\right) -mv+uv^2(1-bv),
\end{split}
\end{align}
which can be rewritten as a first order non-autonomous system
\begin{align}
\begin{split}\label{eq:slow}
 u_r&=  p\\
  p_r&= -\frac{p}{r}-a+u+uv^2\\
\delta v_r&= q\\
\delta q_r&= -\frac{\delta q}{r} +mv-uv^2(1-bv),
\end{split}
\end{align}
This system admits up to three equilibria given by the steady states above: the desert state $P_0=(a,0,0,0)$, and if $\frac{a}{m}>2(b+\sqrt{1+b^2})$, there are two additional equilibria $P_{1,2}:=(U_{1,2},0,V_{1,2},0)$ corresponding to uniform vegetation.

Spots, gaps, and other localized radially symmetric solutions are constructed as orbits of~\eqref{eq:slow} which are asymptotic as $r\to \infty$ to one of these steady states, and which are bounded as $r\to 0$. In order to find solutions of this system, we construct candidate solution orbits in different subsets of the spatial domain $r\in[0,\infty)$ and glue these together to build a solution on the entire domain. In particular, we construct the solution over three primary regions: the core $r\in[0,r_\mathrm{c}]$, where $r_\mathrm{c} = \mathcal{O}(\delta)$, the far field $r\in[r_I,\infty)$, where $r_I=\mathcal{O}(1)$, and the transition region(s) in between; see Figure~\ref{fig:spot_schem}.

 {
\begin{figure}
\hspace{.03\textwidth}
\begin{subfigure}{.45 \textwidth}
\centering
\includegraphics[width=1\linewidth]{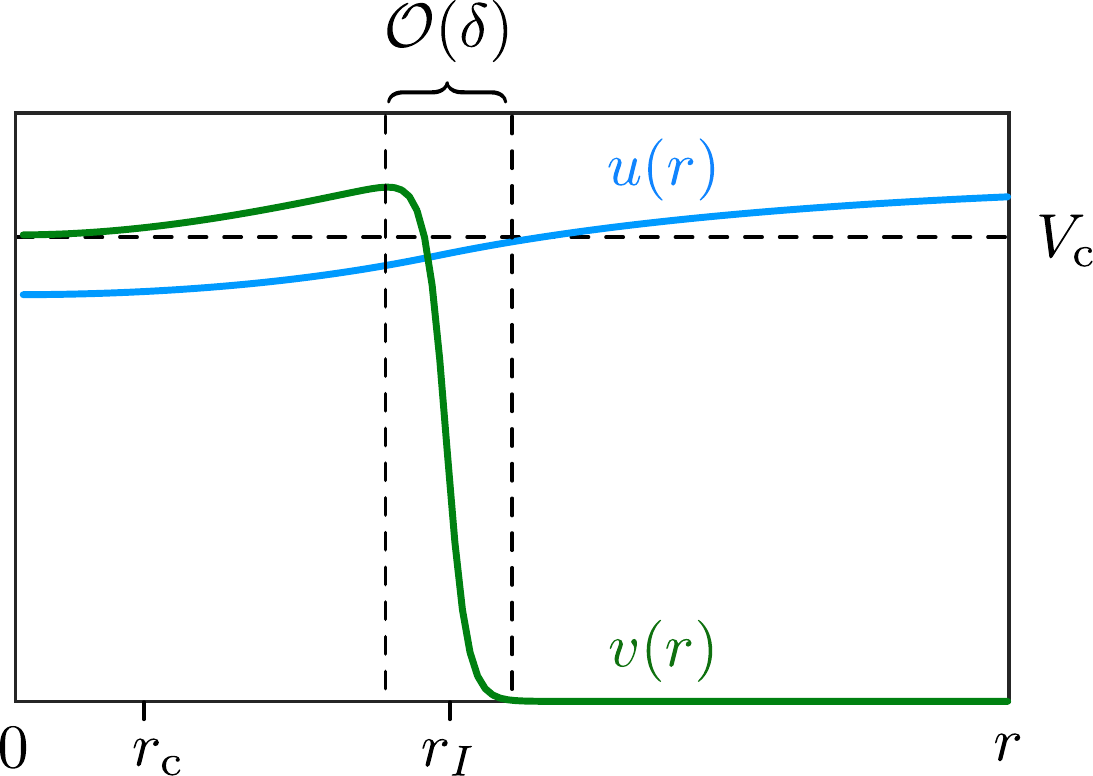}
\end{subfigure}
\hspace{.03\textwidth}
\begin{subfigure}{.45 \textwidth}
\centering
\includegraphics[width=1\linewidth]{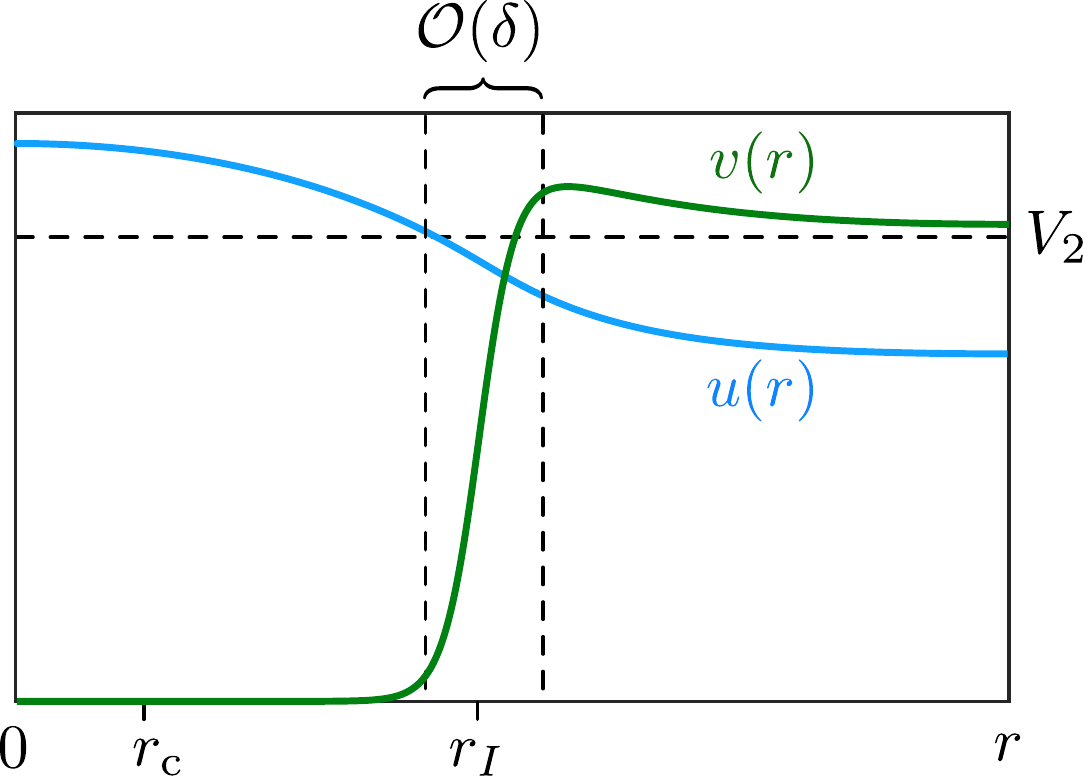}
\end{subfigure}
\caption{Shown is a schematic of a radial profile for a vegetation spot solution (left) and gap solution (right) of~\eqref{eq:klaus_stationary} as in Theorem~\ref{thm:spot_existence} and~\ref{thm:gap_existence}, respectively. The profile contains a sharp transition from the vegetated state to the bare soil state in an interval of width $\mathcal{O}(\delta)$ of the critical radius $r_I=\mathcal{O}(1)$.  }
\label{fig:spot_schem}
\end{figure}}

Our main result concerning the existence of spots is the following.
\begin{Theorem}\label{thm:spot_existence}(Existence of spots)
Fix $a,m,b>0$ satisfying
\begin{align}\label{eq:spot_parameter_restriction}
    \max\left\{\frac{9b}{2}, 4b+\frac{1}{b}\right\}<\frac{a}{m}<\frac{9b}{2}+\frac{2}{b}.
\end{align} Suppose
\begin{align}
\label{eq:spot_condition}
\int_{U_2}^{\frac{9bm}{2}} \frac{u - 2mb + \sqrt{u^2 - 4umb}}{2b^2} \mathrm{d}u&> \frac{1}{2}(a-U_2)^2,    
\end{align}
 {or equivalently,
\begin{align}\label{eq:spot_condition_exp}
    \frac{3}{2} -\log(2)+\frac{U_2}{2m^2}\left(a-U_2\right)+\log\left( \frac{b}{m}\left(a-U_2 \right) \right) - \frac{b}{m}\left(a-U_2 \right) > \frac{2b^2+1}{2m^2}\left(a-U_2 \right)^2
\end{align}}
where $U_2$ is defined as in~\eqref{eq:steadystates}. Then there exists $V_\mathrm{c}(a,b,m)>0$ and $r_I(a,b,m)>0$ such that for all sufficiently small $\delta>0$,~\eqref{eq:modifiedKlausmeier} admits a stationary, bounded, radially symmetric vegetation spot solution $(U,V) = (u_\mathrm{sp},v_\mathrm{sp})(r;a,b,m,\delta)$ satisfying
\begin{align*}
\lim_{r\to 0} v_\mathrm{sp}(r;a,b,m,\delta)&=V_\mathrm{c}(a,b,m)+\mathcal{O}(\delta), \qquad \lim_{r\to \infty} v_\mathrm{sp}(r;a,b,m,\delta)=0
\end{align*}
with a single interface between the vegetated and desert states occurring at the radius $r=r_I(a,b,m)+\mathcal{O}(\delta)$.
\end{Theorem}
Similarly, we have the following theorem concerning the existence of gaps.
\begin{Theorem}\label{thm:gap_existence}(Existence of gaps)
Fix $a,m,b>0$ satisfying~\eqref{eq:spot_parameter_restriction}. Suppose
\begin{align}\label{eq:gap_condition}
\int_{U_2}^{\frac{9bm}{2}} \frac{u - 2mb + \sqrt{u^2 - 4umb}}{2b^2} \mathrm{d}u< \frac{1}{2}(a-U_2)^2    
\end{align}
 {or equivalently,
\begin{align}
    \frac{3}{2} -\log(2)+\frac{U_2}{2m^2}\left(a-U_2\right)+\log\left( \frac{b}{m}\left(a-U_2 \right) \right) - \frac{b}{m}\left(a-U_2 \right) < \frac{2b^2+1}{2m^2}\left(a-U_2 \right)^2
\end{align}
where $U_2$ is defined as in~\eqref{eq:steadystates}. Then there exists $r_I(a,b,m)>0$ and $\nu>0$ such that for all sufficiently small $\delta>0$,~\eqref{eq:modifiedKlausmeier} admits a stationary, bounded, radially symmetric vegetation gap solution $(U,V) = (u_\mathrm{g},v_\mathrm{g})(r;a,b,m,\delta)$ satisfying
\begin{align*}
\lim_{r\to 0} v_\mathrm{g}(r;a,b,m,\delta)&=\mathcal{O}(e^{-\nu/\delta}), \qquad \lim_{r\to \infty} v_\mathrm{g}(r;a,b,m,\delta)=V_2+\mathcal{O}(\delta),
\end{align*}
where $V_2$ is defined as in~\eqref{eq:steadystates}, 
with a single interface between the vegetated and desert states occurring at the radius $r=r_I(a,b,m)+\mathcal{O}(\delta)$.}
\end{Theorem}

 {See Figure~\ref{fig:parameter_condition_abm} for a visualization of the condition~\eqref{eq:spot_parameter_restriction} necessary for the existence of spots/gaps in Theorems~\ref{thm:spot_existence}--\ref{thm:gap_existence}.

\begin{figure}
\centering
\includegraphics[width=0.45\linewidth]{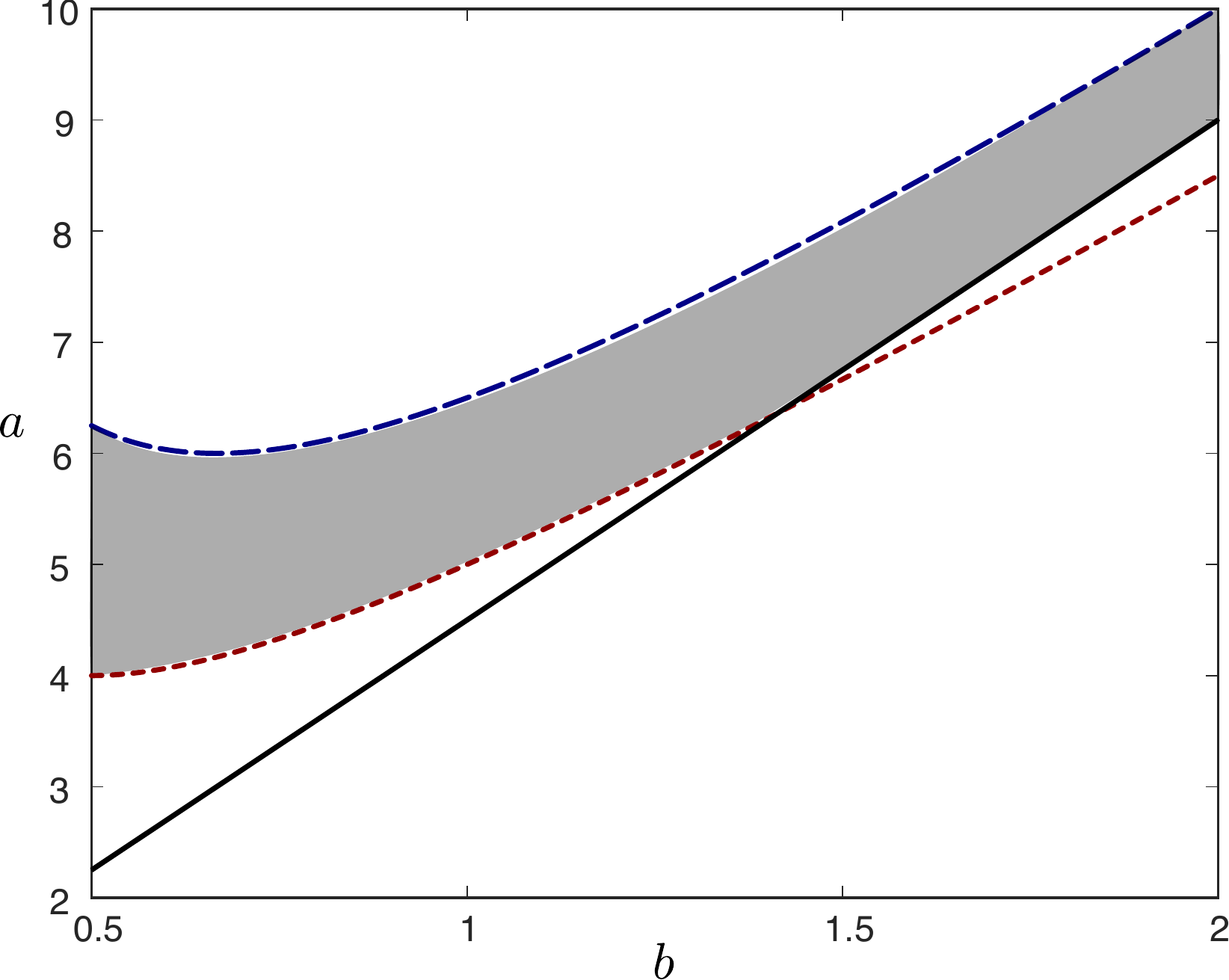}
\caption{ Plotted are the curves $\frac{a}{m} = \frac{9b}{2}$ (black), $\frac{a}{m}  = 4b+\frac{1}{b}$ (dotted red), and $\frac{a}{m}=\frac{9b}{2}+\frac{2}{b}$ (dashed blue). The shaded region corresponds to the region in parameter space where the condition~\eqref{eq:spot_parameter_restriction} is satisfied.  }
\label{fig:parameter_condition_abm}
\end{figure}

}

\begin{Remark}
The methods used in this paper to construct radially symmetric spot and gap solutions can be applied in a similar manner for the construction of solutions such as rings, targets, or other radially symmetric profiles, with different, perhaps more complex, conditions on parameters which ensure their existence. We provide some numerical evidence for the existence of such solutions in~\S\ref{sec:discussion} and describe how one might go about constructing these orbits, but do not provide the lengthy technical details here.
\end{Remark}

The spot and gap solutions of Theorems~\ref{thm:spot_existence}--\ref{thm:gap_existence} will be constructed as heteroclinic orbits for $r\in[0,\infty)$ using geometric singular perturbation theory. The main idea, following~\cite{vHS}, is to find the orbits as intersections of a core manifold of solutions which remain bounded at $r\to0$, and a far-field manifold of solutions which decay to one of the states $P_0$ or $P_2$ as $r\to\infty$. These manifolds are each three-dimensional, and, unlike in~\cite{vHS}, it is not possible to obtain an explicit description of these manifolds due to the fact that the flow on (one of) the slow manifolds of~\eqref{eq:slow} is nonlinear. This introduces complications which are handled through the use of an intermediate scaling and careful qualitative analysis of the nonlinear non-autonomous reduced flow on this slow manifold. The non-autonomous nature of the equation~\eqref{eq:slow} makes this a somewhat challenging construction. To demonstrate how these orbits are constructed using the slow/fast geometry of~\eqref{eq:slow}, it is helpful to first consider the simpler construction of traveling or stationary planar front solutions of~\eqref{eq:modifiedKlausmeier}, which manifest as heteroclinic orbits in an appropriate traveling equation with a similar geometry to that of~\eqref{eq:slow}.

It is shown in companion paper \cite{CDLOR} that the invasion fronts in (\ref{eq:modifiedKlausmeier}) -- that can be seen as spots or gaps with radius $r_I \to \infty$ -- are unstable with respect to a sideband/finger instability. Therefore, we next consider the stability of spots. Without going into the details of a fully rigorous analysis, we first consider the spectral stability problem for large spots. By a careful asymptotic analysis we recover the sideband instability mechanism and conclude that large spots are unstable and will typically form finger-like patterns -- like the planar invasion fronts. By similarity, we conclude that the same is true for large gaps. These observations are confirmed by direct simulations: in Fig. \ref{fig:spot_fingering} in \S\ref{sec:discussion} we show snapshots of a spot and a gap that both evolve towards labyrinthine patterns after undergoing such an instability. Moreover, we subsequently conclude by considering perturbations within a range of intermediate wave numbers that spots (and gaps) with $\mathcal{O}(1)$ radius $r_I$ must also be unstable: only sufficiently small spots and gaps may possibly be stable -- see again \S\ref{sec:discussion}  and especially Fig. \ref{fig:stable_radial_solutions} for a brief numerical study. These small gaps correspond to fairy circles \cite{getzin2016discovery} and also appear as stable vegetation patterns in numerical simulations of the dryland ecosystem model of  \cite{zelnik2015gradual} -- a model that can be seen as a slightly more extended version of (\ref{eq:modifiedKlausmeier}) \cite{CDLOR} (and that has been deduced from more extended models to study fairy circles).

The set-up of the paper is as follows. In~\S\ref{sec:fronts}, we consider the construction of traveling fronts in~\eqref{eq:modifiedKlausmeier}, while in~\S\ref{sec:existence}, we treat the radially symmetric case and provide the proofs of Theorems~\ref{thm:spot_existence}--\ref{thm:gap_existence}. The spectral stability of the radial spot and gap solutions is considered in~\S\ref{sec:stability} using formal asymptotic methods, and we include numerical simulations and a brief discussion in~\S\ref{sec:discussion}.

\section{Stationary and traveling planar fronts}\label{sec:fronts}
To motivate the approach for the existence analysis of radially symmetric  solutions of~\eqref{eq:modifiedKlausmeier}, we first consider the (simpler) case of constructing stationary and traveling front solutions. We pose the traveling wave ansatz $(u,v)(x,y,t) = (u,v)(\xi)$, where $\xi=x-\delta ct$ is a traveling wave coordinate; here $c=0$ corresponds to stationary front solutions, while $c\neq0$ corresponds to solutions which propagate with wave speed $\delta c$.

Using a geometric singular perturbation approach, we can construct bistable front solutions as perturbations from slow/fast heteroclinic orbits between the desert state $(U,V)=(a,0)$ and the vegetated state $(U,V)=(U_{2},V_{2})$ in the traveling wave ODE
\begin{align}
\begin{split}\label{eq:klaus_twode}
0&= u_{\xi\xi}+\delta cu_\xi+a-u-uv^2\\
0&= \delta^2v_{\xi\xi}+\delta cv_\xi -mv+uv^2(1-bv).
\end{split}
\end{align}
We note that in the case $c=0$, the system~\eqref{eq:klaus_twode} corresponds to~\eqref{eq:klaus_stationary} in the far field limit $r\to\infty$. Thus, it is natural to first consider the geometry of~\eqref{eq:klaus_twode}, as solutions with radial symmetry are constructed by matching a solution which is bounded near the core $r=0$, with a stationary solution which (approximately) satisfies~\eqref{eq:klaus_twode} in the far field limit $r\to \infty$.

\subsection{Slow/fast analysis}
We can write~\eqref{eq:klaus_twode} as a first order system
\begin{align}
\begin{split}\label{eq:1d_slow}
 u_\xi&=  p\\
  p_\xi&= -\delta cp-a+u+uv^2\\
\delta v_\xi&= q\\
\delta q_\xi&= -cq +mv-uv^2(1-bv),
\end{split}
\end{align}
which we refer to as the slow system, and upon rescaling $\xi = \delta \zeta$, we obtain the equivalent fast system
\begin{align}
\begin{split}\label{eq:1d_fast}
 u_\zeta&=  \delta p\\
  p_\zeta&= \delta\left(-\delta cp-a+u+uv^2\right)\\
v_\zeta&= q\\
q_\zeta&= -cq +mv-uv^2(1-bv),
\end{split}
\end{align}
Setting $\delta=0$ in~\eqref{eq:1d_fast} yields the layer problem
\begin{align}
\begin{split}\label{eq:1d_layer}
v_\zeta&= q\\
q_\zeta&= -cq +mv-uv^2(1-bv).
\end{split}
\end{align}
This system admits up to three equilibria depending on $u$, given by $M_0(u):=(0,0)$, and if $u>4bm$, $M_\pm(u):=(v_\pm(u),0)$, where
\begin{align}
v_\pm(u) := \frac{1\pm\sqrt{1-\frac{4bm}{u}}}{2b}.
\end{align}
When $u=4bm$, the two equilibria $M_+(4bm)=M_-(4bm)$ coincide. Computing the linearization
\begin{align}\label{eq:layer_lin}
J_\mathrm{layer} = \begin{pmatrix}  0 &1\\ m-2uv+3buv^2 & -c \end{pmatrix},
\end{align}
a short computation shows that for $m,b>0$ the fixed points $M_0$ and $M_+(u),u>4bm$ are always of saddle type, while $M_-(u)$ is focus or node for $c\neq0$, and a center for $c=0$. The equilibrium $M_+(4bm)=M_-(4bm)$ is nonhyperbolic. Taken together, the set of equilibria of the layer equation~\eqref{eq:1d_layer} corresponds to the critical manifold
\begin{align}
\mathcal{M}_0 = \left\{(u,p,v,q)\in \mathbb{R}^4:q=0, mv=uv^2(1-bv)\right\},
\end{align}
obtained by setting $\delta=0$ in~\eqref{eq:1d_slow}. We therefore decompose $\mathcal{M}_0$ into three branches $\mathcal{M}_0 = \mathcal{M}^0_0\cup \mathcal{M}^-_0 \cup \mathcal{F}\cup \mathcal{M}^+_0$, where 
\begin{align}\label{eq:M0branches}
\mathcal{M}^0_0 = \{v=q=0\}, \qquad \mathcal{M}^-_0 = \{q=0, v=v_-(u), u>4bm\}, \qquad \mathcal{M}^+_0 = \{q=0, v=v_+(u), u>4bm\},
\end{align}
and the latter two manifolds $\mathcal{M}^\pm_0$ meet along the nonhyperbolic fold curve $\mathcal{F}:=\{q=0, v =1/2b, u=4bm\}$. The manifolds $\mathcal{M}^0_0$ and $\mathcal{M}^+_0$ are normally hyperbolic and of saddle-type.

The reduced flow on each branch of the critical manifold is obtained by setting $\delta=0$ in~\eqref{eq:1d_slow}, and is given by
\begin{align}
\begin{split}\label{eq:1d_red}
 u_\xi&=  p\\
  p_\xi&= -a+u+uv_*(u)^2,
\end{split}
\end{align}
where $v_*(u)=0, v_\pm(u)$ on the branches $\mathcal{M}^0_0$ and $\mathcal{M}^\pm_0$, respectively.

\subsection{Singular orbits}~\label{sec:fronts_singular}
The desert equilibrium state $P_0=(a,0,0,0)$ lies on the branch $\mathcal{M}^0_0$, while the vegetated state $P_1 = (U_1,0,V_1,0)$ always lies on the branch $\mathcal{M}^-_0$. The state $P_2 = (U_2,0,V_2,0)$ can lie on either $\mathcal{M}^-_0$ or $\mathcal{M}^+_0$; the latter occurs provided $a>4mb+m/b$. In this case, a linear stability analysis shows that both $(U_0,V_0)$ and $(U_2,V_2)$ are temporally stable as homogeneous equilibria of the PDE~\eqref{eq:modifiedKlausmeier}~\cite{CDLOR}, and we can proceed to construct singular bistable fronts connecting $P_0$ and $P_2$ (or vice versa); see Figure~\ref{fig:critical_manifold}.

\begin{figure}
\centering
\includegraphics[width=0.55\linewidth]{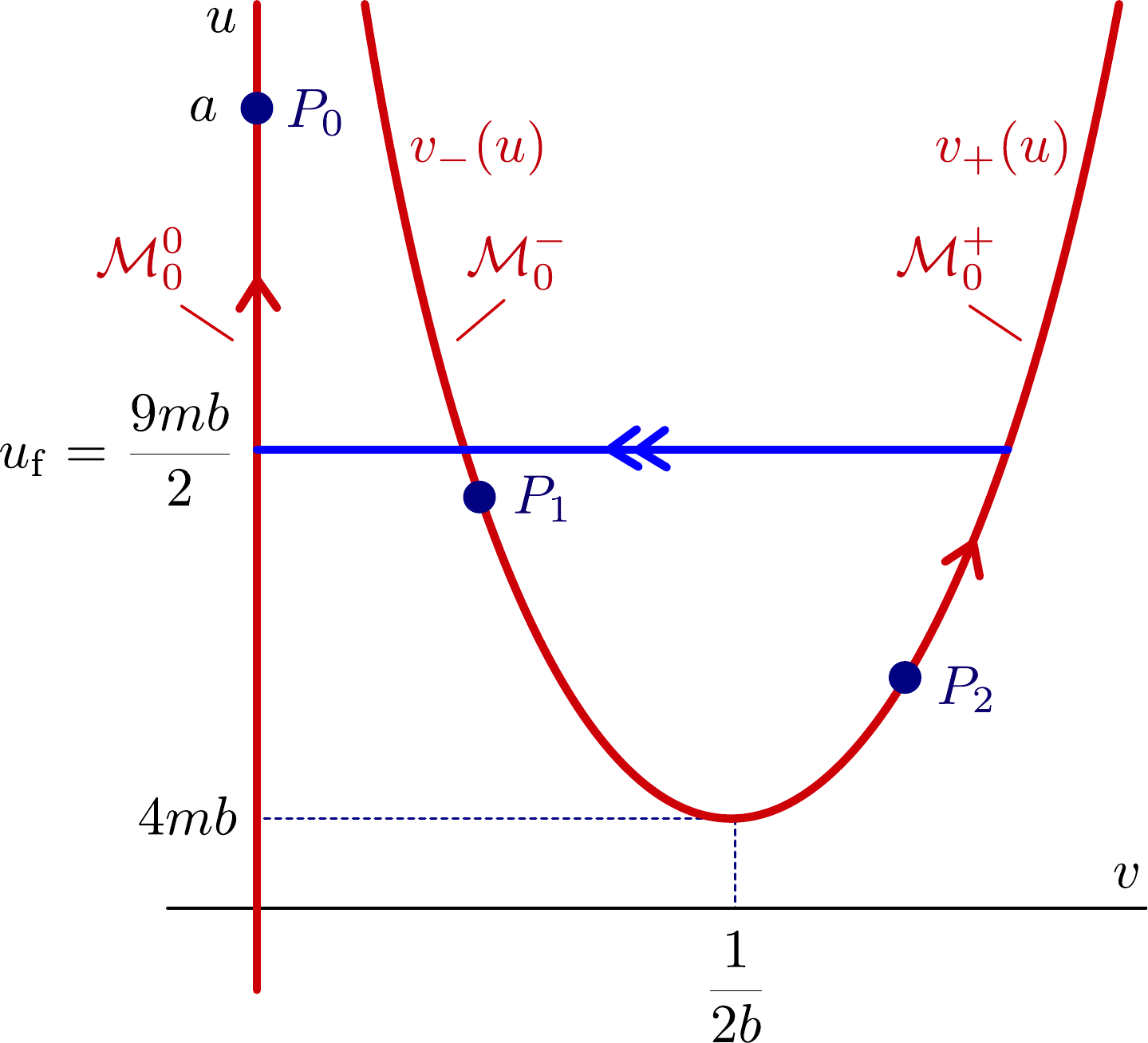}
\caption{A singular heteroclinic orbit between $P_2$ and $P_0$ is formed by concatenating slow orbits on the critical manifolds $\mathcal{M}^0_0$ and $\mathcal{M}_0^+$ with a fast orbit $\phi_\mathrm{vd}$ of the layer problem~\eqref{eq:1d_layer}.}
\label{fig:critical_manifold}
\end{figure}

To do this we form a singular slow/fast/slow heteroclinic orbit between $P_0, P_2$, by concatenating a slow orbit on $\mathcal{M}^0_0$ with another slow orbit on $\mathcal{M}^+_0$ via a fast heteroclinic orbit in the layer problem. Within the layer problem~\eqref{eq:1d_layer}, for any value of $(u,p)$ satisfying $u>4mb$, by adjusting the speed $c$ appropriately it is possible to construct fast heteroclinic orbits $\phi_\mathrm{dv}(\zeta)=(v_\mathrm{dv},q_\mathrm{dv})(\zeta)$ from $\mathcal{M}^0_0$ to $\mathcal{M}^+_0$ and $\phi_\mathrm{vd}(\zeta)=(v_\mathrm{vd},q_\mathrm{vd})(\zeta)$ from $\mathcal{M}^+_0$ to $\mathcal{M}^0_0$, corresponding to fast desert-to-vegetation and vegetation-to-desert fronts, respectively. These orbits can be computed explicitly as follows. Considering~\eqref{eq:1d_layer}, we search for solutions satisfying the ansatz $q=Cv(v-v_+(u))$ for some value of $C\neq 0$.  {Substituting into~\eqref{eq:1d_layer}, we obtain the algebraic equation
\begin{align}
C^2(v-v_+(u))+C^2v&= -cC +ub(v-v_-(u)),
\end{align}
which can be solved to find $C=\pm \sqrt{bu/2}$ and wave speeds
\begin{align*}
    c=c_\mathrm{vd}(u)&=\sqrt{\frac{bu}{2}}\left(v_+(u)-2v_-(u)\right)\\
    c=c_\mathrm{dv}(u)&=-\sqrt{\frac{bu}{2}}\left(v_+(u)-2v_-(u)\right).
\end{align*}}
We then obtain the explicit profiles (up to translation)
\begin{align}
v_\mathrm{vd}(\zeta)&:=\frac{v_+(u)}{2}\left(1-\tanh\left(\frac{v_+(u)\sqrt{bu}}{2\sqrt{2}} \zeta \right)\right),\qquad q_\mathrm{vd}(\zeta) = v_\mathrm{vd}'(\zeta)\\
v_\mathrm{dv}(\zeta)&:=\frac{v_+(u)}{2}\left(1+\tanh\left(\frac{v_+(u)\sqrt{bu}}{2\sqrt{2}} \zeta \right)\right), \qquad q_\mathrm{dv}(\zeta) = v_\mathrm{dv}'(\zeta).
\end{align}

Examining the reduced flow on $\mathcal{M}^0_0$
\begin{align}
\begin{split}\label{eq:1d_red0}
 u_\xi&=  p\\
  p_\xi&= -a+u,
\end{split}
\end{align}
we see that within $\mathcal{M}^0_0$, the equilibrium $(a,0)$, corresponding to $P_0$, is a saddle-type equilibrium with (un)stable manifolds $W^{\mathrm{u}/\mathrm{s}}(a,0)$ given by the lines $p = \pm(u-a)$. On $\mathcal{M}^+_0$, the equilibrium $(U_2,0)$, corresponding to $P_2$, is also saddle-type equilibrium, with (un)stable manifolds $W^{\mathrm{u}/\mathrm{s}}(U_2,0)$ given by the level set $E(u,p)=0$ of the conserved quantity
\begin{align}
    E(u,p) &= -\frac{1}{2}p^2 + \int_{U_2}^{u} \tilde{u}-a+\tilde{u}(v_+(\tilde{u}))^2 d\tilde{u},
\end{align}
noting that $E(U_2,0)=0$. 

To construct a singular heteroclinic orbit from say $P_2$ to $P_0$, we follow the slow unstable manifold $W^{\mathrm{u}}(U_2,0)$ of $P_2$, then a fast jump $\phi_\mathrm{vd}$, then the slow stable manifold $W^{\mathrm{s}}(a,0)$, given by the line $p=-(u-a)$. Since the slow variables $(u,p)$ are constant across the fast jump, this is only possible if there is an intersection of $W^{\mathrm{u}}(U_2,0)$ and $W^{\mathrm{s}}(a,0)$ when $W^{\mathrm{u}}(U_2,0)$ is projected onto $\mathcal{M}^0_0$, which occurs if there exists $u=u_*$ such that $E(u_*,a-u_*)=0$, or equivalently
\begin{align}\label{eq:energy_formula_front}
    \frac{1}{2}(a-u_*)^2 &= \int_{U_2}^{u_*} \tilde{u}-a+\tilde{u}(v_+(\tilde{u}))^2 d\tilde{u},
\end{align}
For an open region in $(a,b,m)$ parameter space, there is a critical value $U_2<u_*<a$, depending on $a,b,m$, which satisfies this criterion, and therefore the speed $c_*$ is determined so that the fast layer jump $\phi_\mathrm{vd}$ exists for $(u,p)=(u_*,a-u_*)$. We denote by $(u_{+,\infty},p_{+,\infty})(\xi)$ the solution of the reduced flow on $\mathcal{M}^+_0$ corresponding to the unstable manifold $W^{\mathrm{u}}(U_2,0)$, which satisfies $(u_{+,\infty},p_{+,\infty})(0)=(u_*,a-u_*)$, and we denote by $(u_{0,\infty},p_{0,\infty})(\xi)$ the solution of the reduced flow on $\mathcal{M}^0_0$ corresponding to the stable manifold $W^{\mathrm{s}}(a,0)$, which satisfies $(u_{0,\infty},p_{0,\infty})(0)=(u_*,a-u_*)$. 

\begin{figure}
\centering
\includegraphics[width=0.5\linewidth]{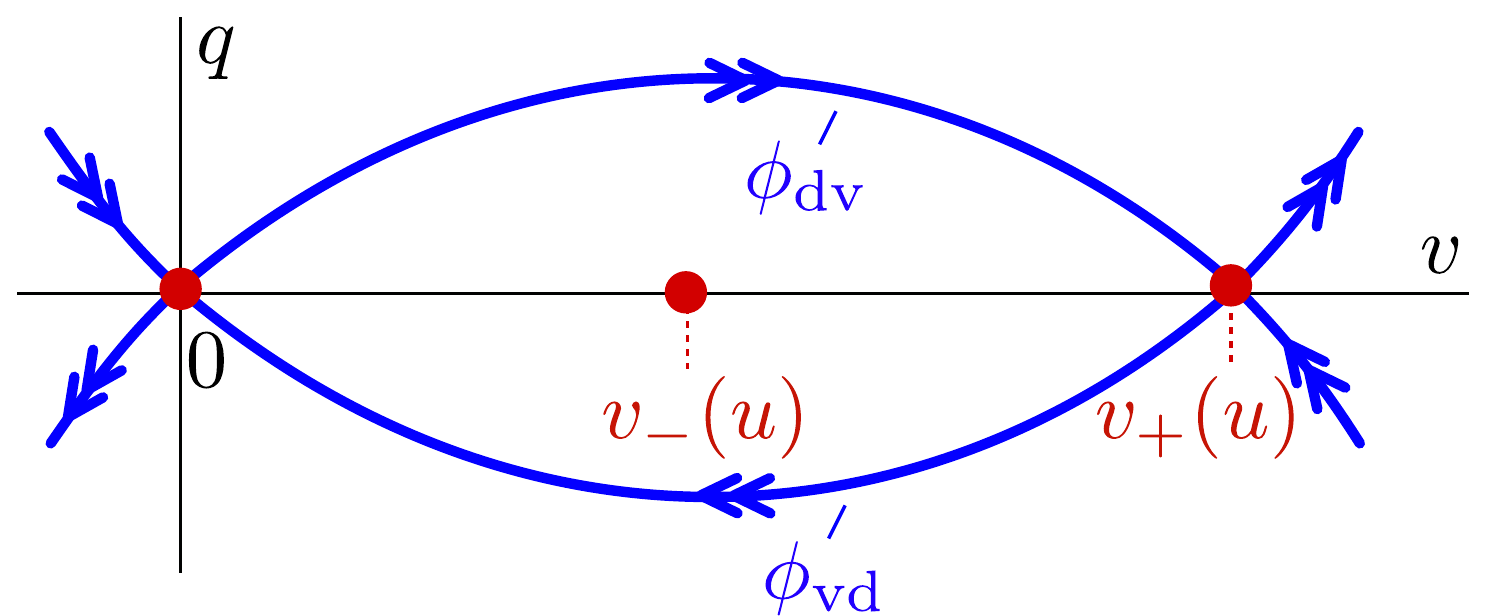}
\caption{The stationary fronts $\phi_\mathrm{vd}$ and $\phi_\mathrm{dv}$ of~\eqref{eq:1d_layer}.}
\label{fig:singular_front}
\end{figure}

Thus the singular front solution follows the solution $(u_{+,\infty},p_{+,\infty})(\xi)$ of the reduced flow on $\mathcal{M}^+_0$, followed by the fast front $\phi_\mathrm{vd}(\zeta)$, at finally the solution $(u_{0,\infty},p_{0,\infty})(\xi)$ of the reduced flow on $\mathcal{M}^0_0$. For $0<\delta\ll1$, these singular fronts can be shown to perturb to front solutions of~\eqref{eq:klaus_twode} using standard methods of geometric singular perturbation theory, using the wave speed $c$ as a free bifurcation parameter. The construction of fronts from $P_0$ to $P_2$ follows similarly. 

\begin{figure}
\centering
\includegraphics[width=0.6\linewidth]{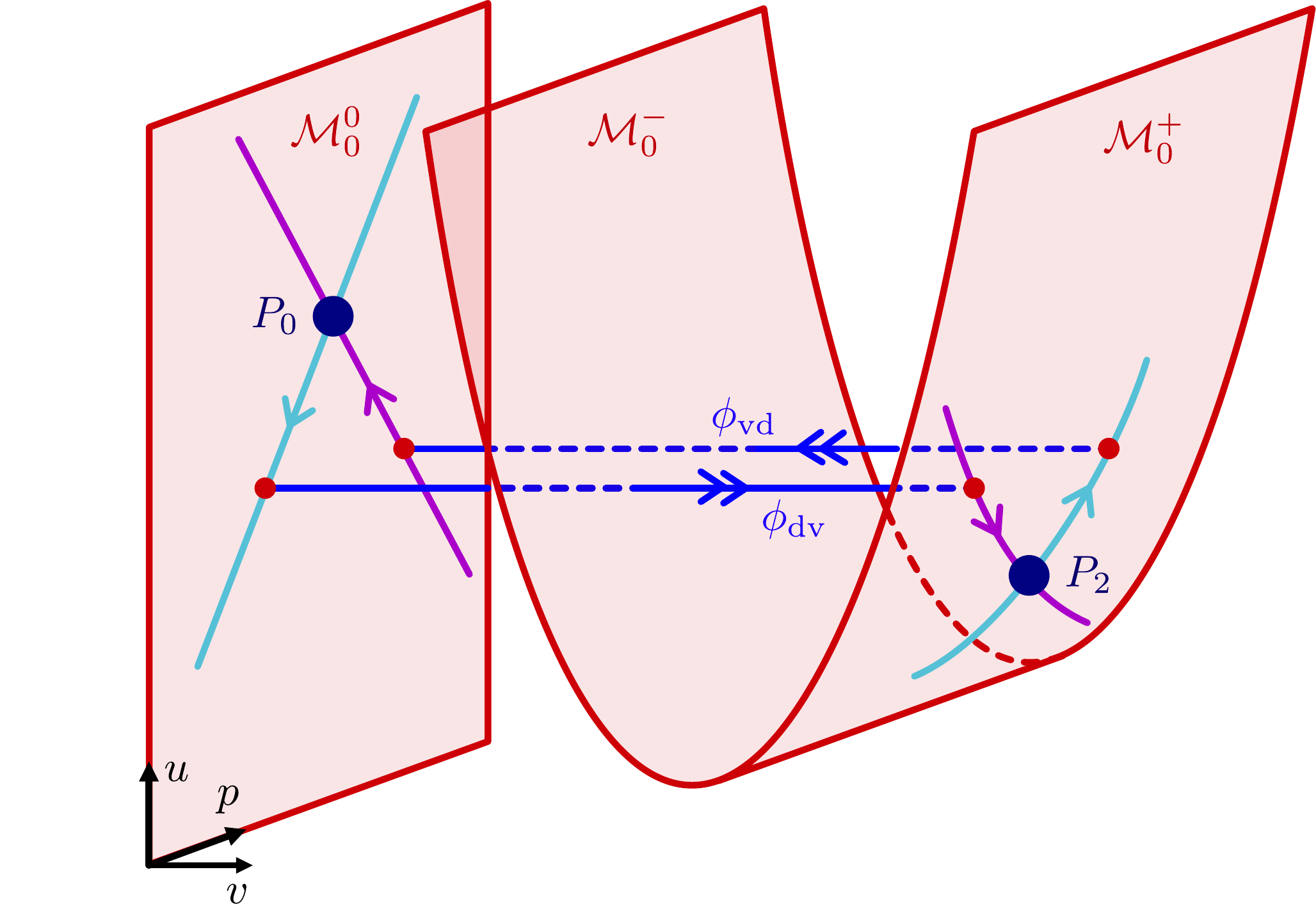}
\caption{Shown are singular orbits representing stationary fronts of~\eqref{eq:klaus_twode} obtained by concatenating slow orbits of~\eqref{eq:1d_red} on the critical manifolds $\mathcal{M}^0_0$ and $\mathcal{M}_0^+$ with fast orbits $\phi_\mathrm{vd}, \phi_\mathrm{dv}$ of the layer problem~\eqref{eq:1d_layer} for $c=0, u=u_\mathrm{f}$.}
\label{fig:singular_front_3D}
\end{figure}

\subsection{Stationary fronts}\label{sec:fronts_stationary}
Stationary fronts correspond to the case $c=0$. We describe the geometry of the singular orbit(s) in this case, as it will be useful in the forthcoming construction of spot and gap solutions. Proceeding as in~\S\ref{sec:fronts_singular}, we find that~\eqref{eq:1d_layer} admits a pair of heteroclinic orbits $\phi_\mathrm{dv}(\zeta)=(v_\mathrm{dv}, q_\mathrm{dv})(\zeta)$ and $\phi_\mathrm{vd}(\zeta)=(v_\mathrm{vd}, q_\mathrm{vd})(\zeta)$ when
\begin{align}
u=u_\mathrm{f}:=\frac{9bm}{2},
\end{align}
where
\begin{align}\begin{split}
v_\mathrm{dv}(\zeta)&:=\frac{1}{3b}\left(1+\tanh\left(\frac{\sqrt{m}}{2} \zeta \right)\right), \qquad q_\mathrm{dv}(\zeta) = v_\mathrm{dv}'(\zeta)\\
v_\mathrm{vd}(\zeta)&:=\frac{1}{3b}\left(1-\tanh\left(\frac{\sqrt{m}}{2} \zeta \right)\right), \qquad q_\mathrm{vd}(\zeta) = v_\mathrm{vd}'(\zeta).
\end{split}
\end{align}
where again $\phi_\mathrm{dv}(\zeta)$ represents the desert-to-vegetation state transition front which jumps from $\mathcal{M}^0_0$ and $\mathcal{M}^+_0$, while $\phi_\mathrm{vd}(\zeta)$ represents the vegetation-to-desert state transition front which jumps from $\mathcal{M}^+_0$ and $\mathcal{M}^0_0$; see Figures~\ref{fig:singular_front}--\ref{fig:singular_front_3D}. Note that we assume $a>u_\mathrm{f}$, so that the equilibrium $P_0$ lies `above' the critical fronts $\phi_\mathrm{dv}, \phi_\mathrm{vd}$. A lengthy but straightforward computation shows that $U_2<u_\mathrm{f}$ provided $\frac{a}{m}<\frac{9b}{2}+\frac{2}{b}$, so that the equilibrium $P_2$ lies `below' the fronts $\phi_\mathrm{dv}, \phi_\mathrm{vd}$.

While this pair of heteroclinic orbits exists for any $a,b,m>0$, in order to construct a singular slow/fast heteroclinic orbit, we still require an intersection of $W^{\mathrm{u}}(U_2,0)$ and $W^{\mathrm{s}}(a,0)$ in the reduced flow when $W^{\mathrm{u}}(U_2,0)$ is projected onto $\mathcal{M}^0_0$. Since the jump height $u=u_\mathrm{f}$ is fixed by the condition $c=0$, this only occurs if $E(u_\mathrm{f},a-u_\mathrm{f})=0$, or equivalently
\begin{align}\label{eq:energy_formula}
    \frac{1}{2}(a-u_\mathrm{f})^2 &= \int_{U_2}^{u_\mathrm{f}} \tilde{u}-a+\tilde{u}(v_+(\tilde{u}))^2 d\tilde{u},
\end{align}
which gives an implicit condition on the parameters $a,b,m$ for which a singular stationary front exists. We will show in~\S\ref{sec:existence} that this condition describes the boundary (in parameter space) which separates the existence region of radial spots versus gaps, which will be constructed as slow/fast fronts in the non-autonomous system~\eqref{eq:slow}. The radius of the spot/gap will serve as a free parameter which allows for the construction of a solution for parameters which satisfy~\eqref{eq:spot_condition} in the case of spots, or~\eqref{eq:gap_condition} in the case of gaps.

\subsection{Sideband (in)stability of planar fronts}\label{sec:fronts_stability}
Given a front solution constructed as in~\S\ref{sec:fronts_singular}, we briefly examine the stability of the front as a planar interface, which will help motivate our formal stability results in the case of radial spot and gap solutions in~\S\ref{sec:stability}. We assume that a heteroclinic vegetation-to-desert front solution $(u_\mathrm{h},v_\mathrm{h})(\xi;\delta)$ exists which connects the state $P_2$ to $P_0$ (the case of a desert-to-vegetation front is similar), and has been constructed as a perturbation from one of the singular slow/fast fronts in~\S\ref{sec:fronts_singular}, with speed $c=c_\mathrm{h}(\delta) = c_\mathrm{vd}(u_*)+\mathcal{O}(\delta)$.

We linearize~\eqref{eq:modifiedKlausmeier} about this front solution in a comoving frame using an ansatz  $(U,V) = (u_\mathrm{h},v_\mathrm{h})(\xi;\delta)+e^{\lambda t+ i\ell y}(u,v)(\xi)$ for $\ell\in\mathbb{R}$, which results in the eigenvalue problem 
\begin{align}
\begin{split}\label{eq:front_stability_problem}
\lambda u&= u_{\xi\xi}+\delta c_\mathrm{h}u_\xi-\ell^2u-\left(1+v_\mathrm{h}(\xi)^2\right)u-2u_\mathrm{h}(\xi)v_\mathrm{h}(\xi)v\\
\lambda v&= \delta^2v_{\xi\xi}+\delta c_\mathrm{h}v_\xi-\delta^2\ell^2v -mv+v_\mathrm{h}(\xi)^2\left(1-bv_\mathrm{h}(\xi)\right)u+u_\mathrm{h}(\xi)\left(2v_\mathrm{h}(\xi)-3bv_\mathrm{h}(\xi)^2\right)v.
\end{split}
\end{align}
Due to translation invariance, this eigenvalue problem has a solution when $\lambda=\ell=0$, with eigenfunction given by the derivative $(u_\mathrm{h}',v_\mathrm{h}')(\xi;\delta)$. For the purposes of this discussion, we focus only on this critical, marginal eigenvalue, and we assume that all other spectrum of the front for $\ell=0$ (that is, the spectrum corresponding to 1D longitudinal perturbations in the direction of propagation) is bounded away from the imaginary axis in the left half plane. 

Focusing on this critical eigenvalue, we consider its continuation for small $|\ell|$: as the eigenvalue problem only depends on $\ell$ through terms of $\mathcal{O}(\ell^2)$, we anticipate that this critical translation eigenvalue expands as
\begin{align}\label{eq:lambda_crit_front}
    \lambda_\mathrm{c}(\ell) =\lambda_{\mathrm{c},2}\ell^2+\mathcal{O}(\ell^4).
\end{align}
This eigenvalue describes the stability of the front to long wavelength perturbations transverse to the front, and the stability is thus determined by the sign of the coefficient $\lambda_{\mathrm{c},2}$. In a companion paper~\cite{CDLOR}, we have developed a procedure to compute this coefficient in a general class of two-component singularly perturbed reaction diffusion systems, which includes~\eqref{eq:modifiedKlausmeier} as an example, and we briefly describe the results here. 

We rewrite the stability problem~\eqref{eq:front_stability_problem} in the form
\begin{align}\label{eq:front_eigenvalue}
    \mathbb{L}\begin{pmatrix}u\\v \end{pmatrix}& = \lambda \begin{pmatrix}u\\v \end{pmatrix}+ \ell^2\begin{pmatrix}u\\\delta^2v \end{pmatrix}
\end{align}
where
\begin{align}
    \mathbb{L}& := \begin{pmatrix}\partial_{\xi\xi}+\delta c_\mathrm{h}\partial_\xi-\left(1+v_\mathrm{h}(\xi)^2\right) & -2u_\mathrm{h}(\xi)v_\mathrm{h}(\xi)\\v_\mathrm{h}(\xi)^2\left(1-bv_\mathrm{h}(\xi)\right) & \delta^2\partial_{\xi\xi}+\delta c_\mathrm{h}\partial_\xi+ u_\mathrm{h}(\xi)\left(2v_\mathrm{h}(\xi)-3bv_\mathrm{h}(\xi)^2\right) \end{pmatrix}, 
\end{align}
We now expand
\begin{align*}
\begin{pmatrix}u\\v\end{pmatrix}(\xi;\ell)  &= \begin{pmatrix}u_\mathrm{h}'\\v_\mathrm{h}'\end{pmatrix} (\xi) +\ell^2\begin{pmatrix}u_{2,\mathrm{c}}\\v_{2,\mathrm{c}}\end{pmatrix}(\xi)+\mathcal{O}\left(\ell^4\right).
\end{align*}
Substituting into~\eqref{eq:front_eigenvalue}, at leading order we have the eigenvalue problem
\begin{align}\label{eq:front_eigenvalue_2}
    \mathbb{L}\begin{pmatrix}u_{2,\mathrm{c}}\\v_{2,\mathrm{c}} \end{pmatrix}& = \lambda_{2,\mathrm{c}} \begin{pmatrix}u_\mathrm{h}'\\v_\mathrm{h}' \end{pmatrix}+\begin{pmatrix}u_\mathrm{h}'\\\delta^2 v_\mathrm{h}' \end{pmatrix}.
\end{align}
This leads to the Fredholm solvability condition
\begin{align}\label{eq:front_fredholm}
    \left\langle \lambda_{2,\mathrm{c}} \begin{pmatrix}u_\mathrm{h}'\\v_\mathrm{h}' \end{pmatrix}+ \begin{pmatrix}u_\mathrm{h}'\\\delta^2 v_\mathrm{h}' \end{pmatrix}, \begin{pmatrix}u^A_\mathrm{h}\\v^A_\mathrm{h} \end{pmatrix}\right \rangle &=0,
\end{align}
where $<U,V> = \int_{-\infty}^\infty U(\xi)V(\xi)\mathrm{d}\xi $, and  $(u^A_\mathrm{h}, v^A_\mathrm{h})(\xi;\delta)$ denotes the bounded solution to the adjoint equation 
\begin{align}
\mathbb{L}^A\begin{pmatrix}u\\v \end{pmatrix} = 0,
\end{align}
where
\begin{align}
    \mathbb{L}^A& := \begin{pmatrix}\partial_{\xi\xi}-\delta c_\mathrm{h}\partial_\xi-\left(1+v_\mathrm{h}(\xi)^2\right) & v_\mathrm{h}(\xi)^2\left(1-bv_\mathrm{h}(\xi)\right) \\-2u_\mathrm{h}(\xi)v_\mathrm{h}(\xi) & \delta^2\partial_{\xi\xi}-\delta c_\mathrm{h}\partial_\xi+ u_\mathrm{h}(\xi)\left(2v_\mathrm{h}(\xi)-3bv_\mathrm{h}(\xi)^2\right) \end{pmatrix}.
\end{align}
From this we obtain an expression for $\lambda_{2,\mathrm{c}}$
\begin{align}\label{eq:front_lambda2c}
  \lambda_{2,\mathrm{c}} &= -\frac{\left\langle \begin{pmatrix}u_\mathrm{h}'\\\delta^2 v_\mathrm{h}' \end{pmatrix}, \begin{pmatrix}u^A_\mathrm{h}\\v^A_\mathrm{h} \end{pmatrix}\right \rangle}{\left\langle  \begin{pmatrix}u_\mathrm{h}'\\v_\mathrm{h}' \end{pmatrix}, \begin{pmatrix}u^A_\mathrm{h}\\v^A_\mathrm{h} \end{pmatrix}\right \rangle}=- \frac{\int_{-\infty}^\infty(u_\mathrm{h}'u^A_\mathrm{h}+\delta^2 v_\mathrm{h}'v^A_\mathrm{h})\mathrm{d}\xi}{\int_{-\infty}^\infty (u_\mathrm{h}'u^A_\mathrm{h}+ v_\mathrm{h}'v^A_\mathrm{h})\mathrm{d}\xi}.
\end{align}

In~\cite{CDLOR}, it is shown that to leading order $\lambda_{2,\mathrm{c}}$ is given by
\begin{align}\label{eq:front_lambda2c_est}
    \lambda_{2,\mathrm{c}}&=\delta \frac{\int_{-\infty}^\infty v_\mathrm{vd}(\zeta)^2(1-v_\mathrm{vd}(\zeta))e^{c_\mathrm{vd} \zeta}v_\mathrm{vd}'(\zeta)\mathrm{d}\zeta}{u_*v_+(u_*)^2\int_{-\infty}^\infty e^{c_\mathrm{vd} \zeta}v_\mathrm{vd}'(\zeta)^2\mathrm{d}\zeta}\left(\int_{-\infty}^0u_{+,\infty}'(\xi)^2\mathrm{d}\xi+\int_0^\infty u_{0,\infty}'(\xi)^2\mathrm{d}\xi\right)+o(\delta)\\
    &>0, \nonumber
\end{align}
so that the traveling planar fronts are always unstable to long wavelength transverse perturbations. This has implications for the (in)stability of spot/gap solutions, as we will show in~\S\ref{sec:stability} using formal asymptotic methods that the spectrum for radial spots/gaps of sufficiently large radius is approximated by that of a nearby stationary front solution.

\section{Existence of radially symmetric spots and gaps}\label{sec:existence}
With the construction of the front solutions in~\S\ref{sec:fronts} in mind, we now focus on the construction of a vegetation spot solution, consisting of a single vegetation patch localized near $r=0$, with a single interface at some radius $r=r_I$ (to be specified), at which the profile transitions from the vegetated state in the core to the desert state in the far-field. The case of gaps is similar, and we will briefly outline the differences in~\S\ref{sec:existence_proofs}.

Throughout the analysis we treat $0<\delta\ll1$ as a singular perturbation parameter. At times, it will also be convenient to consider~\eqref{eq:slow}, which we refer to as the `slow' system, with respect to the rescaled radial coordinate $s=r/\delta$, which results in the system
\begin{align}
\begin{split}\label{eq:fast}
 u_s&=  \delta p\\
  p_s&= -\frac{p}{s}-\delta(a-u-uv^2)\\
v_s&= q\\
q_s&= -\frac{q}{s} +mv-uv^2(1-bv),
\end{split}
\end{align}
We refer to~\eqref{eq:fast} as the `fast' system.

The spot solution will be constructed as a perturbation from the singular limit structure associated with~\eqref{eq:fast}, and consists of three pieces: The core vegetated and far-field desert states are given as slow orbits which lie near equilibria on saddle-type slow manifolds $\mathcal{M}^0_\delta,\mathcal{M}^+_\delta$ within~\eqref{eq:fast} (to be defined below), while the interface between these states is given by a fast layer orbit between these slow manifolds which is inserted at a particular radius $r=r_I$. The construction has very similar geometry as in the construction of traveling fronts in~\S\ref{sec:fronts}, though with some complications due to the nonautonmous nature of the equation and the singularity at $r=0$. Additionally, since the spot interface is stationary, the speed is not available as a free parameter; however, the jump value $r=r_I$ can be thought of as a free parameter which is chosen in such a way in order to ensure that the stable manifold of the far-field desert equilibrium and the unstable manifold associated with the core vegetated states precisely intersect transversely across the fast jump.

\subsection{Far field region and fast transition}\label{sec:farfield}
In the far-field, we consider $r\in[ \bar{r},\infty)$ for arbitrary $\bar{r}>0$ fixed independently of $\delta>0$. Here we define a new variable $k:=1/r$ and similarly consider the region $k\in [0,\bar{k}]$ where $\bar{k}:=1/\bar{r}$ for the corresponding (autonomous) system
\begin{align}
\begin{split}\label{eq:slow_far}
 u_r&=  p\\
  p_r&= -kp-a+u+uv^2\\
\delta v_r&= q\\
\delta q_r&= -\delta k q +mv-uv^2(1-bv)\\
k_r &= -k^2,
\end{split}
\end{align}
which is a slow-fast system with two fast variables and three slow variables. 

\subsubsection{Slow manifolds away from the core}
Setting $\delta=0$, we see that~\eqref{eq:slow_far} admits a three dimensional critical manifold defined by
 {
\begin{align}\label{eq:M0}
\mathcal{M}_0 = \left\{(u,p,v,q,k)\in \mathbb{R}^4\times[0,\bar{k}]:q=0, mv=uv^2(1-bv)\right\}.
\end{align}}

Considering~\eqref{eq:slow_far} on the fast scale $s=r/\delta$, we have the equivalent system
\begin{align}
\begin{split}\label{eq:fast_far}
 u_s&=  \delta p\\
  p_s&= \delta(-kp-a+u+uv^2)\\
v_s&= q\\
q_s&= -\delta k q +mv-uv^2(1-bv)\\
k_s &= -\delta k^2,
\end{split}
\end{align}
where $k=(\delta s)^{-1}$. Setting $\delta=0$ in this system yields the layer problem
\begin{align}
\begin{split}\label{eq:layer}
v_s&= q\\
q_s&= mv-uv^2(1-bv),
\end{split}
\end{align}
in which the slow variables $(u,p,k)$ act as parameters. As in this case of the layer problem~\eqref{eq:1d_layer} associated with the traveling fronts in~\S\ref{sec:fronts}, this system admits up to three equilibria depending on $u$, given by $M_0(u):=(0,0)$, and if $u>4bm$, $M_\pm(u):=(v_\pm(u),0)$, where
\begin{align}
v_\pm(u) := \frac{1\pm\sqrt{1-\frac{4bm}{u}}}{2b}.
\end{align}
When $u=4bm$, the two equilibria $M_+(4bm)=M_-(4bm)$ coincide. Computing the linearization
\begin{align}
J_\mathrm{layer} = \begin{pmatrix}  0 &1\\ m-2uv+3buv^2 & 0 \end{pmatrix},
\end{align}
we find that this corresponds to~\eqref{eq:layer_lin} in the stationary case $c=0$, so that for $m,b>0$, the fixed points $M_0$ and $M_+(u),u>4bm$ are always of saddle type, while $M_-(u)$ is a center for $u>4bm$. The equilibrium $M_+(4bm)=M_-(4bm)$ is nonhyperbolic with a double-zero eigenvalue. Taken together, the set of equilibria of the layer equation~\eqref{eq:layer} corresponds to the critical manifold $\mathcal{M}_0=\mathcal{M}^0_0\cup \mathcal{M}^-_0 \cup \mathcal{F}\cup \mathcal{M}^+_0$ as in~\eqref{eq:M0branches}, where $\mathcal{M}^\pm_0$ meet along the nonhyperbolic fold curve $\mathcal{F}:=\{q=0, v =1/2b, u=4bm\}$. The manifolds $\mathcal{M}^0_0$ and $\mathcal{M}^+_0$ are normally hyperbolic, while $\mathcal{M}^-_0$ is not. We note that compared with the analysis in~\S\ref{sec:fronts}, these manifolds are actually subsets of a $5$-dimensional ambient space due to the additional slow variable $k$, but we slightly abuse notation and continue to refer to these as $\mathcal{M}^*_0,$ for $*=0,\pm$, as they are defined by the same algebraic conditions.

Further, any compact portion of $\mathcal{M}^0_0$ or $\mathcal{M}^+_0$ admits local (un)stable manifolds $\mathcal{W}^{\mathrm{s},\mathrm{u}}(\mathcal{M}^*_0),$ for $*=0,+$, comprised of the union of the local (un)stable manifolds $\mathcal{W}^{\mathrm{s},\mathrm{u}}(M_*(u))$ of the equilibria $M_*(u),$ for $ *=0,+$.

To obtain the reduced dynamics on the critical manifolds, we consider~\eqref{eq:slow_far} for $\delta=0$, given by 
\begin{align}
\begin{split}\label{eq:reduced}
 u_r&=  p\\
  p_r&= -kp-a+u+uv^2\\
k_r &= -k^2,
\end{split}
\end{align}
where we substitute $v=0$ (in the case of $\mathcal{M}^0_0$) or $v=v_\pm(u)$ (in the case of $\mathcal{M}^\pm_0$) into the $p$-equation.

By standard results of geometric singular perturbation theory, (restricting to the region $u>4bm$ in the case of $\mathcal{M}^+_0$) for all sufficiently small $\delta>0$ any compact portions of the normally hyperbolic invariant manifolds $\mathcal{M}^0_0$ and $\mathcal{M}^+_0$ perturb to three-dimensional locally invariant manifolds $\mathcal{M}^0_\delta$ and $\mathcal{M}^+_\delta$, which are $C^1$-$\mathcal{O}(\delta)$-close to their singular counterparts. The slow flow on $\mathcal{M}^0_\delta$ and $\mathcal{M}^+_\delta$ is an $\mathcal{O}(\delta)$-perturbation of the reduced flow~\eqref{eq:reduced}. Similarly, the local (un)stable manifolds $\mathcal{W}^{\mathrm{s},\mathrm{u}}(\mathcal{M}^*_0),*=0,+$ perturb to four-dimensional locally invariant manifolds $\mathcal{W}^{\mathrm{s},\mathrm{u}}(\mathcal{M}^*_\delta),$ for $*=0,+$ which are again $\mathcal{O}(\delta)$-close to their singular counterparts, and consist of the fast (un)stable fibers associated with orbits which lie on the slow manifolds $\mathcal{M}^0_\delta$ and $\mathcal{M}^+_\delta$.

\subsubsection{Fast transition layers}\label{sec:fast_jumps}
We aim to construct fast transition layers consisting of fast jumps between the critical manifolds $\mathcal{M}^0_0$ and $\mathcal{M}^+_0$. We return to the fast system~\eqref{eq:fast_far} and the associated layer problem~\eqref{eq:layer} for values of $k\in[0,\bar{k}]$. We note that the layer problem~\eqref{eq:layer} corresponds to~\eqref{eq:1d_layer} in the stationary case $c=0$. We recall from~\S\ref{sec:fronts} that for values of $u>4mb$, this problem has three equilibria, $M_0(u)$ and $M_\pm(u)$, and at the critical value $u=u_\mathrm{f}:=\frac{9bm}{2}$, the layer problem admits a heteroclinic loop between $M_0(u_\mathrm{f})$ and $M_+(u_\mathrm{f})$, while a homoclinic orbit to either $M_0(u)$ or $M_+(u)$ exists for values of $u>u_\mathrm{f}$ or $u<u_\mathrm{f}$, respectively. As in~\S\ref{sec:fronts_stationary}, we assume that $a>u_\mathrm{f}$ or equivalently, $\frac{a}{m}>\frac{9b}{2}$, so that the equilibrium $P_0$ lies above the heteroclinic loop in the layer problem, and we further assume $4b+\frac{1}{b}<\frac{a}{m}<\frac{9b}{2}+\frac{2}{b}$, so that the equilibrium $P_2$ lies on $\mathcal{M}^+_0$ and below the heteroclinic loop.

The two heteroclinic orbits comprising the heteroclinic loop provide an opportunity for orbits to jump between the manifolds $\mathcal{M}^0_0$ and $\mathcal{M}^+_0$, and effectively transition between the desert/vegetated states. We recall from~\S\ref{sec:fronts_stationary} that these two orbits $\phi_\mathrm{dv}(s)=(v_\mathrm{dv}, q_\mathrm{dv})(s)$ and $\phi_\mathrm{vd}(s)=(v_\mathrm{vd}, q_\mathrm{vd})(s)$ are given explicitly as
\begin{align}
v_\mathrm{dv}(s)&:=\frac{1}{3b}\left(1+\tanh\left(\frac{\sqrt{m}}{2} s \right)\right)\\
v_\mathrm{vd}(s)&:=\frac{1}{3b}\left(1-\tanh\left(\frac{\sqrt{m}}{2} s \right)\right),
\end{align}
and
\begin{align}
q_\mathrm{dv}(s) &= \frac{3b\sqrt{m}}{2}v_\mathrm{dv}(s)\left(v_\mathrm{dv}(s)-\frac{2b}{3}\right)\\
q_\mathrm{vd}(s) &= -\frac{3b\sqrt{m}}{2}v_\mathrm{vd}(s)\left(v_\mathrm{vd}(s)-\frac{2b}{3}\right),
\end{align}
see Figure~\ref{fig:singular_front}. As in~\S\ref{sec:fronts_stationary}, $\phi_\mathrm{dv}(s)$ represents the desert-to-vegetation state transition front which jumps from $\mathcal{M}^0_0$ and $\mathcal{M}^+_0$, while $\phi_\mathrm{vd}(s)$ represents the vegetation-to-desert state transition front which jumps from $\mathcal{M}^+_0$ and $\mathcal{M}^0_0$. These singular orbits serve as candidate interfaces between the desert and vegetated states. 

Taking the union over values of the slow variables $(p,r)$, or equivalently values of $(p,k)$, the orbits $\phi_\mathrm{dv}$ form a two-parameter family of orbits lying in the intersection of the four-dimensional manifolds $\mathcal{W}^{\mathrm{u}}(\mathcal{M}^0_0)$ and $\mathcal{W}^{\mathrm{s}}(\mathcal{M}^+_0)$, and likewise the orbits $\phi_\mathrm{vd}$ form a two-parameter family of intersections between $\mathcal{W}^{\mathrm{u}}(\mathcal{M}^+_0)$ and $\mathcal{W}^{\mathrm{s}}(\mathcal{M}^0_0)$. In order to determine that these intersections are non-degenerate and persist in a suitable sense for $\delta>0$, we show transversality of the intersections in the remaining slow variable $u$.

\begin{Lemma} \label{lem:front_tr}
Consider~\eqref{eq:fast_far} for $\delta=0$ and $b,m>0$, and fix $\bar{p},\bar{k}>0$. The following hold:
\begin{enumerate}[(i)]
\item \label{lem:front_tri} The manifolds $\mathcal{W}^{\mathrm{u}}(\mathcal{M}^0_0)$ and $\mathcal{W}^{\mathrm{s}}(\mathcal{M}^+_0)$ intersect transversely along the three-dimensional manifold 
\begin{align}
\mathcal{H}_{\mathrm{dv}} = \bigcup_{|p|\leq \bar{p}, k\in[0,\bar{k}]}\phi_\mathrm{dv}.
\end{align}
\item \label{lem:front_trii} The manifolds $\mathcal{W}^{\mathrm{u}}(\mathcal{M}^+_0)$ and $\mathcal{W}^{\mathrm{s}}(\mathcal{M}^0_0)$ intersect transversely along the three-dimensional manifold 
\begin{align}
\mathcal{H}_{\mathrm{vd}} = \bigcup_{|p|\leq \bar{p}, k\in[0,\bar{k}]}\phi_\mathrm{vd}.
\end{align}
\end{enumerate}
\end{Lemma}
\begin{proof}
We focus on the statement~\ref{lem:front_tri}, as the proof of~\ref{lem:front_trii} is nearly identical. To prove transversality, it remains to show that the intersection breaks transversely when varying the remaining slow variable $u$ near $u=u_\mathrm{f}=\frac{9bm}{2}$. We accomplish this by computing the splitting of the manifolds $\mathcal{W}^{\mathrm{u}}(\mathcal{M}^0_0)$ and $\mathcal{W}^{\mathrm{s}}(\mathcal{M}^+_0)$ to leading order in $|u- u_\mathrm{f}|$ via a Melnikov-type computation.

Given $(p,k)\in [-\bar{p}, \bar{p}]\times [0,\bar{k}]$, the front $\phi_\mathrm{dv}(s)=(v_\mathrm{dv}, q_\mathrm{dv})(s)$ which lies in the intersection $\mathcal{W}^{\mathrm{u}}(\mathcal{M}^0_0)\cap \mathcal{W}^{\mathrm{s}}(\mathcal{M}^+_0)$ is a solution of the fast layer equation~\eqref{eq:layer} at $u=u_\mathrm{f}$. The adjoint equation associated with the linearization of~\eqref{eq:layer} about $\phi_\mathrm{dv}$ at $u=u_\mathrm{f}$ is given by
\begin{align}
\begin{split}\label{eq:layer_adj}
\begin{pmatrix} v_s \\ q_s\end{pmatrix} = \begin{pmatrix}0 &  -m+u_\mathrm{f}\left(2v_\mathrm{dv}(s)-3bv_\mathrm{dv}(s)^2\right)\\ -1 & 0 \end{pmatrix}\begin{pmatrix} v \\ q\end{pmatrix},
\end{split}
\end{align}
and admits a unique bounded solution (up to multiplication by a constant) $\psi_\mathrm{dv}(s):= \left(-q_\mathrm{dv}(s), v_\mathrm{dv}(s)\right)$. To leading order, the splitting distance of the manifolds $\mathcal{W}^{\mathrm{u}}(\mathcal{M}^0_0)\cap \mathcal{W}^{\mathrm{s}}(\mathcal{M}^+_0)$ is determined to leading order in $|u-u_\mathrm{f}|$ by the Melnikov integral
\begin{align}
M_\mathrm{dv}^u := \int_{-\infty}^\infty D_uF(\phi_\mathrm{dv}(s);u_\mathrm{f})\cdot\psi_\mathrm{dv}(s)\mathrm{d}s,
\end{align}
where $F(v,q;u)$ denotes the right-hand-side of~\eqref{eq:layer}. We compute
\begin{align}
M_\mathrm{dv}^u := \int_{-\infty}^\infty -v_\mathrm{dv}(s)^3(1-bv_\mathrm{dv}(s))\mathrm{d}s<0,
\end{align}
from which we determine that the intersection $\mathcal{W}^{\mathrm{u}}(\mathcal{M}^0_0)\cap \mathcal{W}^{\mathrm{s}}(\mathcal{M}^+_0)$ is transverse in varying the slow variable $u\approx u_\mathrm{f}$.

The proof of~\ref{lem:front_trii} proceeds similarly, and transversality is then determined by the Melnikov coefficient
\begin{align}
M_\mathrm{vd}^u := \int_{-\infty}^\infty D_uF(\phi_\mathrm{vd}(s);u_\mathrm{f})\cdot\psi_\mathrm{vd}(s)\mathrm{d}s,
\end{align}
where $\psi_\mathrm{vd}(s):= \left(-q_\mathrm{vd}(s), v_\mathrm{vd}(s)\right)$. In that case, we similarly find that
\begin{align}
M_\mathrm{vd}^u := \int_{-\infty}^\infty -v_\mathrm{vd}(s)^3(1-bv_\mathrm{vd}(s))\mathrm{d}s<0.
\end{align}

\end{proof}

\subsubsection{Far-field stable manifold $\mathcal{W}^{\mathrm{s}}(\mathcal{B}_\delta^\mathrm{far})$}\label{sec:farfieldmanifold}

We now construct the set of orbits which remain bounded in the far-field, and in particular those which converge to the desert state $v=0$; these orbits will be asymptotic to the manifold $\mathcal{M}^0_\delta$.

We first examine the reduced flow~\eqref{eq:reduced} on $\mathcal{M}^0_0$, given by
\begin{align}
\begin{split}\label{eq:reduced_0}
 u_r&=  p\\
  p_r&= -kp-a+u\\
k_r &= -k^2.
\end{split}
\end{align}
This system admits a single equilibrium $p^\mathrm{far}_0=(a,0,0)$ corresponding to the desert steady state of~\eqref{eq:modifiedKlausmeier} in the far-field limit $r\to \infty$. This fixed point is nonhyperbolic when considered as a fixed point of~\eqref{eq:reduced_0}, but is of saddle-type when restricted to the invariant plane $k=0$. Within $\mathcal{M}^0_0$, this fixed point admits a two-dimensional center-stable manifold $\mathcal{B}_0^\mathrm{far}$ representing the set of solutions $(u,p)(r)$ of~\eqref{eq:reduced_0} which remain bounded as $r\to \infty$.

\begin{figure}
\centering
\includegraphics[width=0.45\linewidth]{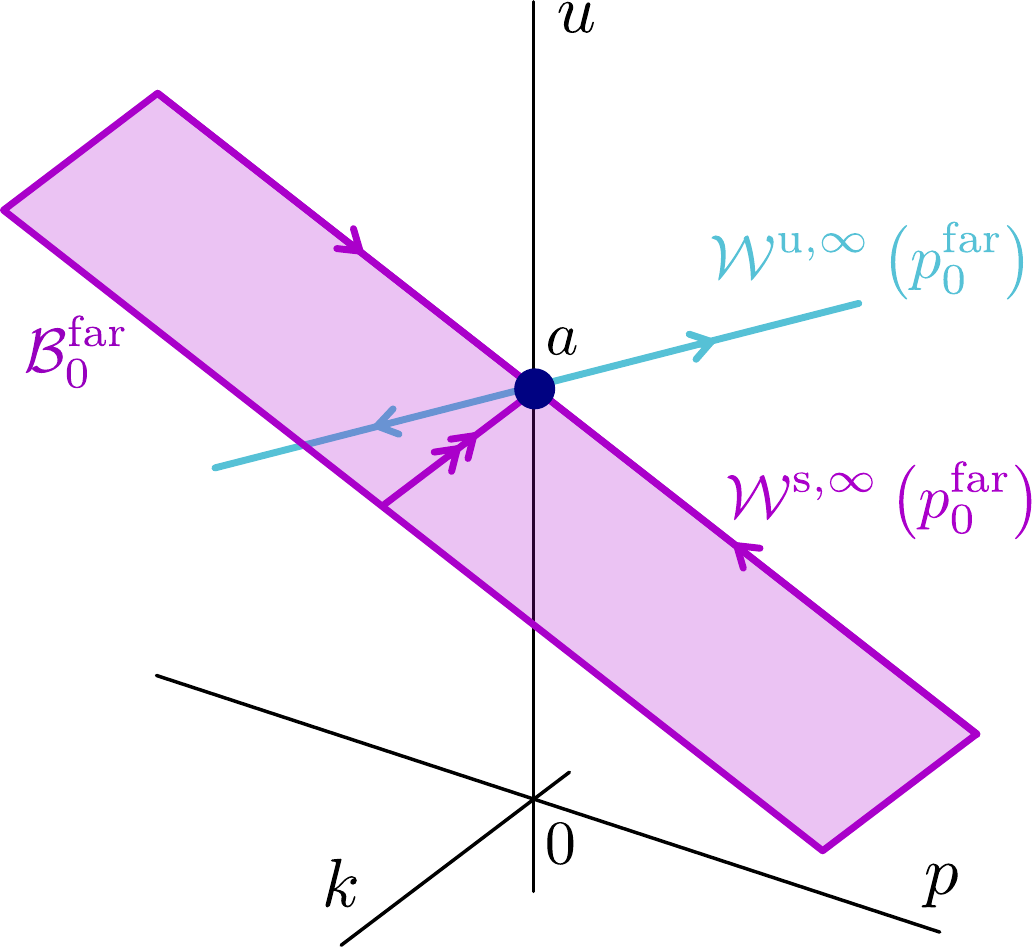}
\caption{Shown are the dynamics of~\eqref{eq:reduced_0} for $\delta=0$ on the far-field manifold $\mathcal{M}^0_0$. The manifold $\mathcal{B}^\mathrm{far}_0$ is the set of all solutions which converge to the equilibrium $p^\mathrm{far}_0=(a,0,0)$ and remain bounded in the far field as $k\to0$. }
\label{fig:far_field}
\end{figure}

To obtain a more explicit description of this set, we instead express~\eqref{eq:reduced_0} as the linear equation
\begin{align}
\begin{split}\label{eq:reduced_0_bessel}
 u_{rr}+\frac{u_r}{r}-u+a&=  0,\\
\end{split}
\end{align}
the solutions of which are given in terms of the modified Bessel functions $I_0,K_0$ of the first and second kind. In particular, the unique solution of~\eqref{eq:reduced_0_bessel} which is bounded as $r\to \infty$ and satisfies $u(\bar{r})=\bar{u}$ for given $\bar{u}\in\mathbb{R}$ and $\bar{r}>0$ is given by
\begin{align}
u^\mathrm{far}_0(r; \bar{u},\bar{r}) = a+\frac{\bar{u}-a}{K_0(\bar{r})}K_0(r).
\end{align}

For any $\bar{r}>0$, we can therefore express $\mathcal{B}_0^\mathrm{far}$ as
\begin{align}\label{eq:bfar}
\begin{split}
\mathcal{B}_0^\mathrm{far} &= \left\{(u,p,k)\in\mathbb{R}^2\times [0,\bar{k}]: \begin{pmatrix}u\\p   \end{pmatrix} = \begin{pmatrix}a+\frac{\bar{u}-a}{K_0(\bar{r})}K_0(1/k)\\ \frac{a-\bar{u}}{K_0(\bar{r})}K_1(1/k)\end{pmatrix},  \bar{u}\in\mathbb{R} \right\}\\
&=\left\{(u,p,k)\in\mathbb{R}^2\times [0,\bar{k}]: \begin{pmatrix}u\\p   \end{pmatrix} = \begin{pmatrix}a+cK_0(1/k)\\ -cK_1(1/k)\end{pmatrix},  c\in\mathbb{R} \right\},
\end{split}
\end{align}
where $K_1(z)=-K_0'(z)$, and we note that since $K_0(z), K_1(z)\to0$ exponentially as $|z|\to\infty$, the quantities $K_0(1/k), K_1(1/k)$ are well defined (and converge to zero exponentially) as $k\to0$.

Within the stable manifold $\mathcal{W}^{\mathrm{s}}(\mathcal{M}^0_0)$ of $\mathcal{M}^0_0$, we can construct the stable manifold of $\mathcal{B}_0^\mathrm{far}$ as the set $\mathcal{W}^{\mathrm{s}}(\mathcal{B}_0^\mathrm{far})$ of stable fibers over trajectories in $\mathcal{B}_0^\mathrm{far}$. This set comprises the singular limit of all solutions which are bounded in the far-field. The set of orbits $\mathcal{B}_0^\mathrm{far}$ perturbs within $\mathcal{M}^0_0$ for sufficiently small $\delta>0$ to a two-dimensional manifold $\mathcal{B}_\delta^\mathrm{far}$ consisting of orbits within $\mathcal{M}^0_\delta$ which are bounded as $r\to \infty$ and in particular converge to the equilibrium $p^\mathrm{far}_0=(a,0,0)$. Likewise, as a subset of $\mathcal{W}^{\mathrm{s}}(\mathcal{M}^0_0)$, the manifold $\mathcal{W}^{\mathrm{s}}(\mathcal{B}_0^\mathrm{far})$ perturbs to a three-dimensional invariant manifold $\mathcal{W}^{\mathrm{s}}(\mathcal{B}_\delta^\mathrm{far})\subset\mathcal{W}^{\mathrm{s}}(\mathcal{M}^0_\delta) $ consisting of stable fibers lying over trajectories within $\mathcal{B}_\delta^\mathrm{far}\subset \mathcal{M}^0_\delta$.

The manifold $\mathcal{W}^{\mathrm{s}}(\mathcal{B}_\delta^\mathrm{far})$ thus describes the set of solutions which remain bounded as $r\to \infty$ and in particular those converging to the homogeneous equilibrium $P_2$ of~\eqref{eq:fast}. In light of the results of Lemma~\ref{lem:front_tr} in~\S\ref{sec:fast_jumps}, the perturbed manifolds $\mathcal{W}^{\mathrm{u}}(\mathcal{M}^+_\delta)\cap \mathcal{W}^{\mathrm{s}}(\mathcal{M}^0_\delta)$ intersect transversely in a three-dimensional manifold which lies within $\mathcal{O}(\delta)$ of the subspace $u=u_\mathrm{f}$. Viewed as a three-dimensional submanifold of $\mathcal{W}^{\mathrm{s}}(\mathcal{M}^0_\delta)$, the manifold $\mathcal{W}^{\mathrm{s}}(\mathcal{B}_\delta^\mathrm{far})$ is an $\mathcal{O}(\delta)$-perturbation of the singular manifold $\mathcal{W}^{\mathrm{s}}(\mathcal{B}_0^\mathrm{far})$ consisting of stable fibers over trajectories within $\mathcal{B}_0^\mathrm{far}\subset \mathcal{M}^0_0$ defined as in~\eqref{eq:bfar}. The manifold $\mathcal{W}^{\mathrm{s}}(\mathcal{B}_0^\mathrm{far})$ therefore transversely intersects $\mathcal{W}^{\mathrm{u}}(\mathcal{M}^+_0)$ in a two-dimensional manifold $\mathcal{H}^\mathrm{far}_0 \subset \mathcal{H}_{\mathrm{vd}} $ consisting of the orbits
\begin{align}
\mathcal{H}^\mathrm{far}_0 = \mathcal{H}_{\mathrm{vd}}\cap\left\{ p = \frac{a-u_\mathrm{f}}{K_0(1/k)}K_1(1/k), k\in[0,\bar{k}]\right\}
\end{align}
within the subspace $\{u=u_\mathrm{f}\}$. This transverse intersection persists for all sufficiently small $\delta>0$, with $\mathcal{W}^{\mathrm{s}}(\mathcal{B}_\delta^\mathrm{far})$ transversely intersecting $\mathcal{W}^{\mathrm{u}}(\mathcal{M}^+_\delta)$ in a two-dimensional manifold $\mathcal{H}^\mathrm{far}_\delta$ which lies within $\mathcal{O}(\delta)$ of $\mathcal{H}^\mathrm{far}_0 $, and likewise $\mathcal{O}(\delta)$-close to the subspace $\{u=u_\mathrm{f}\}$.

\subsection{The core region}

In this section, we construct a three-dimensional manifold of orbits which remain bounded and converge to a set of uniformly vegetated states as $r\to 0$.

\subsubsection{The center-unstable core manifold $\mathcal{W}^\mathrm{cu}_\delta(\mathcal{C})$}
For the core region, we consider $r\in[0,r_\mathrm{c}]$, where $r_\mathrm{c}=\delta s_\mathrm{c}$, or equivalently $s\in[0,s_\mathrm{c}]$, for some $s_\mathrm{c}>0$ fixed independently of $\delta>0$. We use a blow-up rescaling $z=\log s$ to obtain the dynamics in the core region as
\begin{align}
\begin{split}\label{eq:core}
 u_z&=  \delta sp\\
  p_z&= -p-\delta s(a-u-uv^2)\\
v_z&= sq\\
q_z&= -q +s\left(mv-uv^2(1-bv)\right)\\
s_z&= s.
\end{split}
\end{align}
Note that this system admits a two-dimensional manifold of equilibria at $s=0$ defined by $\mathcal{E}=\{p=q=s=0\}$. When $\delta=0$,~\eqref{eq:core} becomes
\begin{align}
\begin{split}\label{eq:core0}
 u_z&=  0\\
  p_z&= -p\\
v_z&= sq\\
q_z&= -q +s\left(mv-uv^2(1-bv)\right)\\
s_z&= s,
\end{split}
\end{align}
and near values of $(u,v)$ satisfying $mv-uv^2(1-bv)=0$, solutions of~\eqref{eq:core0} which are bounded as $z\to -\infty$ can be expanded in terms of modified Bessel functions as follows. For given $u>4bm$, there are three solutions $v=v_*(u)$ of $mv-uv^2(1-bv)=0$, namely $v_*(u)=0$ or $v_*(u)=v_\pm(u)$. Fixing one of these choices, we consider $v\approx v_*(u)$, so that $v=v_*(u)+\bar{v}$, where $|\bar{v}|\ll1$ is assumed small.

In the case of spots, we are interested in solutions near the vegetated state in the core, hence we take $v_*(u)=v_+(u)$. Substituting into~\eqref{eq:core0}, we obtain
\begin{align}
\begin{split}\label{eq:core0_loc}
 u_z&=  0\\
  p_z&= -p\\
\bar{v}_z&= sq\\
q_z&= -q +s\left(\kappa\bar{v}+\mathcal{O}(\bar{v}^2)\right)\\
s_z&= s,
\end{split}
\end{align}
where $\kappa:= ubv_+(u)(v_+(u)-v_-(u))$. Our aim is to construct solutions of~\eqref{eq:core} which remain bounded as $z\to -\infty$ ($s\to 0$). In view of~\eqref{eq:core0}, when $\delta=0$, any such solution must satisfy $p\equiv 0$, while the equation for $(v,q)$ can be re-expressed in terms of the fast variable $s$ as
\begin{align}
\label{eq:core_bessel}
\bar{v}_{ss}+\frac{\bar{v}_s}{s}-\kappa\bar{v}+\mathcal{O}(\bar{v}^2)= 0,
\end{align}
which, at the linear level, is a zero-order Bessel-type equation whose solutions can be expressed as linear combinations of the modified Bessel functions $I_0(\sqrt{\kappa} s), K_0(\sqrt{\kappa} s)$ of the first and second kind. The function $I_0(\zeta)$ is bounded at $\zeta=0$, while $K_0$ diverges logarithmically. 

Therefore, given $u_0>4mb$ and $s_\mathrm{c}>0$, we linearize~\eqref{eq:core} about the solution  {$(u,p,v,q,s)=(u_0,0,v_+(u_0),0,e^{z})$}, integrate, and solve the resulting fixed-point equation in terms of the Bessel functions $I_0(\cdot), K_0(\cdot)$. Defining the subset $\mathcal{C}:=\{p=q=s=0, v=v_+(u), u>4bm\}\subset \mathcal{E}$, this allows us to construct a local three-dimensional center-unstable manifold $\mathcal{W}^\mathrm{cu}_\delta(\mathcal{C})$ of solutions which are bounded as $s\to0$ ($z\to -\infty$), and in particular converge to $\mathcal{C}$ as $s\to0$. This manifold admits the expansion
\begin{align}\label{eq:wcu_expansion}
\mathcal{W}^\mathrm{cu}_\delta(\mathcal{C})=\left\{ (u,p,v,q,s)\in \mathbb{R}^5: \begin{pmatrix} u \\ p \\ v \\q \end{pmatrix} = \begin{pmatrix} u_0+\mathcal{O}(\delta) \\ \mathcal{O}(\delta)\\ v_+(u_0)+cI_0\left(\sqrt{\kappa}s\right)+\mathcal{O}(\delta+c^2)\\ c\sqrt{\kappa}I_1\left(\sqrt{\kappa}s\right)+\mathcal{O}(\delta+c^2)\end{pmatrix}, u_0>4bm, 0\leq s \leq s_\mathrm{c}, |c|\leq c_0\right\}
 \end{align}
for sufficiently small $|c_0|$ and $\delta\ll 1$. Here $I_1=I_0'$ is the first-order modified Bessel function of the first kind.  {The manifold $\mathcal{W}^\mathrm{cu}_\delta(\mathcal{C})$ is parameterized by $u_0$ and $c$, both of which will be selected by intersecting with the three-dimensional far-field manifold $\mathcal{W}^{\mathrm{s}}(\mathcal{B}_\delta^\mathrm{far})$ to obtain a spot solution in the full five-dimensional phase space.}

\subsubsection{Core transition region}
In the previous subsection, we constructed the core center-unstable manifold $\mathcal{W}^\mathrm{cu}_\delta(\mathcal{C})$ of solutions which remain bounded as $s\to0$. Given $s_\mathrm{c}>0$, we obtained expansions for the manifold $\mathcal{W}^\mathrm{cu}_\delta(\mathcal{C})$ valid for $0\leq s\leq s_\mathrm{c}$, provided $\delta$ is taken sufficiently small. We now aim to track this manifold into the region $r=\delta s=\mathcal{O}(1)$.

Hence we consider $r\in[\delta s_\mathrm{c}, r_0]$ where $s_\mathrm{c}$ is as above, fixed large independently of $\delta$, and $r_0>0$ is fixed independently of $\delta$, and $\delta$ is taken sufficiently small. We define the quantity $\tilde{\delta}:=\delta s_\mathrm{c}$, and we note that since $s_\mathrm{c}\gg1$ will be fixed independently of $\delta$, in the limit $\delta \to 0$ we have that $\tilde{\delta}=\mathcal{O}(\delta)$ so that $\tilde{\delta}$ can be bounded as small as desired. We return to the fast system~\eqref{eq:fast}, appending an equation for $r$, which results in the following system.
\begin{align}
\begin{split}\label{eq:fast_sc}
 u_s&=  \frac{\tilde{\delta}}{s_\mathrm{c}} p\\
  p_s&= -\frac{\tilde{\delta} p}{s_\mathrm{c}r}- \frac{\tilde{\delta}}{s_\mathrm{c}}(a-u-uv^2)\\
v_s&= q\\
q_s&= -\frac{\tilde{\delta} q}{s_\mathrm{c}r} +mv-uv^2(1-bv)\\
r_s&=  \frac{\tilde{\delta}}{s_\mathrm{c}}.
\end{split}
\end{align}
We view this as a slow-fast system with timescale separation parameter $1/s_\mathrm{c}$. In the region $r\geq \tilde{\delta}$, the variables $u,p,r$ are slow, while $v,q$ are fast. Rescaling $s=s_\mathrm{c} \zeta$, we obtain the corresponding slow system
\begin{align}
\begin{split}\label{eq:slow_sc}
 u_\zeta&=  \tilde{\delta}p\\
  p_\zeta&= -\frac{\tilde{\delta} p}{r}-\tilde{\delta}(a-u-uv^2)\\
\frac{1}{s_\mathrm{c}}v_\zeta&= q\\
\frac{1}{s_\mathrm{c}}q_\zeta&= -\frac{\tilde{\delta} q}{r} +mv-uv^2(1-bv)\\
r_\zeta&=  \tilde{\delta}
\end{split}
\end{align}
Letting $1/s_\mathrm{c} \to 0$, this system admits a three-dimensional critical manifold $\mathcal{M}^\mathrm{c}_0 = \{q=0, mv=uv^2(1-bv)\}$, and the branch $\mathcal{M}^{\mathrm{c},+}_0 = \{q=0, v=v_+(u), u>4mb\}\subset \mathcal{M}^\mathrm{c}_0$ is normally hyperbolic, of saddle type. The reduced flow on $\mathcal{M}^{\mathrm{c},+}_0$ is given by 
\begin{align}
\begin{split}\label{eq:reduced_sc}
 u_\zeta&=  \tilde{\delta}p\\
  p_\zeta&= -\frac{\tilde{\delta} p}{r}-\tilde{\delta}(a-u-uv_+(u)^2)\\
r_\zeta&=  \tilde{\delta},
\end{split}
\end{align}
noting that the vector field is uniformly bounded in the region of interest $r\in[\tilde{\delta},r_0]$. In order to determine the dynamics in this region, we desingularize the system and rescale the independent variable $\mathrm{d}\zeta = \frac{r}{\tilde{\delta}}\mathrm{d}\tilde{\zeta}$, which results in the system
\begin{align}
\begin{split}\label{eq:reduced_sc_desing}
 u_{\tilde{\zeta}}&=  rp\\
  p_{\tilde{\zeta}}&= -p-r(a-u-uv_+(u)^2)\\
r_{\tilde{\zeta}}&= r.
\end{split}
\end{align}
The system~\eqref{eq:reduced_sc_desing} admits an invariant manifold at $r=0$ with dynamics
\begin{align}
\begin{split}\label{eq:reduced_sc_desing0}
 u_{\tilde{\zeta}}&=  0\\
  p_{\tilde{\zeta}}&= -p
\end{split}
\end{align}
and a line of equilibria given by $\ell_0:=\{(u,p)=(u_0,0), u_0>4bm\}$ which are attracting within the manifold $\{r=0\}$, each with a one dimensional stable manifold. In the normal ($r$) direction, this line of equilibria is repelling and admits a unique two-dimensional unstable manifold $\mathcal{W}^\mathrm{u}(\ell_0)$, which satisfies the following proposition, the proof of which follows by standard invariant manifold theory.

\begin{Proposition}\label{prop:well0}
Consider~\eqref{eq:reduced_sc_desing}. The line of equilibria $\ell_0$ admits a unique two-dimensional unstable manifold $\mathcal{W}^\mathrm{u}(\ell_0)$, which for all sufficiently small $r_0>0$ can be represented as a graph
$\mathcal{W}^\mathrm{u}(\ell_0) = \{p = h_0(u,r),0\leq r< r_0 \}$ where 
\begin{align}
h_0(u,r)=\frac{\left(u-a+uv_+(u)^2\right)}{2}r\left(1+\mathcal{O}(r)\right).
\end{align}
\end{Proposition}

However, it does not suffice to restrict our attention to small values of $r$, and in fact $r$ may need to be taken large. To understand how solutions originating near $r=0$ behave for large values of $r$, we need further information on the nonlinear vector field~\eqref{eq:reduced_sc_desing}. In general, detailed estimates are not available for $\mathcal{W}^\mathrm{u}(\ell_0)$ when $r$ is not small. However, more can be said near a value of $u$ which satisfies $a-u-uv_+(u)^2=0$.  Note that this is the condition satisfied by the equilibrium $P_{2}$ of the full system~\eqref{eq:fast} provided this equilibrium lies on the branch $v=v_+(u)$. 

In this case, since $a-U_2-U_2v_+(U_2)^2=0$, the line $\ell_2:=\{u=U_2,p=0\}$ is invariant. Examining the linearization of~\eqref{eq:reduced_sc_desing} about the invariant line $\ell_2$ reveals a single zero eigenvalue with eigenvector $(1,0,0)$, and a negative eigenvalue $\lambda=-1$, while the dynamics along $\ell_2$ are simply $r_{\tilde{\zeta}}= r$. Therefore, there exists a two-dimensional, normally attracting manifold $\mathcal{W}^\mathrm{c}(\ell_2)$ which contains the line $\ell_2$, which can be represented as a graph 
\begin{align}
\mathcal{W}^\mathrm{c}(\ell_2) = \left\{p = h_2(u,r), 0\leq r\leq r_0, |u-U_2|\leq \delta_u\right\},
\end{align}
 where $ h_2(u,r)=\mathcal{O}(|u-U_2|)$. Moreover, in the region $0\leq r\ll1$, the manifold $\mathcal{W}^\mathrm{c}(\ell_2)$ coincides with $\mathcal{W}^\mathrm{u}(\ell_0)$, and hence for simplicity we denote the union of these manifolds by $\mathcal{W}^\mathrm{u}(\ell_0)$. We emphasize that this (combined) manifold $\mathcal{W}^\mathrm{u}(\ell_0)$ is locally invariant and normally attracting. See Figure~\ref{fig:core_transition} for a depiction of $\mathcal{W}^\mathrm{u}(\ell_0)$.
 
\begin{figure}
\centering
\includegraphics[width=0.6\linewidth]{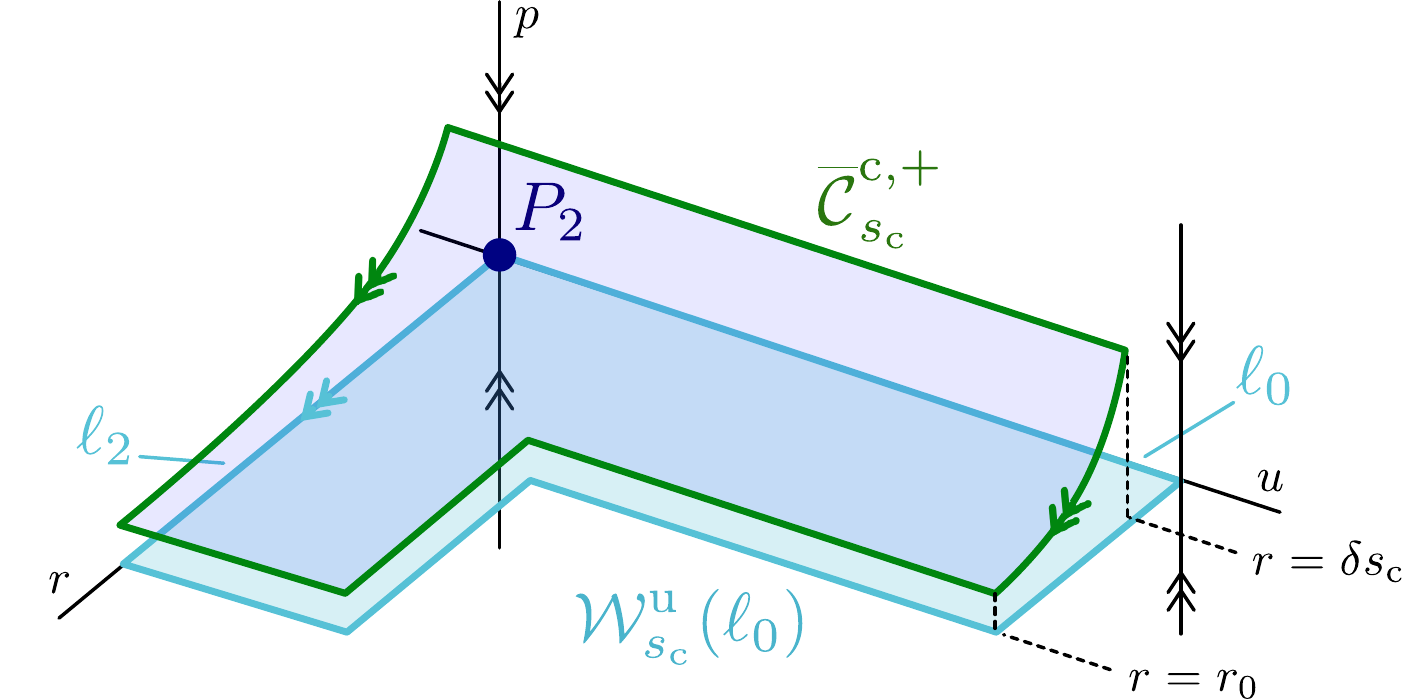}
\caption{Shown is the slow flow on $\mathcal{M}^{\mathrm{c},+}_{s_\mathrm{c}}$. The manifold $\mathcal{W}^\mathrm{cu}_\delta(\mathcal{C})$ aligns along the unstable fibers $\mathcal{W}^\mathrm{u}(\overline{\mathcal{C}}^{\mathrm{c,+}}_{s_\mathrm{c}})$ of the manifold $\overline{\mathcal{C}}^{\mathrm{c,+}}_{s_\mathrm{c}}$ in the full system~\eqref{eq:fast_sc}. Within $\mathcal{M}^{\mathrm{c},+}_{s_\mathrm{c}}$, $\overline{\mathcal{C}}^{\mathrm{c,+}}_{s_\mathrm{c}}$ aligns along $\mathcal{W}^\mathrm{u}_{s_\mathrm{c}}(\ell_0)$ under the forward evolution of~\eqref{eq:fast_sc}.}
\label{fig:core_transition}
\end{figure}

In the full system~\eqref{eq:fast_sc}, the manifold $\mathcal{M}^{\mathrm{c},+}_0$ perturbs to a three-dimensional slow manifold $\mathcal{M}^{\mathrm{c},+}_{s_\mathrm{c}}$ which is $\mathcal{O}(1/s_\mathrm{c})$-close to $\mathcal{M}^{\mathrm{c},+}_0$, with slow flow given by an $\mathcal{O}(1/s_\mathrm{c})$ perturbation of the reduced flow~\eqref{eq:reduced_sc}. In particular, the manifold $\mathcal{W}^\mathrm{u}(\ell_0)$ perturbs within $\mathcal{M}^{\mathrm{c},+}_{s_\mathrm{c}}$ to a locally invariant manifold $\mathcal{W}^\mathrm{u}_{s_\mathrm{c}}(\ell_0)$. Furthermore, the four-dimensional stable/unstable manifolds $\mathcal{W}^\mathrm{s,u}(\mathcal{M}^{\mathrm{c},+}_0)$ formed by the union of the stable/unstable fibers of basepoints on $\mathcal{M}^{\mathrm{c},+}_0$ also perturb to stable/unstable manifolds $\mathcal{W}^\mathrm{s,u}(\mathcal{M}^{\mathrm{c},+}_{s_\mathrm{c}})$. We can identify the subset of these manifolds corresponding to the (un)stable fibers of basepoints on $\mathcal{W}^\mathrm{u}_{s_\mathrm{c}}(\ell_0)$ as three-dimensional locally invariant manifolds $\mathcal{W}^\mathrm{s,u}(\mathcal{W}^\mathrm{u}_{s_\mathrm{c}}(\ell_0))$.

We now track the core center-unstable manifold $\mathcal{W}^\mathrm{cu}_\delta(\mathcal{C})$ through the region $r\in[\delta s_\mathrm{c}, r_0]$. We recall that, given any fixed (large) $s_\mathrm{c}>0$, $\mathcal{W}^\mathrm{cu}_\delta(\mathcal{C})$ admits the expansion~\eqref{eq:wcu_expansion}, which is valid up to $s=s_\mathrm{c}$ for all sufficiently small $\delta>0$. At $s=s_\mathrm{c}$ (corresponding to the subspace $r=\tilde{\delta}$), $\mathcal{W}^\mathrm{cu}_\delta(\mathcal{C})$ is aligned along the unstable fibers within $\mathcal{W}^\mathrm{u}(\mathcal{M}^{\mathrm{c},+}_{s_\mathrm{c}})$ of base point orbits on $\mathcal{M}^{\mathrm{c},+}_{s_\mathrm{c}}$ lying on a curve $\mathcal{C}^{\mathrm{c,+}}_{s_\mathrm{c}}=\left\{(u,p,r): r=\delta s_\mathrm{c}, ~p=\mathcal{O}(\delta, 1/s_\mathrm{c})\right\}$. In particular $\mathcal{W}^\mathrm{cu}_\delta(\mathcal{C})$ transversely intersects the stable fibers of these orbits within $\mathcal{W}^\mathrm{s}(\mathcal{M}^{\mathrm{c},+}_{s_\mathrm{c}})$. Tracking under the forward-flow of~\eqref{eq:fast_sc}, by the exchange lemma, $\mathcal{W}^\mathrm{cu}_\delta(\mathcal{C})$ aligns $\mathcal{O}(e^{-s_\mathrm{c}})$-close to the unstable fibers within $\mathcal{W}^\mathrm{u}(\mathcal{M}^{\mathrm{c},+}_{s_\mathrm{c}})$ of the forward evolution $\overline{\mathcal{C}}^{\mathrm{c,+}}_{s_\mathrm{c}}$ of the manifold $\mathcal{C}^{\mathrm{c,+}}_{s_\mathrm{c}}$ within $\mathcal{M}^{\mathrm{c},+}_{s_\mathrm{c}}$; see Figure~\ref{fig:core_transition}.

As the flow on $\mathcal{M}^{\mathrm{c},+}_{s_\mathrm{c}}$ is an $\mathcal{O}(1/s_\mathrm{c})$ perturbation of the reduced flow~\eqref{eq:reduced_sc}, we are thus able to determine how $\mathcal{W}^\mathrm{cu}_\delta(\mathcal{C})$ emerges at $r=r_0$, noting that for larger values of $r$, we only have detailed estimates on $\mathcal{W}^\mathrm{cu}_\delta(\mathcal{C})$ near the solution $u=U_2$ of $a-u-uv_+(u)^2=0$.

\subsection{Dynamics on $\mathcal{M}^+_\delta$}\label{sec:m_delta}
In the region $r\geq r_0$, where $r_0>0$ is taken sufficiently small and fixed independently of $\delta$, we return to the fast system~\eqref{eq:fast} and append an equation for $r$
\begin{align}
\begin{split}\label{eq:fast_r}
 u_s&=  \delta p\\
  p_s&= -\frac{\delta p}{r}-\delta(a-u-uv^2)\\
v_s&= q\\
q_s&= -\frac{\delta q}{r} +mv-uv^2(1-bv)\\
r_s&= \delta.
\end{split}
\end{align}
In this region, when $\delta=0$ this system admits a critical manifold $\mathcal{M}_0$ defined by~\eqref{eq:M0}, which can be decomposed into the branches $\mathcal{M}_0 = \mathcal{M}^0_0\cup \mathcal{M}^-_0 \cup \mathcal{F}\cup \mathcal{M}^+_0$ as in~\eqref{eq:M0branches}. For all sufficiently small $\delta>0$, (any normally hyperbolic portion of) $\mathcal{M}^+_0$ perturbs to a three-dimensional invariant manifold $\mathcal{M}^+_\delta$ and its four-dimensional (un)stable manifolds perturb to four dimensional locally invariant manifolds $\mathcal{W}^\mathrm{s,u}(\mathcal{M}^+_\delta)$. As a result of the analysis of the previous section, we know that $\mathcal{W}^\mathrm{cu}_\delta(\mathcal{C})$ approaches the set $r=r_0$ aligned along the strong unstable fibers of orbits on $\mathcal{M}^+_\delta$ which are $\mathcal{O}(1/s_\mathrm{c}+\delta)$-close to the intersection $\mathcal{W}^\mathrm{u}(\ell_0)\cap\{r=r_0\}$. By Proposition~\ref{prop:well0}, this set is given by the graph
\begin{align}\label{eq:gamma_in}
\Gamma_\mathrm{in}:=\left\{p =p_\mathrm{in}(u):= h_0(u,r_0)=\frac{\left(u-a+uv_+(u)^2\right)}{2}r_0\left(1+\mathcal{O}(r_0)\right)\right\},
\end{align}
and the projection of $\mathcal{W}^\mathrm{cu}_\delta(\mathcal{C})\cap\{r=r_0\}$ onto $\mathcal{M}^+_\delta$ along the unstable fibers within $\mathcal{W}^\mathrm{u}(\mathcal{M}^+_\delta)$ is therefore within $\mathcal{O}(1/s_\mathrm{c}+\delta)$ of this graph. 

Likewise, we consider the far-field stable manifold $\mathcal{W}^{\mathrm{s}}(\mathcal{B}_\delta^\mathrm{far})$, which we recall from~\S\ref{sec:farfieldmanifold} transversely intersects $\mathcal{W}^{\mathrm{u}}(\mathcal{M}^+_\delta)$ in a two-dimensional manifold $\mathcal{H}^\mathrm{far}_\delta$ which lies within $\mathcal{O}(\delta)$ of the set $\mathcal{H}^\mathrm{far}_0$ given by 
\begin{align}
\mathcal{H}^\mathrm{far}_0 = \mathcal{H}_{\mathrm{vd}}\cap\left\{ p = \frac{a-u_\mathrm{f}}{K_0(1/k)}K_1(1/k), k\in[0,\bar{k}]\right\},
\end{align}
where $\mathcal{H}_\mathrm{vd}$ is as in Lemma~\ref{lem:front_tr} and we recall $k=1/r$. In other words, $\mathcal{W}^{\mathrm{s}}(\mathcal{B}_\delta^\mathrm{far})$ intersects $\mathcal{W}^{\mathrm{u}}(\mathcal{M}^+_\delta)$ transversely along the unstable fibers of orbits lying within $\mathcal{O}(\delta)$ of the set
\begin{align}\label{eq:gamma_out}
\Gamma_\mathrm{out}:=\mathcal{B}_0^\mathrm{far}\cap\{u=u_\mathrm{f}\} &= \left\{u=u_\mathrm{f}, p = p_\mathrm{out}(r):= \frac{a-u_\mathrm{f}}{K_0(r)}K_1(r),  r\in [\bar{r},\infty) \right\},
\end{align}
where $\bar{r}>0$ is arbitrary. 

We aim to show the existence of $r=r_I$ such that the manifolds $\mathcal{W}^{\mathrm{s}}(\mathcal{B}_\delta^\mathrm{far})$ and $\mathcal{W}^\mathrm{cu}_\delta(\mathcal{C})$ intersect transversely at $r=r_I$ near the fast jump in the set $\{u=u_\mathrm{f}\}$. To do this, we will track orbits on $\mathcal{W}^\mathrm{cu}_\delta(\mathcal{C})$ as they evolve according to the dynamics of~\eqref{eq:fast_r} until reaching the set $\{u=u_\mathrm{f}\}$.

In order to track $\mathcal{W}^\mathrm{cu}_\delta(\mathcal{C})$ through this region, we examine the reduced flow on $\mathcal{M}^+_0$, given by
\begin{align}
\begin{split}\label{eq:slow_r_red}
 u_\xi&=   p\\
  p_\xi&= -\frac{p}{r}-(a-u-uv_+(u)^2)\\
r_\xi&= 1,
\end{split}
\end{align}
where we've introduced the variable $\mathrm{d} \xi=\mathrm{d}r$. At $r=r_0$ orbits of $\mathcal{W}^\mathrm{cu}_\delta(\mathcal{C})$ are aligned along the unstable fibers of orbits crossing $\Gamma_\mathrm{in}$. At $u=u_\mathrm{f}$, $\mathcal{W}^{\mathrm{s}}(\mathcal{B}_\delta^\mathrm{far})$ are aligned along orbits crossing $\Gamma_\mathrm{out}$. Hence we aim to show in the reduced flow~\eqref{eq:slow_r_red} that the forward evolution of trajectories in $\Gamma_\mathrm{in}$ transversely intersects the set $\Gamma_\mathrm{out}$ within the set $\{u=u_\mathrm{f}\}$. This transverse intersection will then persist under perturbation, thereby obtaining the transverse intersection of $\mathcal{W}^{\mathrm{s}}(\mathcal{B}_\delta^\mathrm{far})$ and $\mathcal{W}^\mathrm{cu}_\delta(\mathcal{C})$ for sufficiently small $\delta>0$.

We have the following proposition, which is the main result of this section.
\begin{Proposition}\label{prop:reduced_tr}
Consider the set $\Gamma_\mathrm{in}$ of initial conditions at $r=r_0$ for the system~\eqref{eq:slow_r_red}. The forward evolution of $\Gamma_\mathrm{in}$ under the flow of~\eqref{eq:slow_r_red} traces out a two-dimensional manifold $\overline{\Gamma}_\mathrm{in}$ which intersects the set $\{u=u_\mathrm{f}\}$ in a curve $\Gamma_\mathrm{f}$. If the parameters $a,b,m$, satisfy
\begin{align}
\int_{U_2}^{u_\mathrm{f}} \frac{u - 2mb + \sqrt{u^2 - 4umb}}{2b^2} \mathrm{d}u> \frac{1}{2}(a-U_2)^2,
\end{align}
then, within $\{u=u_\mathrm{f}\}$, there exists $r=r_I>0$ such that the curve $\Gamma_\mathrm{f}$ transversely intersects $\Gamma_\mathrm{out}$ at $r=r_I$.
\end{Proposition}
We begin with the following lemma which describes the set $\Gamma_\mathrm{out}$, given by the graph of the function $p_\mathrm{out}(r)$ in~\eqref{eq:gamma_out}.
\begin{Lemma} \label{lem:out_monotonic} Regarding the function $p_\mathrm{out}(r)$, the following hold.
\begin{enumerate}[(i)]
\item \label{lem:out_monotonici} $p_\mathrm{out}'(r)<0$ for $r\in(0,\infty)$.
\item \label{lem:out_monotonicii} $\lim_{r\to0}p_\mathrm{out}(r) = \infty$
\item \label{lem:out_monotoniciii} $\lim_{r\to\infty}p_\mathrm{out}(r) = a-u_\mathrm{f}$
\end{enumerate}
\end{Lemma}
\begin{proof}
For~\ref{lem:out_monotonici}, we recall that $p_\mathrm{out}(r) = \frac{a-u_\mathrm{f}}{K_0(r)}K_1(r)$. Note that $a - u_\mathrm{f} > 0$.
We have that 
\begin{align*}
    p'(r) &= (a - u_\mathrm{f})\left(\frac{K_1(r)^2}{K_0(r)^2} - \frac{K_2(r)}{2K_0(r)} - \frac{1}{2}\right) \\
    &= \frac{a - u_\mathrm{f}}{K_0(r)^2}\left(K_1(r)^2 - \frac{1}{2}K_2(r)K_0(r) - \frac{1}{2}K_0(r)^2\right).
\end{align*}

Using the integral form for products of Bessel functions~\cite[\S10.32.17]{handbook}, we see that 
\begin{align*}
    \frac{1}{2}K_1(r)^2 &= \int_{0}^{\infty} K_2(2r \cosh(t)) dt= \int_{0}^{\infty} K_0(2r \cosh(t)) \cosh(2t) dt \text{, }\\
    K_2(r)K_0(r) &= 2\int_{0}^{\infty} K_2(2r \cosh(t)) \cosh(2t) dt \\
    K_0(r)^2 &= 2\int_{0}^{\infty} K_0(2r \cosh(t))  dt.
\end{align*}
Note that
\begin{align*}
K_1(r)^2 = \int_{0}^{\infty} K_0(2r \cosh(t)) \cosh(2t) dt + \int_{0}^{\infty} K_2(2r \cosh(t)) dt,
\end{align*}
so that
\begin{align*}
    K_1(r)^2 - \frac{1}{2}K_2(r)K_0(r) - \frac{1}{2}K_0(r)^2
    &= \int_{0}^{\infty} K_0(2r \cosh(t)) \cosh(2t) dt + \int_{0}^{\infty} K_2(2r \cosh(t)) dt \\
    &-  \int_{0}^{\infty} K_2(2r \cosh(t)) \cosh(2t) dt - \int_{0}^{\infty} K_0(2r \cosh(t))  dt\\
    &= \int_{0}^{\infty} \left[ \cosh(2t) - 1 \right] \left[K_0(2r \cosh(t))- K_2(2r \cosh(t)) \right] dt\\
    & < 0,
\end{align*}
since $\cosh(2t) - 1 \geq 0$ and $K_0(x)- K_2(x) < 0$ for all $x \geq 0$, which completes the proof of~\ref{lem:out_monotonici}.

The limits~\ref{lem:out_monotonicii} and~\ref{lem:out_monotoniciii} follow directly from asymptotic properties of the functions $K_0,K_1$~\cite[\S10.30]{handbook}.
\end{proof}

We now describe the evolution of the set of initial conditions $\Gamma_\mathrm{in}$ to the set $u=u_\mathrm{f}$. At $r=r_0$, we represent the initial conditions $\Gamma_\mathrm{in}$ via~\eqref{eq:gamma_in} as the graph $p=p_\mathrm{in}(u):=h_0(u,r_0)$. For $a/m>\max\left\{9b/2, 4b+1/b\right\}$, the function $f(u):=-(a-u-uv_+(u)^2)$ admits a unique zero $u=U_2\in(4bm,u_\mathrm{f})$, coinciding with the uniformly vegetated equilibrium state $P_2$. We have the following lemma.
\begin{Lemma} \label{lem:rf_well_defined} 
Fix $u_\mathrm{in}\in(U_2,u_\mathrm{f})$. The forward evolution of the initial condition in $\Gamma_\mathrm{in}$ given by $p=p_\mathrm{in}(u_\mathrm{in})$ at $r=r_0$ eventually reaches the set $\{u=u_\mathrm{f}\}$ at some value of $(p,r)=(p_\mathrm{f},r_\mathrm{f})(u_\mathrm{in})$.
\end{Lemma}
\begin{proof}
We show that for any initial condition $u_\mathrm{in}\in(U_2,u_\mathrm{f})$, under the flow of~\eqref{eq:slow_r_red}, that $u_\xi$ is nondecreasing and therefore eventually the $u$ coordinate will reach $u=u_\mathrm{f}$ at some $r=r_\mathrm{f}(u_\mathrm{in})$. Since $u_\xi = p$, we achieve this by showing that $p_\xi\geq 0$ along such a trajectory, which ensures that $u_\xi > p_\mathrm{in}(u_\mathrm{in})>0$. We note that $p_\xi > 0$ initially at $r=r_0$ (via~\eqref{eq:gamma_in} and~\eqref{eq:slow_r_red}), and at any location where $p_\xi=0$, we have that $p_{\xi\xi}= \frac{p}{r^2} + f'(u)p > 0$. Thus, $u_\xi =p$ is nondecreasing, which ensures that $u$ will increase towards $u_\mathrm{in}$.
\end{proof}
Taken over all values of $u_\mathrm{in}\in(U_2,u_\mathrm{f})$, we obtain a curve $(p,r)=(p_\mathrm{f},r_\mathrm{f})(u_\mathrm{in})$ parameterized by the initial $u$-coordinate $u_\mathrm{in}\in(U_2,u_\mathrm{f})$. In order to prove Proposition~\ref{prop:reduced_tr}, we show that the curve $(p,r)=(p_\mathrm{f},r_\mathrm{f})(u_\mathrm{in})$ transversely intersects the curve $p = p_\mathrm{out}(r)$ (that is, the set $\Gamma_\mathrm{out}$) within the plane $u=u_\mathrm{f}$, for some value of $u_\mathrm{in}$ and corresponding $r=r_\mathrm{f}(u_\mathrm{in})=:r_I$.

\begin{Lemma} \label{lem:p_monotonic_u} The curve $(p,r)=(p_\mathrm{f},r_\mathrm{f})(u_\mathrm{in})$ satisfies $r_\mathrm{f}'(u_\mathrm{in})<0$ and $p_\mathrm{f}'(u_\mathrm{in})<0$ for $u_\mathrm{in}\in(U_2,u_\mathrm{f})$.
\end{Lemma}

\newcommand{\pt}{\Tilde{p}}
\newcommand{\ut}{\Tilde{u}}

\begin{proof}
We begin with the statement concerning the sign of $r_\mathrm{f}'(u_\mathrm{in})$. Consider two trajectories with two different initial conditions $u_\mathrm{in}=u_{\mathrm{in},1}$ and $u_\mathrm{in}=u_{\mathrm{in},2}$, with $u_{\mathrm{in},1} < u_{\mathrm{in},2}$, which trace out solution curves $(u,p)=(u(r;u_\mathrm{in}),p(r;u_\mathrm{in}))$. At $r=r_0$, we have that $p_\mathrm{in}(u_{\mathrm{in},1})=p(r_0; u_{\mathrm{in},1})< p(r_0; u_{\mathrm{in},2})= p_\mathrm{in}(u_{\mathrm{in},2})$ by~\eqref{eq:gamma_in}.

We claim this implies $p(r; u_{\mathrm{in},1})< p(r; u_{\mathrm{in},2})$ for all $r>r_0$. Suppose for contradiction that $p(\tilde{r};u_{\mathrm{in},2}) = p(\tilde{r}; u_{\mathrm{in},2})$, or for some $r=\tilde{r}$, which represents the first $r$ value where the two trajectories cross. (Note that since $p(r_0; u_{\mathrm{in},1})< p(r_0; u_{\mathrm{in},2})$, and $p(r;u_\mathrm{in})$ is continuous, we know that there exists this first value $\tilde{r}$.) Thus $p(r;u_{\mathrm{in},2}) < p(r; u_{\mathrm{in},2})$ for all $r <\tilde{r}$, which implies that $u(r;u_{\mathrm{in},2}) <u(r; u_{\mathrm{in},2})$ for all $r < \tilde{r}$, since $u_{\mathrm{in},1} < u_{\mathrm{in},2}$ and $u_\xi = p$. Recall that $p_\xi = -\frac{p}{r} + f(u)$. Then, since $f(u)$ is an increasing function of $u$, we have that $p_\xi(\tilde{r};u_{\mathrm{in},1}) < p_\xi(\tilde{r};u_{\mathrm{in},2})$. (Note that this is a strict inequality since $f'(u) > 0$.) This contradicts the fact that these solution curves intersect at $\tilde{r}$ since $p(r;u_{\mathrm{in},2}) < p(r; u_{\mathrm{in},2})$ for all $r <\tilde{r}$.

Therefore $p(r;u_{\mathrm{in},2}) < p(r; u_{\mathrm{in},2})$ for all $r$. This also means that $u'(r;u_{\mathrm{in},1}) < u'(r;u_{\mathrm{in},2})$ for all $r$. Thus, since $u_{\mathrm{in},1} < u_{\mathrm{in},2}$, we also obtain that $u(r;u_{\mathrm{in},1}) < u(r;u_{\mathrm{in},2})$ for all $r$. Recall that $u(r;u_{\mathrm{in},1}) < u(r;u_{\mathrm{in},2}) < u_\mathrm{f}$ for $r  <r_\mathrm{f}(u_{\mathrm{in},2})$. Thus, $r_\mathrm{f}(u_{\mathrm{in},2}) < r_\mathrm{f}(u_{\mathrm{in},1})$. Thus, whenever $u_{\mathrm{in},1} < u_{\mathrm{in},2}$, we have that $r_\mathrm{f}(u_{\mathrm{in},2}) < r_\mathrm{f}(u_{\mathrm{in},1})$, so that $r_\mathrm{f}(u_\mathrm{in})$ is a strictly decreasing function.

We now turn to the sign of $p_\mathrm{f}'(u_\mathrm{in})$. Again, we consider two trajectories with initial conditions $u_\mathrm{in}=u_{\mathrm{in},1},u_{\mathrm{in},2}$, with $u_{\mathrm{in},1} < u_{\mathrm{in},2}$. Suppose for contradiction that $p_\mathrm{f}(u_{\mathrm{in},1}) \leq p_\mathrm{f}(u_{\mathrm{in},2})$. Since $u'(r;u_{\mathrm{in},1}), u'(r;u_{\mathrm{in},2})>0$, there exists a least $u=\tilde{u}<u_\mathrm{f}$ and $\tilde{r}_1,\tilde{r}_2<r_\mathrm{f}$ at which $u(\tilde{r}_1;u_{\mathrm{in},1})= u(\tilde{r}_2;u_{\mathrm{in},2})=\tilde{u}$ and $p(\tilde{r}_1;u_{\mathrm{in},1}) = p(\tilde{r}_2;u_{\mathrm{in},2})=\tilde{p}$. Since $f'(u)>0$, and using a similar argument as above, we have that $\tilde{r}_1 > \tilde{r}_2$. 

We express the solution curve $(u,p)=(u(r;u_{\mathrm{in},1}),p(r;u_{\mathrm{in},1}))$ as a graph $u=u_1(p)$ over $p$, and similarly the curve $(u,p)=(u(r;u_{\mathrm{in},1}),p(r;u_{\mathrm{in},1}))$ as $u=u_2(p)$. Then at $p=\tilde{p}$, we have
\begin{align*}
\frac{du_i}{dp} = \frac{\tilde{p}}{-\frac{\tilde{p}}{\tilde{r}_i} + f(\tilde{u})}, \quad i=1,2
\end{align*}
from which we see that $\frac{du_1}{dp}<\frac{du_2}{dp}$, which contradicts the fact that $u_1(p)$ must cross $u_2(p)$ from below. Thus, we conclude that for any $u_{\mathrm{in},1} < u_{\mathrm{in},2}$, we have $p_\mathrm{f}(u_{\mathrm{in},1}) > p_\mathrm{f}(u_{\mathrm{in},2})$, so $p_\mathrm{f}(u_\mathrm{in})$ is a strictly decreasing function of $u_\mathrm{in}$.
\end{proof}

Lemma~\ref{lem:p_monotonic_u} above shows that $p_\mathrm{f}(u_\mathrm{in})$ and $r_\mathrm{f}(u_\mathrm{in})$ are both strictly decreasing functions of $u_\mathrm{in}\in(U_2,u_\mathrm{f})$. We now consider the limiting behavior of $(p_\mathrm{f},r_\mathrm{f})(u_\mathrm{in})$ as $u_\mathrm{in}$ approaches the limits $u_\mathrm{in}=U_2$ and $u_\mathrm{in}=u_\mathrm{f}$. In preparation, we consider~\eqref{eq:slow_r_red} in the limit $r\to\infty$, resulting in the vector field
\begin{align}\label{eq:slow_r_red_infty}
\begin{split}
    u_\xi &= p\\
    p_\xi &= u - a+u(v_+(u))^2 = f(u),
    \end{split}
\end{align}
which admits the conserved quantity
\begin{align}\label{eq:energy}
    E(u,p) &= -\frac{1}{2}p^2 + \int_{U_2}^{u} f(\tilde{u}) d\tilde{u}\\
    &= -\frac{1}{2}p^2 + \int_{U_2}^{u} \tilde{u}-a+\tilde{u}(v_+(\tilde{u}))^2 d\tilde{u}.
\end{align}
Note that $E(U_2,0)=0$, and define $p_\mathrm{f,\infty}$ to be the unique positive solution of $E(u_\mathrm{f},p_\mathrm{f,\infty})=0$, corresponding to the intersection of the unstable manifold of the saddle equilibrium $(u,p)=(U_2,0)$ of~\eqref{eq:slow_r_red_infty} with the set $u=u_\mathrm{f}$. We have the following.

\begin{Lemma} \label{lem:pin_limits} The curve $(p,r)=(p_\mathrm{f},r_\mathrm{f})(u_\mathrm{in})$ satisfies the following.
\begin{enumerate}[(i)]
\item \label{lem:pin_limitsi} $\lim_{u_\mathrm{in}\to u_\mathrm{f}}(p_\mathrm{f}(u_\mathrm{in}),r_\mathrm{f}(u_\mathrm{in})) = (p_\mathrm{in}(u_\mathrm{f}),r_0)$
\item \label{lem:pin_limitsii} $\lim_{u_\mathrm{in}\to U_2}(p_\mathrm{f}(u_\mathrm{in}),r_\mathrm{f}(u_\mathrm{in})) = (p_\mathrm{f,\infty},\infty)$
\end{enumerate}
\end{Lemma}
\begin{proof}
The limit~\ref{lem:pin_limitsi} follows directly from the definition of $(p_\mathrm{f}(u_\mathrm{in}),r_\mathrm{f}(u_\mathrm{in}))$ and~\eqref{eq:gamma_in}. 

For~\ref{lem:pin_limitsii}, we aim to compute the limit $\lim_{r_\mathrm{f}\to \infty} p_\mathrm{f}(r_\mathrm{f})$. Note that $p_\mathrm{in}(u_\mathrm{in})\to 0$ as $u_\mathrm{in} \to U_2$, as $\Gamma_\mathrm{in}$ coincides at $u=U_2$ with the invariant line $\ell_2$ corresponding to the fixed point $(u,p)=(U_2,0)$. Hence to determine the behavior of trajectories lying on $\Gamma_\mathrm{in}$ with values of $u\approx U_2$, we can track such trajectories along the invariant line $\ell_2$ to large values of $r$. In particular, for any fixed $R_0\gg1$, there exists $\delta_{R_0}$ such that the forward evolution of $\Gamma_\mathrm{in}$ traces out a two dimensional manifold $\overline{\Gamma}_\mathrm{in}$ which contains the invariant line $\ell_2$ and intersects the plane $r=R_0$ in a curve which can be represented as a graph $p=h_2(u,R_0)$ over $|u-U_2|<\delta_{R_0}$ satisfying $\partial_u h_2(U_2,R_0)=\sqrt{f'(U_2)}+\mathcal{O}(1/R_0)$.

We set $r=1/k$ and arrive at the system
\begin{align}
\begin{split}\label{eq:slow_r_red_k}
 u_\xi&=   p\\
  p_\xi&= -kp+f(u)\\
k_\xi&= -k^2.
\end{split}
\end{align}
The invariant set $k=0$ (corresponding to $r=\infty$) contains the limiting system~\eqref{eq:slow_r_red_infty} for the variables $(u,p)$, whose solutions lie on level sets of the function $E(u,p)$, with the saddle-type equilibrium $(u,p)=(U_2,0)$ satisfying $E(U_2,p)=0$. The two branches of this level set correspond to the one-dimensional stable and unstable manifolds $\mathcal{W}^{\mathrm{u}/\mathrm{s},\infty}(U_2,0)$ of the equilibrium within the invariant set $k=0$, which are tangent to the lines $p=\pm\sqrt{f'(U_2)}(u-U_2)$. The manifolds $\mathcal{W}^{\mathrm{u}/\mathrm{s},\infty}(U_2,0)$ extend to two-dimensional center-stable/center-unstable manifolds $\mathcal{W}^{\mathrm{cu}/\mathrm{cs},\infty}(U_2,0)$ for small $k\ll1$ which intersect along the invariant line $\ell_2=\{u=U_2, p=0\}$. Let $k_0:= 1/R_0\ll1$. Recall the manifold $\overline{\Gamma}_\mathrm{in}$ intersects the set $k=k_0$ in a a curve which can be represented as a graph over $|u-U_2|<\delta_{R_0}$ given by $p=h_2(u,R_0)$ satisfying $\partial_u h_2(U_2,R_0)=\sqrt{f'(U_2)}+\mathcal{O}(k_0)$. Therefore at $k=k_0$, $\overline{\Gamma}_\mathrm{in}$ is aligned close to $\mathcal{W}^{\mathrm{cu},\infty}(U_2,0)$ at the linear level and transversely intersects $\mathcal{W}^{\mathrm{cs},\infty}(U_2,0)$ along the invariant line $\ell_2$. Thus as $k\to0$, $\overline{\Gamma}_\mathrm{in}$ aligns along the branch of the level set $E(U_2,p)=0$ corresponding to $\mathcal{W}^{\mathrm{u},\infty}(U_2,0)$, and contains the invariant line $\ell_2$. From this, we see that as $u_\mathrm{in}\to U_2$, $r_\mathrm{f}(u_\mathrm{in})\to \infty$ and the corresponding solution approaches the manifold $\mathcal{W}^\mathrm{u}_0(U_2,0)$, so that $p_\mathrm{f}(u_\mathrm{in})\to p_\mathrm{f,\infty}$ as claimed.
\end{proof}

Combining this with the results of Lemma~\ref{lem:out_monotonic}, we are able to complete the proof of Proposition~\ref{prop:reduced_tr}.
\begin{proof}[Proof of Proposition~\ref{prop:reduced_tr}]
As described above, the forward evolution of $\Gamma_\mathrm{in}$ reaches the set $u=u_\mathrm{f}$ in a curve $(p_\mathrm{f},r_\mathrm{f})(u_\mathrm{in})$ parameterized by $u_\mathrm{in}\in(U_2,u_\mathrm{f})$. By Lemma~\ref{lem:p_monotonic_u}, $p_\mathrm{f}(u_\mathrm{in})$ and $r_\mathrm{f}(u_\mathrm{in})$ are both strictly decreasing functions of $u_\mathrm{in}\in(U_2,u_\mathrm{f})$, so we can express the curve $(p_\mathrm{f},r_\mathrm{f})(u_\mathrm{in})$ as a graph $p_\mathrm{f}=p_\mathrm{f}(r)$  satisfying $p_\mathrm{f}'(r)>0$ for $r\in (r_0, \infty)$, $\lim_{r\to r_0}p_\mathrm{f}(r)= p_\mathrm{in}(u_\mathrm{f})=\mathcal{O}(r_0)$ and $\lim_{r\to \infty}p_\mathrm{f}(r)= p_\mathrm{f,\infty}$.

Furthermore, by Lemma~\ref{lem:out_monotonic}, within the set $u=u_\mathrm{f}$, $\Gamma_\mathrm{out}$ is given by a graph $p=p_\mathrm{out}(r)$ which satisfies $p_\mathrm{out}'(r)<0$ for $r\in (0,\infty)$ with $\lim_{r\to0}p_\mathrm{out}(r) = \infty$ and $\lim_{r\to\infty}p_\mathrm{out}(r) = a-u_\mathrm{f}$. 

Therefore, in order for the sets $\Gamma_\mathrm{in}$ and $\Gamma_\mathrm{out}$ to intersect transversely at some $r=r_I$, it only remains to check whether $p_\mathrm{f,\infty}>a-u_\mathrm{f}$. Equivalently, a transverse intersection occurs provided $E(u_\mathrm{f}, a-u_\mathrm{f}) > E(u_\mathrm{f},p_\mathrm{f,\infty})=E(U_2,0)=0$, which occurs if
\begin{align}
0<-\frac{1}{2}(a-u_\mathrm{f})^2 + \int_{U_2}^{u_\mathrm{f}} u-a+u(v_+(u))^2 \mathrm{d}u,
\end{align}
or equivalently
\begin{align}
\int_{U_2}^{u_\mathrm{f}} \frac{u - 2mb + \sqrt{u^2 - 4umb}}{2b^2} \mathrm{d}u> \frac{1}{2}(a-U_2)^2.
\end{align}
\end{proof}

\subsection{Proof of Theorems~\ref{thm:spot_existence}--\ref{thm:gap_existence}}\label{sec:existence_proofs}
The results of the preceding sections~\S\ref{sec:farfield}--\ref{sec:m_delta} allow us to complete the construction of radial spot solutions in~\eqref{eq:modifiedKlausmeier}.
\begin{proof}[Proof of Theorem~\ref{thm:spot_existence}]
As in~\S\ref{sec:m_delta}, we track the far-field manifold $\mathcal{W}^\mathrm{s}(\mathcal{B}^\mathrm{far}_\delta)$ into a neighborhood of $\mathcal{M}^+_\delta$, where it aligns along the unstable fibers of orbits which cross the set $\{u=u_\mathrm{f}\}$ within $\mathcal{O}(\delta)$ of the set $\Gamma_\mathrm{out}$. Likewise, the manifold $\mathcal{W}^\mathrm{cu}_\delta(\mathcal{C})$ of solutions bounded at the core can be tracked into a neighborhood of $\mathcal{M}^+_\delta$, where it aligns within $\mathcal{O}(1/s_\mathrm{c}+\delta)$ of the unstable fibers of orbits crossing the set $\{r=r_0\}$ along the curve $\Gamma_\mathrm{in}$. By Proposition~\ref{prop:reduced_tr}, the forward evolution of $\Gamma_\mathrm{in}$ reaches the set $u=u_\mathrm{f}$ in the curve $\Gamma_\mathrm{f}$ which transversely intersects $\Gamma_\mathrm{out}$ at some $r=r_I$, provided
\begin{align}\label{eq:spot_integral_thmproof}
\int_{U_2}^{u_\mathrm{f}} \frac{u - 2mb + \sqrt{u^2 - 4umb}}{2b^2} \mathrm{d}u> \frac{1}{2}(a-U_2)^2.
\end{align}

Therefore, the manifolds $\mathcal{W}^\mathrm{s}(\mathcal{B}^\mathrm{far}_\delta)$ and $\mathcal{W}^\mathrm{cu}_\delta(\mathcal{C})$ intersect transversely, corresponding to a radial spot solution bounded on $r\in[0,\infty)$, with a single sharp interface occurring at $r=r_I+\mathcal{O}(\delta)$. The value $V_\mathrm{c}(a,b,m)$ is determined by the coordinate $v_+(u_0)$ in corresponding fiber of $\mathcal{W}^\mathrm{cu}_\delta(\mathcal{C})$ in the limit $\delta\to0$; see~\eqref{eq:wcu_expansion}. 

 { Finally, the condition~\eqref{eq:spot_condition_exp} can be obtained from a lengthy but straightforward computation by carrying out the integration in~\eqref{eq:spot_integral_thmproof} and using the steady equation satisfied by $U_2$. }
\end{proof}
Regarding gaps, we similarly complete the proof of Theorem~\ref{thm:gap_existence}.
\begin{proof}[Proof of Theorem~\ref{thm:gap_existence}]
The argument is similar to that of Theorem~\ref{thm:spot_existence}. We briefly outline the differences which result in the opposite condition~\eqref{eq:gap_condition}.

The geometry of the construction is quite similar, except opposite in that the far field manifold consists of solutions asymptotic to the equilibrium $(U_2,V_2)$ on the critical manifold $\mathcal{M}^+_\delta$, and the core manifold consisting of solutions bounded as $r\to0$ is constructed from orbits originating near the desert state $(U_0,V_0)$ on $\mathcal{M}^0_\delta$. In this case building the core manifold is somewhat less involved since the flow on $\mathcal{M}^0_\delta$ is linear, and the solutions are given explicitly in terms of modified Bessel functions. However, in the far field one must deal with the nonlinear flow on $\mathcal{M}^+_\delta$.

In this case, the manifold $\mathcal{W}^\mathrm{cu}_\delta(\mathcal{C})$ intersects  $\mathcal{W}^{\mathrm{s}}(\mathcal{M}^+_\delta)$ in the set $\{u=u_\mathrm{f}\}$ transversely along the unstable fibers of orbits lying within $\mathcal{O}(\delta)$ of the set
\begin{align}\label{eq:gamma_in_gap}
\Gamma_\mathrm{in}:= \left\{u=u_\mathrm{f}, p = p_\mathrm{in}(r):= \frac{u_\mathrm{f}-a}{I_0(r)}I_1(r),  r\in [0,\bar{r}] \right\},
\end{align}
where $\bar{r}>0$ is arbitrary. We note that, using asymptotic properties of modified Bessel functions, it can be shown similarly as in~\S\ref{sec:m_delta} that $p_\mathrm{in}$ is an increasing function of $r$ with $\lim_{r\to0}p_\mathrm{in}(r)=0$ and $\lim_{r\to\infty}p_\mathrm{in}(r)=u_\mathrm{f}-a$.

In the far field, analogously to the case of spots, we can construct the two-dimensional far field manifold $\mathcal{B}^\mathrm{far}_0$ within $\mathcal{M}^+_0$ as the set of solutions of the system~\eqref{eq:slow_r_red} which remain bounded as $r\to\infty$. In the limit $r\to\infty$, this system approaches the system~\eqref{eq:slow_r_red_infty}, the solutions of which are given by level sets of the conserved quantity~\eqref{eq:energy}. Let $p_\mathrm{far}(r_I)$ denote the $p$ coordinate at $r=r_I$ of the solution which is bounded as $r\to\infty$ and satisfies $u(r_I)=u_\mathrm{f}$. Then using similar arguments as in~\S\ref{sec:m_delta}, we see that $p_\mathrm{far}$ is an increasing function of $r_I$ which satisfies $\lim_{r_I\to\infty}p_\mathrm{far}(r_I)=-p_{\mathrm{f}, \infty}$, where $p_{\mathrm{f}, \infty}$ satisfies $E(u_\mathrm{f}, -p_{\mathrm{f}, \infty})=0$.

Thus in order to have an intersection of $\mathcal{W}^\mathrm{cu}_\delta(\mathcal{C})$ and the stable fibers $\mathcal{W}^\mathrm{s}_\delta(\mathcal{B}^\mathrm{far}_\delta)$, and thus a radial gap solution with a single interface at some value of $r=r_I$, a sufficient condition is $E(u_\mathrm{f}, u_\mathrm{f}-a) < E(u_\mathrm{f},-p_\mathrm{f,\infty})=E(U_2,0)=0$, or equivalently,
\begin{align}
\int_{U_2}^{u_\mathrm{f}} \frac{u - 2mb + \sqrt{u^2 - 4umb}}{2b^2} \mathrm{d}u< \frac{1}{2}(a-U_2)^2,
\end{align}
which is precisely the opposite condition as that which guarantees the existence of spots.  {The estimate for $v_\mathrm{g}(r)$ as $r\to0$ is due to the exponential decay along the fast fibers of $\mathcal{W}^\mathrm{u}(\mathcal{M}^0_\delta)$, and the fact that the subspace $\{v=q=0\}$ is invariant under the flow of~\eqref{eq:core} for $\delta>0$.}
\end{proof}

\subsection{Rings, targets, and other radially symmetric solutions}\label{sec:other_radial}
The techniques used in~\S\ref{sec:farfield}--\ref{sec:existence_proofs} to construct spot and gap solution could be used to construct other localized solutions with radial symmetry as follows: The general strategy is the same, in that to construct a solution which is asymptotically constant, and bounded as $r\to0$, as in the case of spots/gaps, we construct a core manifold $\mathcal{W}^\mathrm{u}_\delta(\mathcal{C})$ of states originating near one of the desert or vegetated steady states $(U_0,V_0)$ or $(U_2,V_2)$. We similarly construct a far-field stable manifold $\mathcal{W}^\mathrm{s}_\delta(\mathcal{B}^\mathrm{far}_\delta)$ of solutions bounded as $r\to\infty$ consisting of stable fibers of one of the steady states $(U_0,V_0)$ or $(U_2,V_2)$, which could be the same or different from the state near the core. 

Then, to construct a radially symmetric profile with a desired number of interfaces, the core manifold $\mathcal{W}^\mathrm{u}_\delta(\mathcal{C})$ is tracked along a number of fast jumps alternating between $\mathcal{M}^0_\delta$ and $\mathcal{M}^+_\delta$ as a sequence of radii $r_j, j=1,2,3,\ldots$, in between which $\mathcal{W}^\mathrm{u}_\delta(\mathcal{C})$ follows the slow flow of the corresponding slow manifold $\mathcal{M}^0_\delta$ or $\mathcal{M}^+_\delta$, entering along its fast stable fibers, and exiting aligned along its fast unstable fibers according to the exchange lemma. 

This procedure could be used, in principle, to construct ring or target profiles with any desired (finite) number of interfaces. Some examples obtained numerically are presented in~\S\ref{sec:discussion}.

\section{Spot instabilities}\label{sec:stability}

In this section, we examine the stability of the spot solutions from Theorem~\ref{thm:spot_existence}, and in particular we demonstrate several instabilities exhibited by these solutions when considering 2D perturbations. As we are primarily interested in demonstrating potential instability mechanisms, we do not take a rigorous approach, but rather employ formal asymptotic arguments; however, we emphasize that rigorous results could be obtained using similar methods as in the existence analysis in~\S\ref{sec:existence}.

We linearize~\eqref{eq:modifiedKlausmeier} about a radial spot solution $(u_\mathrm{sp},v_\mathrm{sp})(r)=(u_\mathrm{sp},v_\mathrm{sp})(r;a,b,m,\delta)$ of  Theorem~\ref{thm:spot_existence} using an ansatz of the form 
\begin{align*}
    (U,V) = (u_\mathrm{sp},v_\mathrm{sp})(r)+e^{\lambda t+ i\ell \theta}(u,v)(r)
\end{align*} for $\ell\in\mathbb{Z}$, which results in the eigenvalue problem 
\begin{align}
\begin{split}\label{eq:klaus_stability_problem}
\lambda u&= u_{rr}+\frac{1}{r}u_r-\frac{\ell^2}{r^2}u-\left(1+v_\mathrm{sp}(r)^2\right)u-2u_\mathrm{sp}(r)v_\mathrm{sp}(r)v\\
\lambda v&= \delta^2\left(v_{rr}+\frac{1}{r}v_r-\frac{\ell^2}{r^2}v\right) -mv+v_\mathrm{sp}(r)^2\left(1-bv_\mathrm{sp}(r)\right)u+u_\mathrm{sp}(r)\left(2v_\mathrm{sp}(r)-3bv_\mathrm{sp}(r)^2\right)v.
\end{split}
\end{align}
We consider the essential spectrum associated with spot solutions in~\S\ref{sec:essential_spectrum}. The point spectrum for wave numbers $|\ell|=\mathcal{O}(1)$ is considered in~\S\ref{sec:point_spectrum}, and is evaluated asymptotically in the limit of spots of large radius in~\S\ref{sec:larger1}, where the sideband stability from nearby planar front solutions (see~\S\ref{sec:fronts_stability}) is recovered in the limit $r_I\gg1$. Finally, the point spectrum for large wavenumbers $|\ell|\gg1$ is considered in~\S\ref{sec:largel}.

\subsection{Essential Spectrum}\label{sec:essential_spectrum}
The essential spectrum associated with the radial spot solution $(u_\mathrm{sp},v_\mathrm{sp})(r;a,b,m,\delta)$ is determined by considering the limit $r\to \infty$ in~\eqref{eq:klaus_stability_problem}, and computing the $1$D essential spectrum of the asymptotic rest state $\lim_{r\to \infty}(u_\mathrm{sp},v_\mathrm{sp})(r) = (a,0)$. 

\begin{Lemma}
Consider a spot solution $(u_\mathrm{sp},v_\mathrm{sp})(r;a,b,m,\delta)$ of Theorem~\ref{thm:spot_existence}. Then the essential spectrum $\Sigma_\mathrm{ess}\subset \{\lambda\in \mathbb{C}: \Re \lambda \leq -\beta \}$,   {where $\beta = \min\{1, m\}>0$}.
\end{Lemma}
\begin{proof}
Letting $r\to \infty$, and writing~\eqref{eq:klaus_stability_problem} as a first order system, we obtain
\begin{align}
\begin{split}\label{eq:klaus_stability_essential}
\begin{pmatrix}u_r\\ p_r\\v_r\\ q_r\end{pmatrix}=A_\infty \begin{pmatrix}u\\ p\\v\\ q\end{pmatrix}, \qquad A_\infty = \begin{pmatrix}0&1&0&0\\ 1+\lambda &0&0&0\\0&0&0&\frac{1}{\delta}\\ 0&0&\frac{m+\lambda}{\delta}&0\end{pmatrix}.
\end{split}
\end{align}
The essential spectrum $\Sigma_\mathrm{ess}$ consists of $\lambda \in\mathbb{C}$ for which the matrix $A_\infty $ is not hyperbolic. A short computation shows that this can only occur in the region $\{\lambda\in \mathbb{C}: \Re \lambda \leq -\beta \}$,   {where $\beta = \min\{1, m\}>0$}.
\end{proof}

\subsection{Point spectrum for $|\ell|=\mathcal{O}(1)$}\label{sec:point_spectrum}
Near the interface $r=r_I$, we change variables to $r = r_I+\delta s$ for $s\in(-\nu |\log \delta|, \nu |\log \delta|)$ for some $\nu\gg1$.  {Due to the exponential convergence of the front of the fast subsystem between the critical manifolds $\mathcal{M}^0_0$ and $\mathcal{M}^+_0$, for $\nu$ chosen sufficiently large, this interval captures the portion of the spot solution which lies outside an $\mathcal{O}(\delta)$-neighborhood of the slow manifolds $\mathcal{M}^0_\delta$ and $\mathcal{M}^+_\delta$. In this region,} the eigenvalue problem~\eqref{eq:klaus_stability_problem} becomes
\begin{align}
\begin{split}\label{eq:klaus_stability_problem_fast}
\delta^2\lambda u&= u_{ss}+\frac{\delta}{r_I+\delta s}u_s-\frac{\delta^2\ell^2}{(r_I+\delta s)^2}u-\delta^2\left(1+v_\mathrm{sp}^2\right)u-2\delta^2u_\mathrm{sp}v_\mathrm{sp}v\\
\lambda v&= v_{ss}+\frac{\delta}{r_I+\delta s}v_s-\frac{\delta^2\ell^2}{(r_I+\delta s)^2}v -mv+v_\mathrm{sp}^2\left(1-bv_\mathrm{sp}\right)u+u_\mathrm{sp}\left(2v_\mathrm{sp}-3bv_\mathrm{sp}^2\right)v.
\end{split}
\end{align}
For wavenumbers $|\ell|=\mathcal{O}(1)$ with respect to $\delta$, we expand solutions of this eigenvalue problem in terms of the reduced fast system
\begin{align}
\begin{split}\label{eq:klaus_stability_problem_fast_reduced}
\lambda v&= v_{ss} -mv+u_\mathrm{f}\left(2v_\mathrm{vd}-3bv_\mathrm{vd}^2\right)v,
\end{split}
\end{align}
which is an eigenvalue problem of Sturm-Liouville type, obtained by linearizing the fast subsystem about the front $\phi_\mathrm{vd}$, and which has a solution given by the derivative $v_\mathrm{vd}'$ when $\lambda=0$. For~\eqref{eq:klaus_stability_problem_fast}, we expand the eigenfunction
\begin{align*}
\begin{pmatrix} u\\v\end{pmatrix}&=\begin{pmatrix} 0\\ v_\mathrm{vd}'\end{pmatrix}+\delta \begin{pmatrix} \bar{u}_1\\ \bar{v}_1\end{pmatrix}+\mathcal{O}(\delta^2)
\end{align*}
 and eigenvalue parameter $\lambda(\ell) =\delta \lambda_1(\ell)+\mathcal{O}(\delta^2)$, as well as the solution
\begin{align*}
\begin{pmatrix} u_\mathrm{sp}\\v_\mathrm{sp}\end{pmatrix}&=\begin{pmatrix} u_\mathrm{f}\\ v_\mathrm{vd}\end{pmatrix}+\delta \begin{pmatrix} u_1\\ v_1\end{pmatrix}+\mathcal{O}(\delta^2).
\end{align*}
Substituting into~\eqref{eq:klaus_stability_problem_fast}, we obtain to leading order
\begin{align}
\begin{split}
0&= (\bar{u}_1)_{ss}+\frac{\delta}{r_I}(\bar{u}_1)_s-\frac{\delta^2\ell^2}{r_I^2}\bar{u}-2\delta u_\mathrm{f}v_\mathrm{vd}v_\mathrm{vd}'\\
 \lambda_1 v_\mathrm{vd}'&=  \mathcal{L}_0\bar{v}_1+\frac{1}{r_I}v_\mathrm{vd}''-\frac{\delta \ell^2}{r_I^2}v_\mathrm{vd}'+ v_\mathrm{vd}^2\left(1-bv_\mathrm{vd}\right)\bar{u}_1+ u_1\left(2v_\mathrm{vd}-3bv_\mathrm{vd}^2\right)v_\mathrm{vd}'+ u_\mathrm{f}\left(2-6bv_\mathrm{vd}^2\right)v_1v_\mathrm{vd}',
\end{split}
\end{align}
where
\begin{align}
\mathcal{L}_0: = \partial_s^2-m+u_\mathrm{f}\left(2v_\mathrm{vd}-3bv_\mathrm{vd}^2\right).
\end{align}
Considering the first equation of~\eqref{eq:klaus_stability_problem_fast}, we write as a first order system
 \begin{align}
 \begin{split}\label{eq:klaus_stability_problem_fast_up}
 (\bar{u}_1)_s &= \delta \bar{p}_1\\
(\bar{p}_1)_s&=-\frac{\delta}{r_I}\bar{p}_1+2 u_\mathrm{f}v_\mathrm{vd}v_\mathrm{vd}',
\end{split}
\end{align}
so that $\bar{u}_1$ is constant to leading order. We now expand the existence problem across the fast jump near $r\approx r_I$ as
\begin{align}
\begin{split}\label{eq:klaus_existence_exp}
0&= (u_1)_{ss}+\frac{\delta}{r_I+\delta s}(u_1)_s+\delta(a-u-uv^2)\\
0&= (v_1)_{ss}+\frac{1}{r_I+\delta s}v_\mathrm{vd}'+\frac{\delta}{r_I+\delta s}(v_1)_s - mv_1+u_\mathrm{f}\left(2v_\mathrm{vd}-3bv_\mathrm{vd}^2\right)v_1+ u_1v_\mathrm{vd}^2(1-bv_\mathrm{vd})
\end{split}
\end{align}
and differentiate the second equation with respect to $s$ to obtain to leading order
\begin{align}
0&= \mathcal{L}_0 (v_1)_s +\frac{1}{r_I}v_\mathrm{vd}''+u_\mathrm{f}\left(2-6bv_\mathrm{vd}\right)v_1v_\mathrm{vd}'+ (u_1)_sv_\mathrm{vd}^2(1-bv_\mathrm{vd})+u_1\left(2v_\mathrm{vd}-3bv_\mathrm{vd}^2\right)v_\mathrm{vd}'.
\end{align}
Substituting into the second equation of~\eqref{eq:klaus_stability_problem_fast}, we have to leading order
\begin{align*}
 \lambda_1 v_\mathrm{vd}'&=  \mathcal{L}_0\bar{v}_1- \mathcal{L}_0 (v_1)_s+ v_\mathrm{vd}^2\left(1-bv_\mathrm{vd}\right)\bar{u}_1-(u_1)_sv_\mathrm{vd}^2(1-bv_\mathrm{vd}).
\end{align*}
Using the fact that $\mathcal{L}_0$ is self adjoint, we take the inner product of this equation with $v_\mathrm{vd}'$, and obtain the solvability condition
\begin{align*}
 \lambda_1 &=  \left((u_1)_s-\bar{u}_1\right)\frac{ -\int_{-\infty}^\infty v_\mathrm{vd}^2\left(1-bv_\mathrm{vd}\right)v_\mathrm{vd}' \mathrm{d}s}{\int_{-\infty}^\infty v_\mathrm{vd}'^2\mathrm{d}s}.
\end{align*}
Since $v_\mathrm{vd}'$ is strictly negative, the sign of $\lambda_1$ is given by the sign of the prefactor $\left((u_1)_s-\bar{u}_1\right)$. The value of $(u_1)_s$ is easily determined through the expansion of the existence problem since
\begin{align}
u_\mathrm{sp}' = \delta (u_1)_s+\mathcal{O}(\delta^2),
\end{align} 
so that $(u_1)_s = \frac{a-u_\mathrm{f}}{K_0(r_I)}K_1(r_I)$ corresponding to the leading order $p$-value across the fast jump at $r=r_I$. From this we obtain the solvability condition
\begin{align}\label{eq:lambda1_solvability}
 \lambda_1 &=  \left(\frac{(a-u_\mathrm{f})}{K_0(r_I)}K_1(r_I)-\bar{u}_1\right)\frac{ -\int_{-\infty}^\infty v_\mathrm{vd}^2\left(1-bv_\mathrm{vd}\right)v_\mathrm{vd}' \mathrm{d}s}{\int_{-\infty}^\infty v_\mathrm{vd}'^2\mathrm{d}s}.
\end{align}

It remains to determine the constant $\bar{u}_1$ in~\eqref{eq:lambda1_solvability}. To determine $\bar{u}_1$, we recall from~\eqref{eq:klaus_stability_problem_fast_up} that $\bar{u}_1$ is constant, whilst $\bar{p}_1$ satisfies to leading order
\begin{align*}
\bar{p}_1^0 &= \bar{p}_1^++\int_{-\infty}^\infty2 u_\mathrm{f}v_\mathrm{vd}v_\mathrm{vd}'\mathrm{d}s\\
&=\bar{p}_1^+-u_\mathrm{f}v_+(u_\mathrm{f})^2,
\end{align*}
where $\bar{p}_1^0, \bar{p}_1^+$ denote the limiting values of $\bar{p}_1$ on either side of the fast jump, when the solution approaches the critical manifolds $\mathcal{M}^0_0, \mathcal{M}^+_0$, respectively. To determine $\bar{u}_1$, we construct bounded eigenfunctions $(\bar{u}^0, \bar{p}^0)(r)$ and $(\bar{u}^+, \bar{p}^+)(r)$ in the slow regions near $\mathcal{M}^0_0, \mathcal{M}^+_0$, respectively, such that across the fast jump at $r=r_I$, we have
\begin{align}
\begin{split}\label{eq:klaus_stability_slowmatch}
\bar{u}^0(r_I) &= \bar{u}^+(r_I)\\
\bar{p}^0(r_I) &= \bar{p}^+(r_I)-u_\mathrm{f}v_+(u_\mathrm{f})^2.
\end{split}
\end{align}
We begin by analyzing the linearized equation in the slow variables on $\mathcal{M}^0_0$. We note that here $v_\mathrm{sp}(r)=0$ to leading order; inspecting~\eqref{eq:klaus_stability_problem} and recalling $\lambda = \delta \lambda_1$, we obtain that $v=0$ to leading order on $\mathcal{M}^0_0$, so that $u$ satisfies the leading order equation
\begin{align}
0 = u_{rr}+\frac{1}{r}u_r-\frac{\ell^2}{r^2}u-u,
\end{align}
which is a modified Bessel's equation. For each $\ell$, this equation admits a unique solution which is bounded as $r\to\infty$, which is the modified Bessel function of the second kind $K_\ell(r)$. We therefore obtain the leading order solution $(\bar{u}^0, \bar{p}^0)(r) = (\alpha K_\ell, \alpha K_\ell')(r)$ in the slow region $r>r_I$.

Near $\mathcal{M}^+_0$, the slow reduced equations are not as straightforward, due to the nonlinear reduced flow on $\mathcal{M}^+_0$. In particular inspecting~\eqref{eq:klaus_stability_problem}, to leading order we have that $v$ satisfies
\begin{align}
v = \frac{v_+(r)^2(1-bv_+(r))}{m-u_+(r)(2v_+(r)-3bv_+(r)^2)}u,
\end{align}
where $v_+(r):=v_+(u_+(r))$ and $u_+(r)$ is the solution in the slow region for $r<r_I$. Substituting into~\eqref{eq:klaus_stability_problem}, we obtain the leading order equation for $u$ in the slow region
\begin{align}
\begin{split}\label{eq:slow_+_ell}
0&= u_{rr}+\frac{1}{r}u_r-\frac{\ell^2}{r^2}u-\left(1+v_+(r)^2\right)u-\frac{2u_+(r)v_+(r)^3(1-bv_+(r))}{m-u_+(r)(2v_+(r)-3bv_+(r)^2)}u\\
&=u_{rr}+\frac{1}{r}u_r-\frac{\ell^2}{r^2}u-u-\frac{v_+(r)^2(m+u_+(r)v_+(r)^2)}{m-u_+(r)(2v_+(r)-3bv_+(r)^2)}u.
\end{split}
\end{align}
 {The solutions of this equation do not appear to have a nice representation in terms of special functions. However, this equation is of the form
\begin{align}
\begin{split}\label{eq:slow_+_ell_f}
0&=u_{rr}+\frac{1}{r}u_r-\frac{\ell^2}{r^2}u-u-f_+(r)u,
\end{split}
\end{align}
where $f_+(r)$ has a well-defined limit as $r\to0$. As $r\to0$, the equation behaves like a Bessel-type equation, and it is possible to show that there is a unique (up to a constant multiple) solution $u=u^*_\ell(r)$ for each $\ell$ which is bounded as $r\to0$. To see this, we rewrite~\eqref{eq:slow_+_ell_f} as
\begin{align}
\begin{split}\label{eq:slow_+_ell_aut}
u_\eta &= d\\
d_\eta&= \ell^2 u+r^2u+r^2f_+(r)u\\
r_\eta&= r,
\end{split}
\end{align}
where $\eta=\log r$. The system~\eqref{eq:slow_+_ell_aut} has a fixed point at the origin which admits a two-dimensional unstable manifold corresponding to a one-dimensional space of solutions of the nonautonomous linear system~\eqref{eq:slow_+_ell_f} which are bounded as $\eta\to-\infty$ ($r\to0$). This space is spanned by a nontrivial solution, which we denote by $u=u^*_\ell(r)$, which will serve as a candidate eigenfunction in the slow region $r<r_I$. We therefore} obtain the leading order solution $(\bar{u}^+, \bar{p}^+)(r) = (\beta u^*_\ell, \beta (u^*_\ell)')(r)$ in the slow region $r<r_I$. Using the conditions~\eqref{eq:klaus_stability_slowmatch}, we obtain
\begin{align}
\begin{split}
\alpha K_\ell(r_I) &= \beta u^*_\ell(r_I)\\
\alpha K_\ell'(r_I) &= \beta (u^*_\ell)'(r_I)-u_\mathrm{f}v_+(u_\mathrm{f})^2,
\end{split}
\end{align}
which we can solve to determine
\begin{align}
\begin{split}
\bar{u}_1&=\alpha K_\ell(r_I) = \frac{u_\mathrm{f}v_+(u_\mathrm{f})^2K_\ell(r_I)u^*_\ell(r_I)}{K_\ell(r_I)(u^*_\ell)'(r_I)-K_\ell'(r_I)u^*_\ell(r_I)},
\end{split}
\end{align}
and so
\begin{align}\label{eq:lambda_1_ell}
 \lambda_1=\lambda_1(\ell) &=  \left(\frac{(a-u_\mathrm{f})K_1(r_I)}{K_0(r_I)}-\frac{u_\mathrm{f}v_+(u_\mathrm{f})^2K_\ell(r_I)u^*_\ell(r_I)}{K_\ell(r_I)(u^*_\ell)'(r_I)-K_\ell'(r_I)u^*_\ell(r_I)}\right)\frac{ -\int_{-\infty}^\infty v_\mathrm{vd}^2\left(1-bv_\mathrm{vd}\right)v_\mathrm{vd}' \mathrm{d}s}{\int_{-\infty}^\infty v_\mathrm{vd}'^2\mathrm{d}s}.
\end{align}
In general, we require information about the solution $u^*_\ell$ to be able to determine the sign of this quantity as a function of $\ell$. This is nontrivial to do in general, as $u^*_\ell$ likely does not have a direct representation in terms of special functions. However, in certain limiting cases we can approximate~\eqref{eq:lambda_1_ell}. For sufficiently large spots $r_I\gg1$, we argue in~\S\ref{sec:larger1} that such spots inherit instabilities from nearby planar front solutions (see~\S\ref{sec:fronts_stability}). We consider the case of large wavenumbers $|\ell|\gg1$ in~\S\ref{sec:largel}.

\subsection{Large spots: recovering the sideband instability}\label{sec:larger1}
In this section, we consider the critical eigenvalue expression~\eqref{eq:lambda_1_ell} in the case of a very large radial spot solution, that is $r_I\gg1$. Near the core, such a solution is approximately constant, while at the interface, the solution resembles a stationary planar front between the desert and vegetated equilibrium states. In the limit $r_I\to\infty$, in the far field the solution approaches the stationary planar front and inherits the (in)stability properties of the front. To see this, in this section we estimate the expression~\eqref{eq:lambda_1_ell} in the asymptotic limit $r_I\to\infty$.

 {
\begin{Remark}
To investigate this limit, one option would be to apply the approach from~\S\ref{sec:fronts_stability} to~\eqref{eq:klaus_stability_problem_fast} under the assumption $r_I\gg1$, as for the stability of traveling fronts. However, since we do not intend to repeat the analysis from~\cite{CDLOR} which results in the expression~\eqref{eq:front_lambda2c_est}, in this section we provide a more direct method by estimating the expression~\eqref{eq:lambda_1_ell} in the asymptotic limit $r_I\to\infty$.
\end{Remark}}


We estimate the expression~\eqref{eq:lambda_1_ell} for finite values of $\ell\in\mathbb{Z}$ as $r_I\to\infty$. The expression is explicit (in terms of special functions) except for the value of $u^*_\ell(r_I)$, where $u^*_\ell(r)$ is the unique bounded solution (up to a constant) of~\eqref{eq:slow_+_ell} as $r\to0$. In the case of large $r_I\gg1$, we can approximate $u^*_\ell(r_I)$ as follows. Note that when $\ell=\pm1$,~\eqref{eq:slow_+_ell} reduces to 
\begin{align}
\begin{split}\label{eq:slow_+_ell_pm1}
0&=u_{rr}+\frac{1}{r}u_r-\frac{1}{r^2}u-u-\frac{v_+(r)^2(m+u_+(r)v_+(r)^2)}{m-u_+(r)(2v_+(r)-3bv_+(r)^2)}u,
\end{split}
\end{align}
which admits a solution bounded at $r=0$, given by the derivative $u_+'(r)$, where $u_+(r)$ is the core solution on $\mathcal{M}^+_0$ which satisfies $u_+(r_I)=u_\mathrm{f}$ and $u_+'(r_I) =\frac{(a-u_\mathrm{f})K_1(r_I)}{K_0(r_I)}$. Using reduction of order, we can find another linearly independent solution of this equation, given by
\begin{align*}
    u_2(r):=u_+'(r)\int_{r_I}^r\frac{\mathrm{d}s}{su_+'(s)^2},
\end{align*}
and the Wronskian of $u_1:=u_+'$ and $u_2$ is given by $W(u_1,u_2)(r)=r^{-1}$. We now assume $r_I\gg1$ and attempt to construct the leading-order bounded solution of~\eqref{eq:slow_+_ell} on the interval $[0,r_I]$ as $r_I\to\infty$. We split this interval into $[0,r_I]=[0,r_I-R]\cup[r_I-R,r_I]$, where $r_I\gg R\gg1$. Since the vector field for the existence problem on $\mathcal{M}^+_0$ for large $r$ is given by~\eqref{eq:slow_r_red_infty}, by taking $R$ sufficiently large, we can arrange for $(u_+,u_+')(r_I-R)=(U_2,0)+\mathcal{O}(e^{-\nu R})$ for some $\nu>0$ fixed independent of $R$. Thus on the interval $[0,r_I-R]$~\eqref{eq:slow_+_ell} is approximately
\begin{align}
\begin{split}
0&=u_{rr}+\frac{1}{r}u_r-\frac{\ell^2}{r^2}u-\kappa_0u,
\end{split}
\end{align}
where
\begin{align}
    \kappa_0 = 1+\frac{v_+(U_2)^2(m+U_2v_+(U_2)^2)}{m-U_2(2v_+(U_2)-3bv_+(U_2)^2)},
\end{align}
which admits a unique bounded solution as $r\to0$ given by the modified Bessel function $I_\ell(\sqrt{\kappa_0}r)$. Note that since $r_I\gg R$, by asymptotic properties of Bessel functions $I_\ell(\sqrt{\kappa_0}(r_I-R))\sim \frac{e^{\sqrt{\kappa_0}(r_I-R)}}{\sqrt{2\pi r_I}}$.

On the interval $[r_I-R,r_I]$, under the assumption $\frac{\ell^2-1}{r^2}=\frac{\ell^2-1}{r_I^2}\left(1+\mathcal{O}\left(\frac{R^2}{r_I^2}\right)\right)\ll1$, we expand the bounded solution of~\eqref{eq:slow_+_ell} as $u^*_\ell =u_1(r)+\frac{\ell^2-1}{r_I^2}\tilde{u}(r)$ where $\tilde{u}$ satisfies
\begin{align}
\begin{split}\label{eq:slow_+_ell_pm1_tilde}
0&=\tilde{u}_{rr}+\frac{1}{r}\tilde{u}_r-\frac{1}{r^2}\tilde{u}-\tilde{u}-\frac{v_+(r)^2(m+u_+(r)v_+(r)^2)}{m-u_+(r)(2v_+(r)-3bv_+(r)^2)}\tilde{u} = u_1
\end{split}
\end{align}
to leading order in $\frac{\ell^2-1}{r_I^2}$. We write the solution of this system using variation of constants as
\begin{align}
    \tilde{u}(r) = C_1u_1(r)+C_2u_2(r)-u_1(r)\int_{r_I}^rsu_1(s)u_2(s)\mathrm{d}s+u_2(r)\int_{r_I-R}^rsu_1(s)^2\mathrm{d}s.
\end{align}
Recalling $u^*_\ell(r) =u_1(r)+\frac{\ell^2-1}{r_I^2}\tilde{u}(r)$, 
in order to construct a bounded solution as $r\to0$, we must have $(u,u')(r)\approx (C_3I_\ell, C_3 I_\ell')(\kappa_0r)$ at $r=r_I-R$ for some constant $C_3$ in order to match with the bounded solution $\mathcal{I}_\ell$ on the interval $[0,r_I-R]$. Using the fact that $u_1(r_I-R)$ decays exponentially in $R$ as $R\to \infty$, while $u_2(r_I-R)$ grows exponentially as $R\to\infty$ and $I_\ell(\sqrt{\kappa_0}(r_I-R)) \sim e^{-\sqrt{\kappa_0}R}I_\ell(\sqrt{\kappa_0}r_I)$ for $1\ll R\ll r_I$, we see that we must choose $C_2\approx 0$ in order to ensure $u(r)$ remains bounded as $r\to0$. Thus we obtain the solution 
\begin{align}\label{eq:u_star_ell}
    u^*_\ell(r)\sim u_1(r)-\frac{(\ell^2-1)}{r_I^2}u_1(r)\int_{r_I}^rsu_1(s)u_2(s)\mathrm{d}s+\frac{(\ell^2-1)}{r_I^2}u_2(r)\int_{r_I-R}^rsu_1(s)^2\mathrm{d}s.
\end{align}

It remains to determine the coefficient~\eqref{eq:lambda_1_ell}. From~\eqref{eq:u_star_ell}, we have that
\begin{align*}
    u^*_\ell(r_I)&\sim u_1(r_I)=u_+'(r_I)\\
    (u^*_\ell)'(r_I)&\sim u_1'(r_I)+\frac{(\ell^2-1)}{r_I^2}u_2'(r_I)\int_{r_I-R}^{r_I}s u_1(s)^2\mathrm{d}s\\
    &= u_+''(r_I)+\frac{(\ell^2-1)}{r_I^3u_+'(r_I)}\int_{r_I-R}^{r_I}s u_+'(s)^2\mathrm{d}s\\
    &\sim u_+''(r_I)+\frac{(\ell^2-1)}{r_I^2u_+'(r_I)}\int_{r_I-R}^{r_I} u_+'(s)^2\mathrm{d}s.
\end{align*}
 From this, we find that 
\begin{align*}
    &\left(K_\ell(r_I)(u^*_\ell)'(r_I)-K_\ell'(r_I)u^*_\ell(r_I)\right)\sim K_\ell(r_I)\left(u_+''(r_I)+\frac{(\ell^2-1)}{r_I^2u_+'(r_I)}\int_{r_I-R}^{r_I} u_+'(s)^2\mathrm{d}s\right)-K_\ell'(r_I)u_+'(r_I)\\
    &=K_\ell(r_I)\left(u_\mathrm{f}-a+u_\mathrm{f}v_+(u_\mathrm{f})^2+\frac{(\ell^2-1)}{r_I^2u_+'(r_I)}\int_{r_I-R}^{r_I} u_+'(s)^2\mathrm{d}s\right) -u_+'(r_I)\left(K_\ell'(r_I)+\frac{K_\ell(r_I)}{r_I}\right)\\
    &=K_\ell(r_I)u_\mathrm{f}v_+(u_\mathrm{f})^2+K_\ell(r_I)\frac{(\ell^2-1)}{r_I^2u_+'(r_I)}\int_{r_I-R}^{r_I} u_+'(s)^2\mathrm{d}s+ \frac{\left(u_\mathrm{f}-a\right)}{K_0(r_I)}\left(K_\ell(r_I)K_0(r_I) +K_1(r_I)\left(K_\ell'(r_I)+\frac{K_\ell(r_I)}{r_I}\right)\right)\\
    &=K_\ell(r_I)u_\mathrm{f}v_+(u_\mathrm{f})^2+K_\ell(r_I)\frac{(\ell^2-1)}{r_I^2u_+'(r_I)}\int_{r_I-R}^{r_I} u_+'(s)^2\mathrm{d}s+ \frac{\left(u_\mathrm{f}-a\right)}{K_0(r_I)}\left( K_1(r_I)K_\ell'(r_I)-K_\ell(r_I)K_1'(r_I)\right)\\
    &\sim K_\ell(r_I)u_\mathrm{f}v_+(u_\mathrm{f})^2+\frac{(\ell^2-1)}{r_I^2}\left(\frac{K_\ell(r_I)}{u_+'(r_I)}\int_{r_I-R}^{r_I} u_+'(s)^2\mathrm{d}s+\frac{\left(a-u_\mathrm{f}\right)\pi}{4r_IK_0(r_I)}e^{-2r_I}\right),
\end{align*}
where we used the recurrence formulas~\cite[\S 10.29(i)]{handbook} and asymptotic relations~\cite[\S 10.40(i)]{handbook} for modified Bessel functions. Therefore, we obtain
\begin{align}\label{eq:lambda_1_ell_asymp}
\begin{split}
\lambda_1(\ell) &\sim \frac{\ell^2-1}{r_I^2u_\mathrm{f}v_+(u_\mathrm{f})^2} \left(\int_{r_I-R}^{r_I}u_+'(s)^2\mathrm{d}s+\frac{(u_\mathrm{f}-a)^2}{2}\right)\frac{ -\int_{-\infty}^\infty v_\mathrm{vd}^2\left(1-bv_\mathrm{vd}\right)v_\mathrm{vd}' \mathrm{d}s}{\int_{-\infty}^\infty v_\mathrm{vd}'^2\mathrm{d}s}\\
&\sim \frac{\ell^2-1}{r_I^2}\frac{1}{u_\mathrm{f}v_+(u_\mathrm{f})^2} \left(\int_{-\infty}^{0}u_{+,\infty}'(s)^2\mathrm{d}s+\frac{(u_\mathrm{f}-a)^2}{2}\right)\frac{ -\int_{-\infty}^\infty v_\mathrm{vd}^2\left(1-bv_\mathrm{vd}\right)v_\mathrm{vd}' \mathrm{d}s}{\int_{-\infty}^\infty v_\mathrm{vd}'^2\mathrm{d}s},
\end{split}
\end{align}
where we again use asymptotic properties of modified Bessel functions, and the fact that on the interval $r\in[r_I-R,r_I]$, the solution $u_+(r)$ is approximately $u_+(r)\sim u_{+, \infty}(r-r_I)$, where $u_{+, \infty}(s)$ is the solution for the reduced flow on $\mathcal{M}^+_0$ which forms part of the singular stationary front solution in the limit $r_I\to\infty$, which satisfies $u_{+, \infty}(0) = u_\mathrm{f}$. Returning to~\eqref{eq:front_lambda2c_est}, and noting that 
\begin{align}
    u_{0,\infty}'(\xi) = (a-u_\mathrm{f})e^{\xi}
\end{align}
in the case of a stationary front $c_\mathrm{vd}=0$, we see that for large $r_I\gg1$, we recover the coefficient~\eqref{eq:front_lambda2c_est} for $\lambda_1(\ell)$ with the prefactor $\frac{\ell^2-1}{r_I^2}$.

\subsection{Point spectrum for $|\ell|\gg1$}\label{sec:largel}
When $|\ell|\gg1$ (but $\mathcal{O}(1)$ with respect to $\delta$), we can obtain asymptotic approximations for $K_\ell$ and $u^*_\ell$ in~\eqref{eq:lambda_1_ell}. For the former, we can employ standard asymptotic results for modified Bessel functions~\cite{handbook}. For the latter, we consider~\eqref{eq:slow_+_ell} in the limit of large $|\ell|$. Considering~\eqref{eq:lambda_1_ell} and rearranging the term involving $K_\ell$ and $u^*_\ell$, we see that 
\begin{align}\label{eq:lambda_1_ell_re}
\lambda_1(\ell) &=  \left(\frac{(a-u_\mathrm{f})K_1(r_I)}{K_0(r_I)}-\frac{u_\mathrm{f}v_+(u_\mathrm{f})^2}{\frac{(u^*_\ell)'(r_I)}{u^*_\ell(r_I)}-\frac{K_\ell'(r_I)}{K_\ell(r_I)}}\right)\frac{ -\int_{-\infty}^\infty v_\mathrm{vd}^2\left(1-bv_\mathrm{vd}\right)v_\mathrm{vd}' \mathrm{d}s}{\int_{-\infty}^\infty v_\mathrm{vd}'^2\mathrm{d}s}.
\end{align}
Defining $w^*_\ell(r)=\frac{r}{|\ell|}\frac{(u^*_\ell)'(r)}{u^*_\ell(r)}$, and using~\eqref{eq:slow_+_ell}, we have that $w=w^*_\ell(r)$ satisfies the equation
\begin{align}
    w_r =\frac{|\ell|}{r}\left(1-w^2\right)+\frac{r}{|\ell|}\left(1+f_+(r)\right),
\end{align}
where
\begin{align}
    f_+(r) = \frac{v_+(r)^2(m+u_+(r)v_+(r)^2)}{m-u_+(r)(2v_+(r)-3bv_+(r)^2)}.
\end{align}
Appending an equation for $r$ and rescaling the spatial coordinate, we have the equivalent autonomous system
\begin{align}
\begin{split}
    w' &=\left(1-w^2\right)+\frac{r^2}{\ell^2}\left(1+f_+(r)\right)\\
    r'&=\frac{r}{|\ell|},
    \end{split}
\end{align}
which, for large $|\ell|\gg1$, admits two invariant manifolds which are defined up to $r=0$, given by $w_\pm(r;\ell)=\pm1+\mathcal{O}(|\ell|^{-1})$. The manifold $w=1$ is attracting while $w=-1$ is repelling. For large $|\ell|$, in order for $u^*_\ell(r)$ to be bounded as $r\to0$, the solution $w^*_\ell(r)$ must lie on the manifold $w_+(r;\ell)$, and hence $w^*_\ell(r)\sim 1$ for large $|\ell|$. Therefore, $\frac{(u^*_\ell)'(r)}{u^*_\ell(r)}\sim \frac{|\ell|}{r}$
as $|\ell|\to\infty$. A similar argument (or using the asymptotic expressions in~\cite[\S10.41]{handbook}) shows that $\frac{K_\ell'(r)}{K_\ell(r)}\sim -\frac{|\ell|}{r}$. Finally, returning to~\eqref{eq:lambda_1_ell_re}, we have that
\begin{align*}
\lambda_1(\ell) &=  \left(\frac{(a-u_\mathrm{f})K_1(r_I)}{K_0(r_I)}-\frac{u_\mathrm{f}v_+(u_\mathrm{f})^2r_I}{2|\ell|}\right)\frac{ -\int_{-\infty}^\infty v_\mathrm{vd}^2\left(1-bv_\mathrm{vd}\right)v_\mathrm{vd}' \mathrm{d}s}{\int_{-\infty}^\infty v_\mathrm{vd}'^2\mathrm{d}s}\\
&\sim \left(\frac{(a-u_\mathrm{f})K_1(r_I)}{K_0(r_I)}\right)\frac{ -\int_{-\infty}^\infty v_\mathrm{vd}^2\left(1-bv_\mathrm{vd}\right)v_\mathrm{vd}' \mathrm{d}s}{\int_{-\infty}^\infty v_\mathrm{vd}'^2\mathrm{d}s}\\
&>0
\end{align*}
for $|\ell|\gg1$ sufficiently large (but $\mathcal{O}(1)$ with respect to $\delta$), so that spots of radius $r_I=\mathcal{O}(1)$ are always unstable. In particular, the spots we have constructed in~\S\ref{sec:existence} are unstable to (suitably) large wave numbers $\ell$. However, for sufficiently small spots ($r_I=o(1)$ with respect to $\delta$), the above argument is no longer valid, and it is not possible to rule out stable spots; see~\S\ref{sec:discussion} for some solutions obtained numerically. However, the regime $r_I=o(1)$ lies outside the scope of the existence analysis in~\S\ref{sec:existence}.

Next, we note that if $\ell = \bar{\ell}\delta^{-1/2}$ for some $0<\bar{\ell}=\mathcal{O}(1)$, a similar analysis as in~\S\ref{sec:point_spectrum} results in the solvability condition
\begin{align}\label{eq:lambda1_solvability_large_ell}
 \lambda_1 &= -\frac{\bar{\ell}^2}{r_I^2}+ \left(\frac{(a-u_\mathrm{f})}{K_0(r_I)}K_1(r_I)\right)\frac{ -\int_{-\infty}^\infty v_\mathrm{vd}^2\left(1-bv_\mathrm{vd}\right)v_\mathrm{vd}' \mathrm{d}s}{\int_{-\infty}^\infty v_\mathrm{vd}'^2\mathrm{d}s}
\end{align}
(in place of~\eqref{eq:lambda_1_ell}) so that spots are always unstable to wavenumbers of the form $\ell = \bar{\ell}\delta^{-1/2}$ where $\bar{\ell}$ is small but $\mathcal{O}(1)$. For larger wavenumbers, the first term in~\eqref{eq:lambda1_solvability_large_ell} dominates so that $\lambda_1(\ell)<0$. Thus a switch from unstable to stable wavenumbers occurs at some critical $\ell = \bar{\ell}\delta^{-1/2}$.

Finally, if $\ell = \frac{\bar{\ell}}{\delta}$ for some $0<\bar{\ell}=\mathcal{O}(1)$, to construct a bounded solution of~\eqref{eq:klaus_stability_problem_fast}, to leading order we must have $u\equiv0$, and the fast equation in~\eqref{eq:klaus_stability_problem_fast} reduces to
\begin{align}
\begin{split}\label{eq:klaus_stability_problem_fast_large_ell}
\lambda v+\frac{\bar{\ell}^2}{r_I^2}v&= v_{ss} -mv+u_\mathrm{f}\left(2v_\mathrm{vd}-3bv_\mathrm{vd}^2\right)v,
\end{split}
\end{align}
which is just the Sturm-Liouville problem~\eqref{eq:klaus_stability_problem_fast_reduced}, but with $\lambda$ shifted by $\frac{\bar{\ell}^2}{r_I^2}$. The problem~\eqref{eq:klaus_stability_problem_fast_reduced} has an eigenvalue at $0$ due to translation invariance of the front $v_\mathrm{vd}$, while any other eigenvalues are bounded away from the imaginary axis. Hence we have that any eigenvalues of~\eqref{eq:klaus_stability_problem_fast_large_ell} lie in the region $\left\{\lambda\in\mathbb{C}:\mathrm{Re} \lambda\leq -\frac{\bar{\ell^2}}{r_I^2}\right\}$. Thus there are no further instabilities of $\lambda(\ell)$ for $|\ell| = \mathcal{O}(1/\delta)$ or larger.

\section{Numerical simulations and discussion}\label{sec:discussion}
The results of~\S\ref{sec:stability} demonstrate that the spot and gap solutions of Theorems~\ref{thm:spot_existence}--\ref{thm:gap_existence} are \emph{unstable} with respect to sufficiently large wave numbers $|\ell|\gg1$ for spots of radius $0<r_I=\mathcal{O}(1)$ with respect to $\delta$. We further showed in~\S\ref{sec:larger1} that spots are unstable to smaller wave numbers in the limit of large radius $r_I\gg1$. The radius can be controlled by choosing parameters closer/further from the hypersurface in $(a,b,m)$ parameter space given by the relation~\eqref{eq:energy_formula}, which simultaneously represents the existence condition for stationary fronts, and the boundary between the existence regions for the spots and gaps. Taking parameters closer to this threshold results in spots/gaps of larger radius, which are therefore better approximated by the corresponding nearby stationary front, with the leading order expression for $\lambda_1(\ell)$ approximated by the asymptotic relation~\eqref{eq:lambda_1_ell_asymp}.

\begin{figure}
\hspace{.1\textwidth}
\begin{subfigure}{.3 \textwidth}
\centering
\includegraphics[width=1\linewidth]{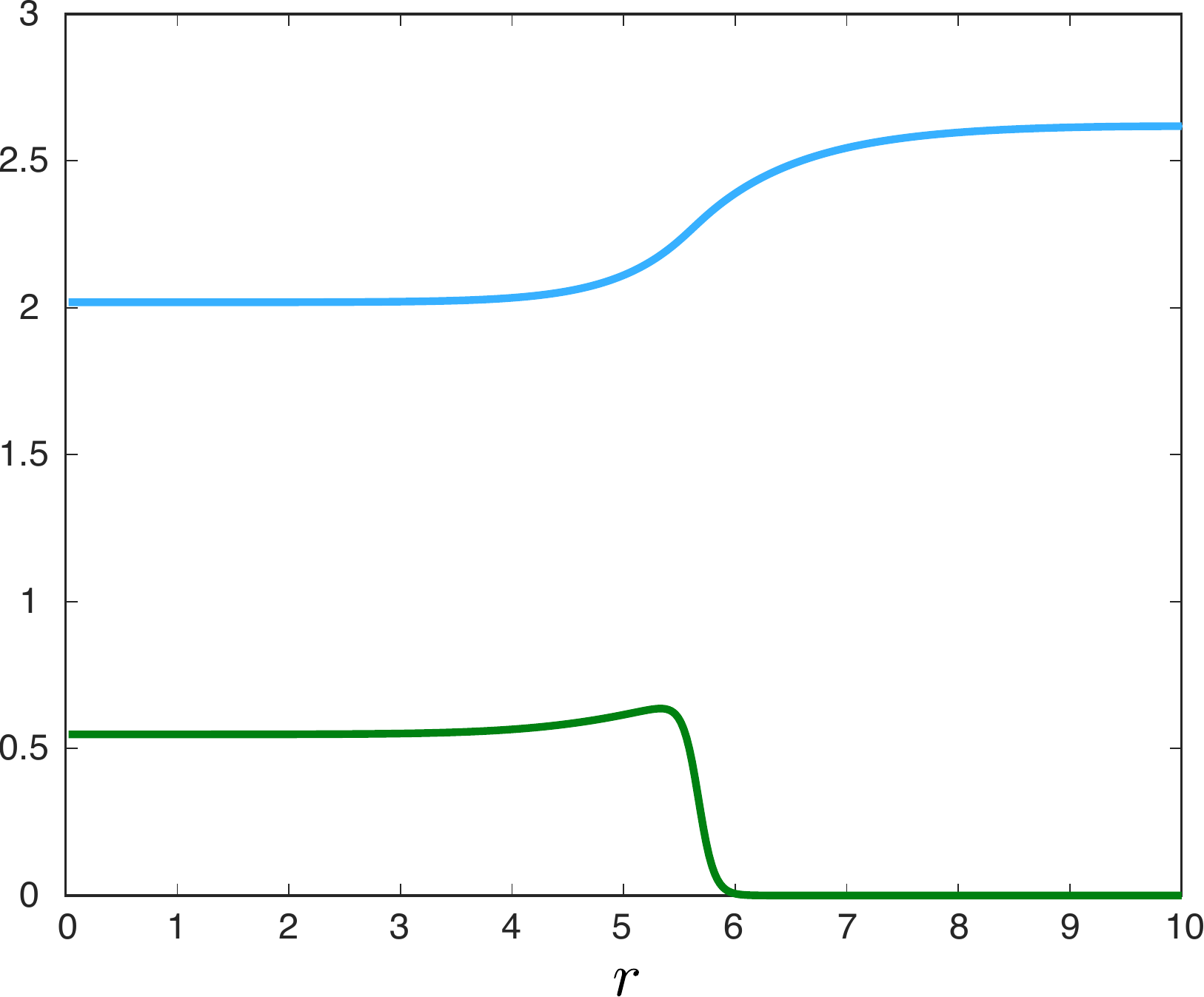}
\end{subfigure}
\hspace{.1\textwidth}
\begin{subfigure}{.3 \textwidth}
\centering
\includegraphics[width=1\linewidth]{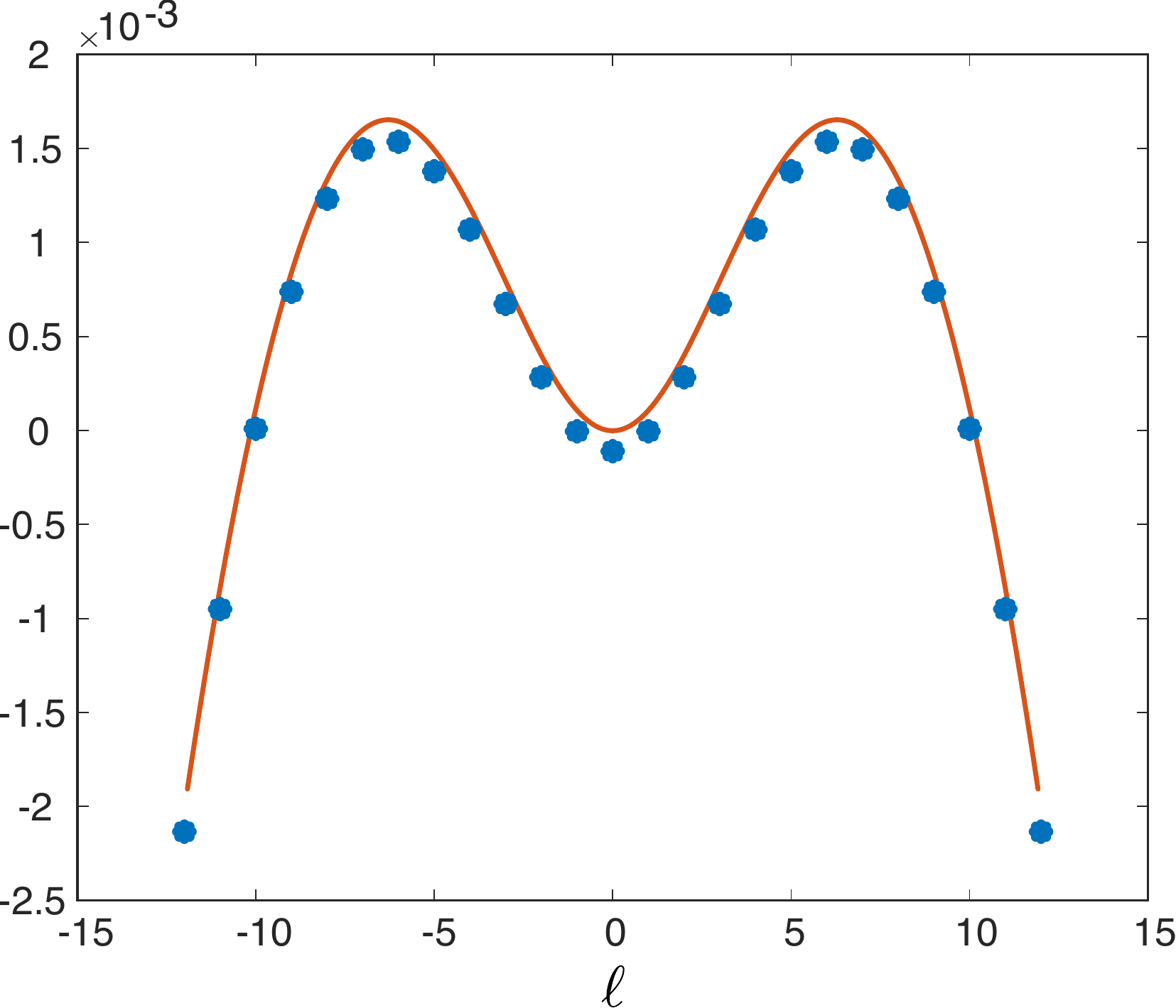}
\end{subfigure}
\hspace{.1\textwidth}\\

\hspace{.1\textwidth}
\begin{subfigure}{.3 \textwidth}
\centering
\includegraphics[width=1\linewidth]{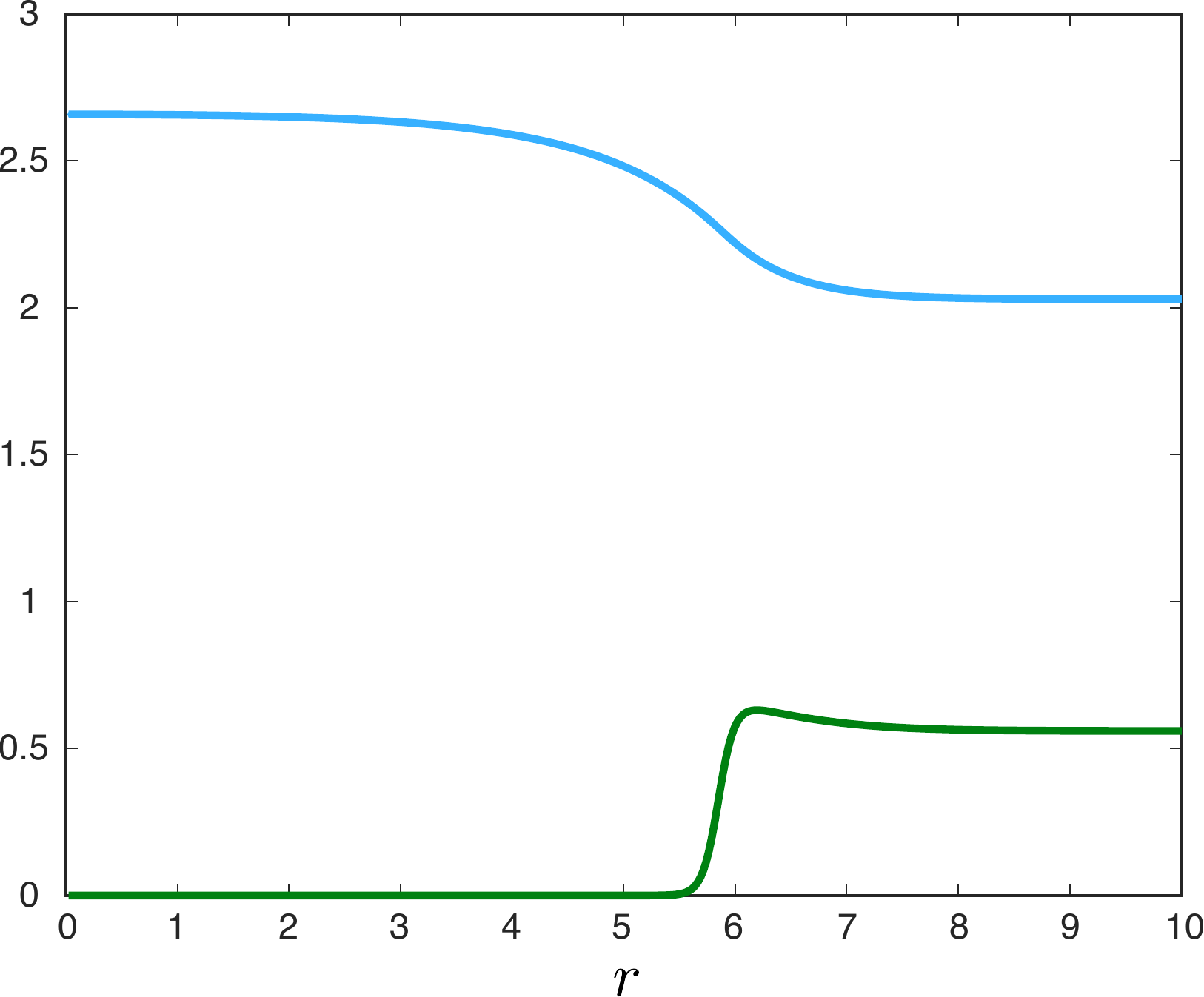}
\end{subfigure}
\hspace{.1\textwidth}
\begin{subfigure}{.3 \textwidth}
\centering
\includegraphics[width=1\linewidth]{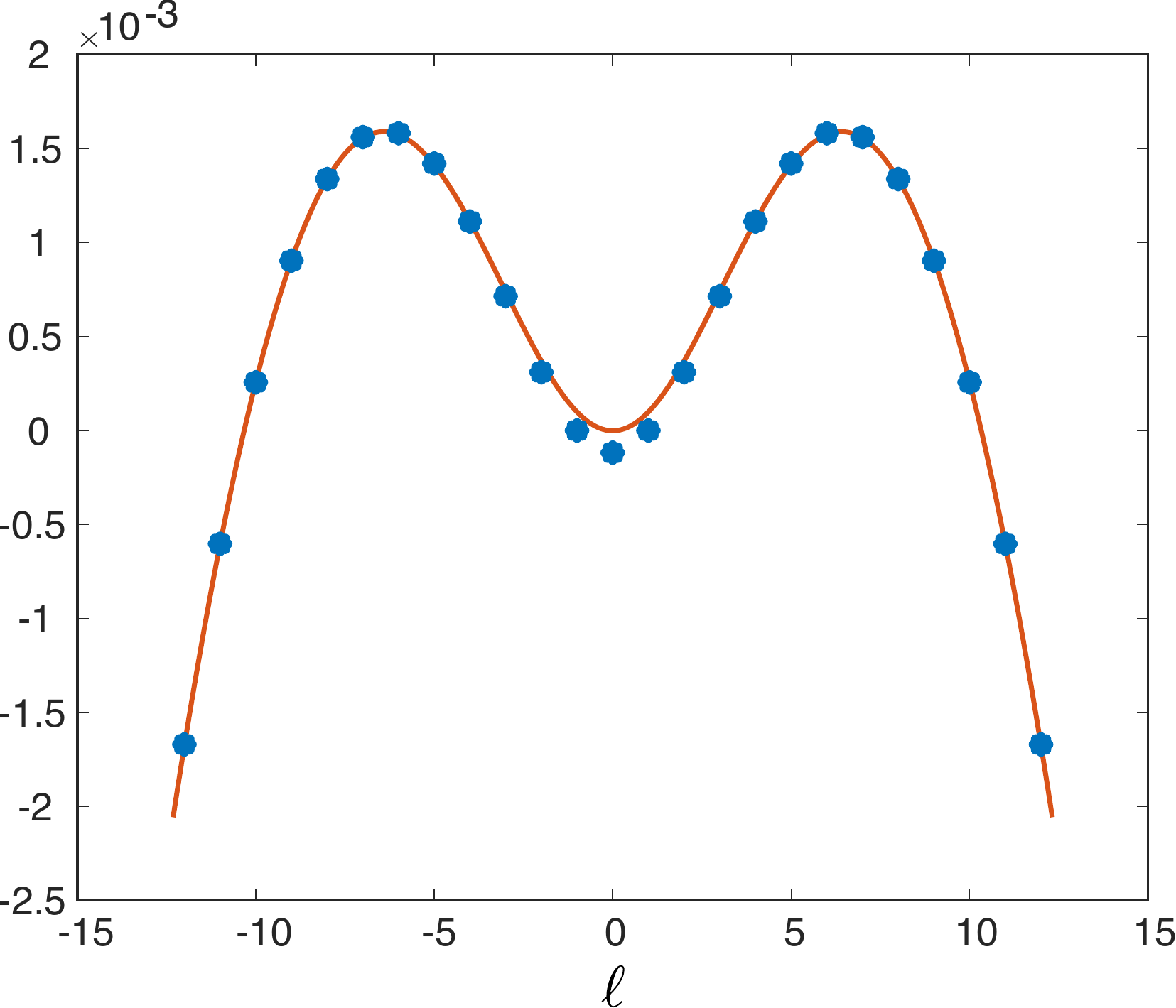}
\end{subfigure}
\hspace{.1\textwidth}
\caption{Stationary radial profiles obtained in~\eqref{eq:klaus_stationary} for $b=1.0, m=0.5, \delta = 0.05$ corresponding to a spot solution (top panels, $a=2.625$) and a gap solution (bottom panels, $a=2.665$). The left panels depict the corresponding radial profiles $(u,v)(r)$ ($u$-profile in blue, $v$-profile in green), while the right panels depict the corresponding critical eigenvalues (blue dots) $\lambda(\ell)$ for $-12\leq\ell\leq 12$. Also plotted (red) is the critical eigenvalue curve for a (slowly) traveling front found for the same parameter values: for $a=2.625$ (top), the corresponding front has speed $c=0.012$, while for $a=2.665$ (bottom), the front has speed $c=-0.013$. The radial profiles were obtained by solving the stationary equation~\eqref{eq:klaus_stationary} using Matlab's fsolve routine, where finite differences were employed for the spatial discretization with Neumann boundary conditions. The eigenvalues $\lambda(\ell)$ were obtained by linearizing~\eqref{eq:modifiedKlausmeier} about the radial profile and using Matlab's eigs routine.  }
\label{fig:unstable_spot}
\end{figure}

In Figure~\ref{fig:unstable_spot}, we demonstrate this for particular parameter values $m=0.5, b=1$, and values of $a$ nearby $a\approx2.6369$ which is the value of $a$ satisfying~\eqref{eq:energy_formula} for $(m,b) = (0.5,1)$, thus representing the location of the stationary front from~\S\ref{sec:fronts_stationary} in the limit $\delta\to0$. Figure~\ref{fig:unstable_spot} depicts a radial profile of a spot solution of radius $r_I\approx 5.66$ (the value of $r_I$ is approximated by the location of the inflection point of the $v$-profile of the solution), as well as the corresponding eigenvalues $\lambda(\ell)$ for $-12\leq \ell\leq 12$. We see good agreement when comparing with the curve obtained by numerically continuing the critical eigenvalue~\eqref{eq:lambda_crit_front} under the rescaled wavenumber $\ell\to\ell/r_I$, for a (slowly) traveling front found for the same parameter values.  (The front has a wave speed close to zero, as we are near the parameter values corresponding to the singular stationary front from~\S\ref{sec:fronts_stationary}.) Figure~\ref{fig:unstable_spot} shows similar agreement for the same computations performed for a radial gap solution of radius $r_I\approx 5.85$. We also point out that these spectral computations show agreement with the analysis in~\S\ref{sec:largel}, in that the spots/gaps are unstable for a range of `large' wave numbers (note here that $1/\sqrt{\delta} \approx 4.47$), and that $\lambda(\ell)$ becomes negative for sufficiently large $|\ell|$.

\begin{figure}
\hspace{.01\textwidth}
\begin{subfigure}{.23 \textwidth}
\centering
\includegraphics[width=1\linewidth]{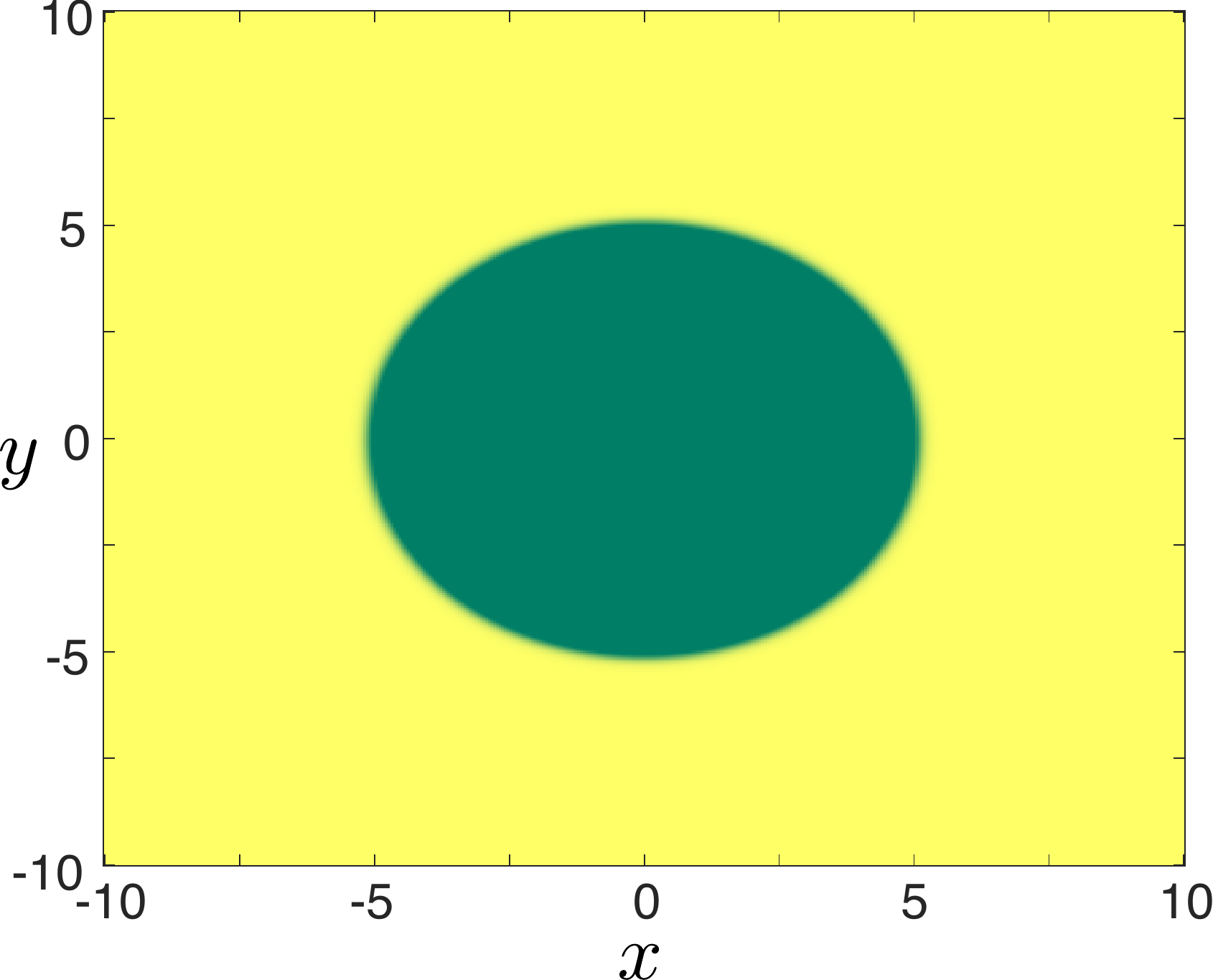}
\end{subfigure}
\begin{subfigure}{.23 \textwidth}
\centering
\includegraphics[width=1\linewidth]{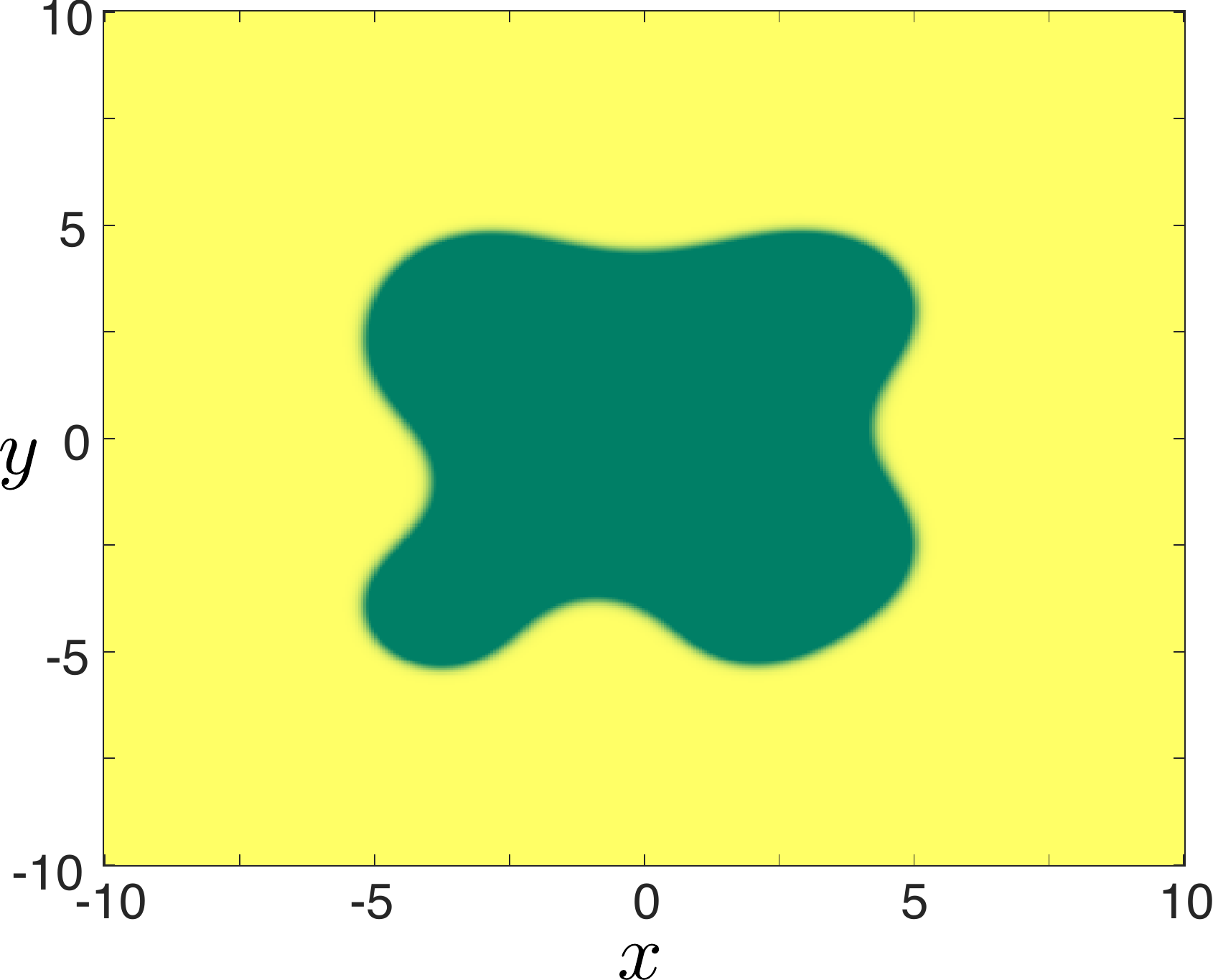}
\end{subfigure}
\begin{subfigure}{.23 \textwidth}
\centering
\includegraphics[width=1\linewidth]{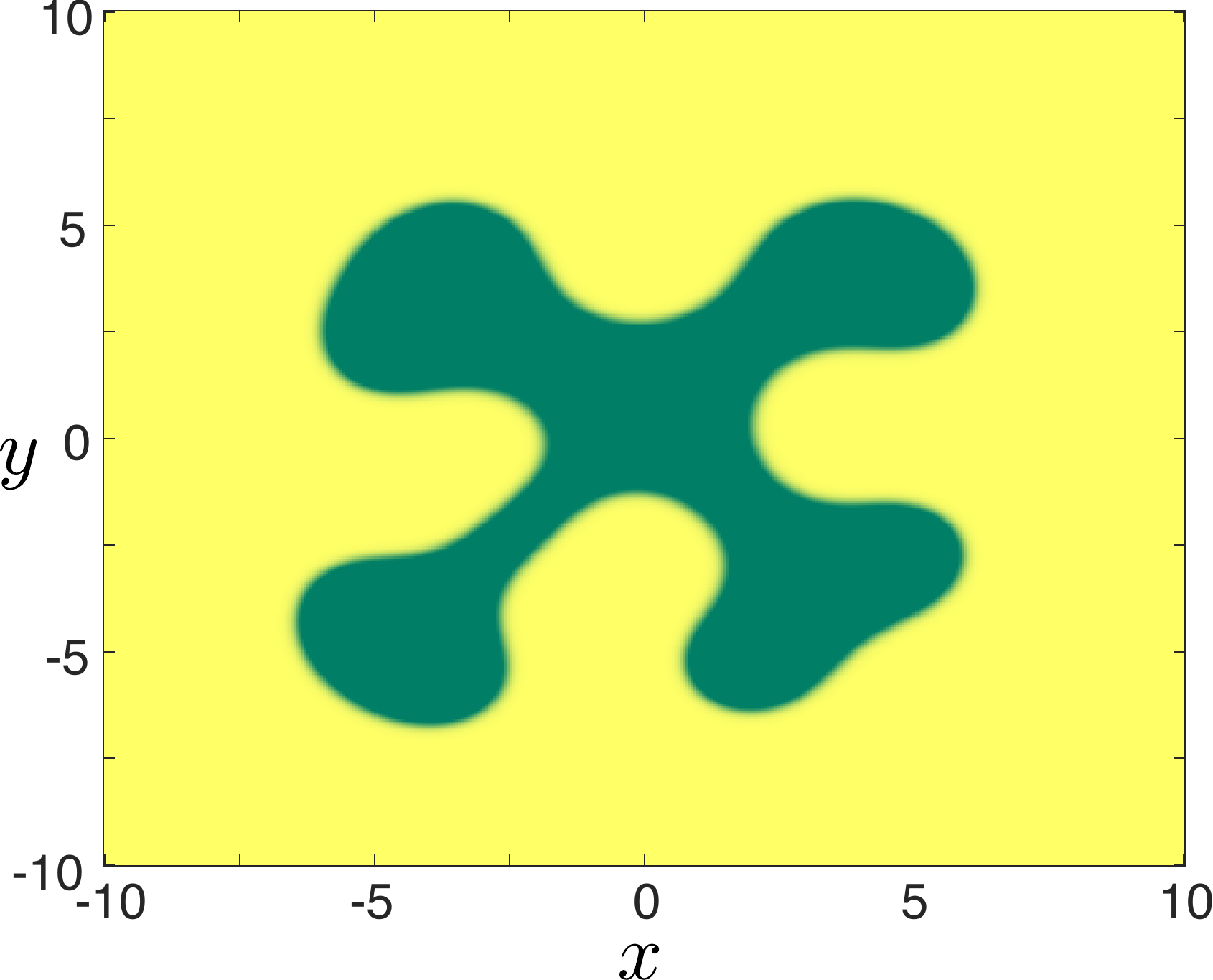}
\end{subfigure}
\begin{subfigure}{.23 \textwidth}
\centering
\includegraphics[width=1\linewidth]{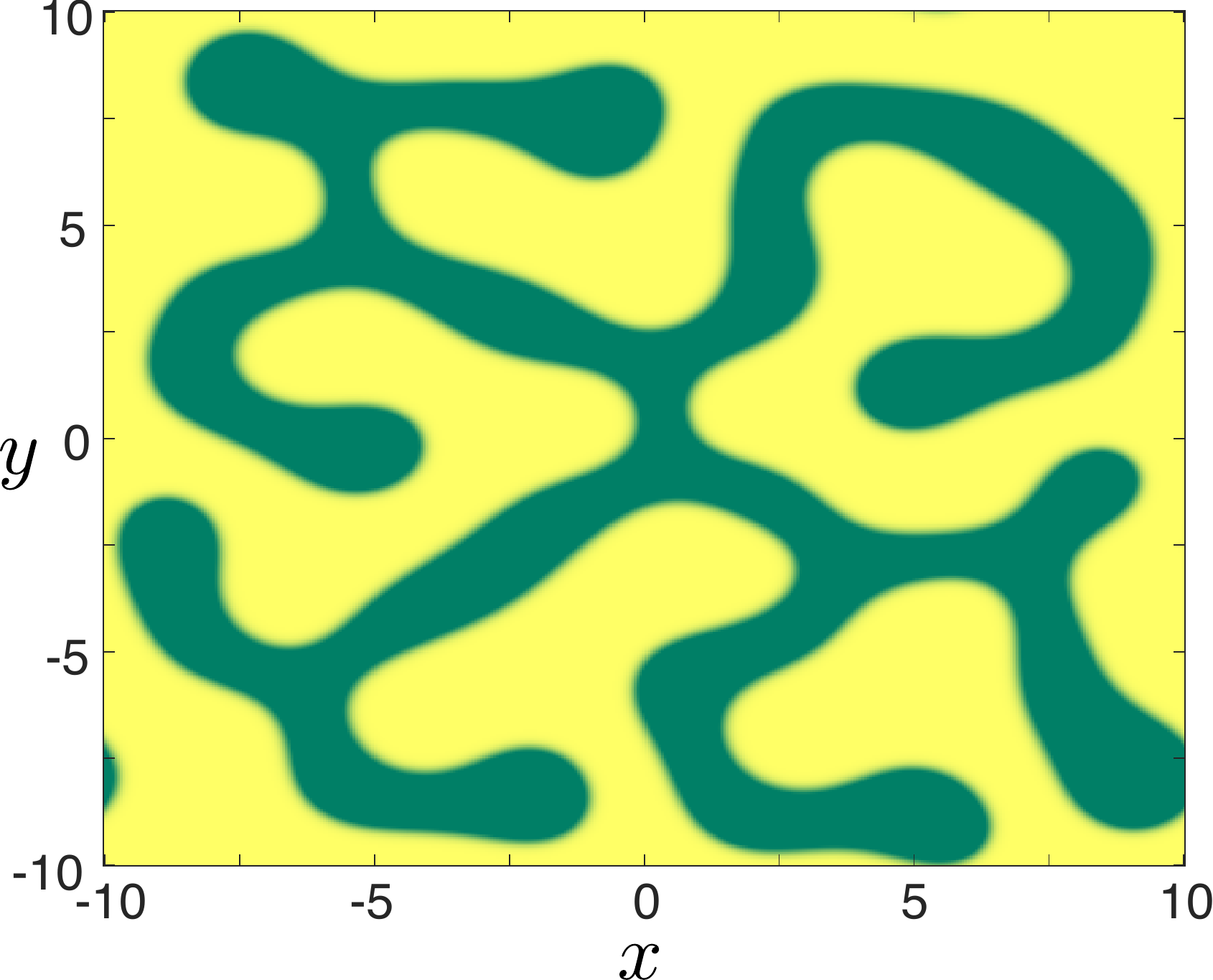}
\end{subfigure}
\hspace{.01\textwidth}\\

\hspace{.01\textwidth}
\begin{subfigure}{.23 \textwidth}
\centering
\includegraphics[width=1\linewidth]{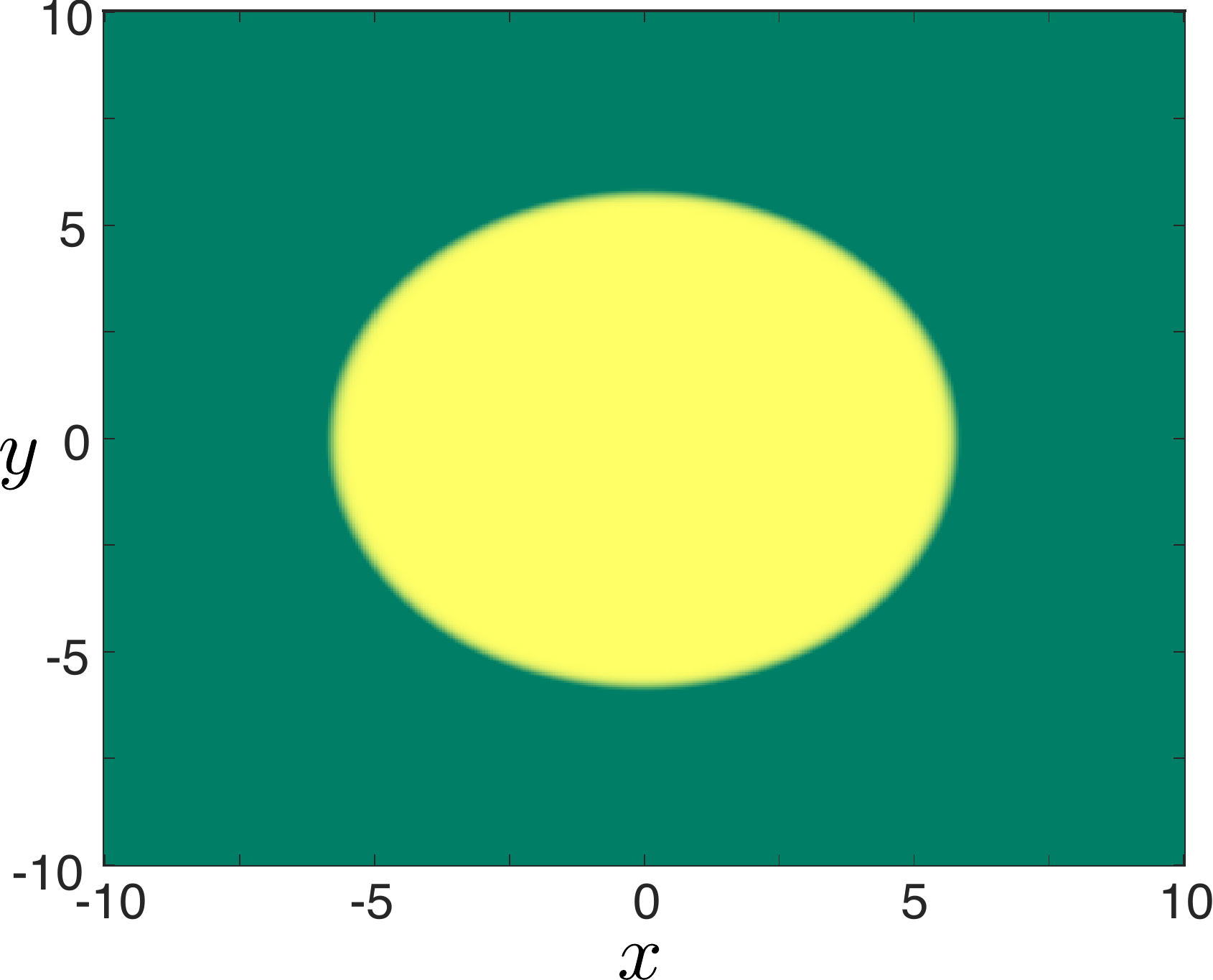}
\end{subfigure}
\begin{subfigure}{.23 \textwidth}
\centering
\includegraphics[width=1\linewidth]{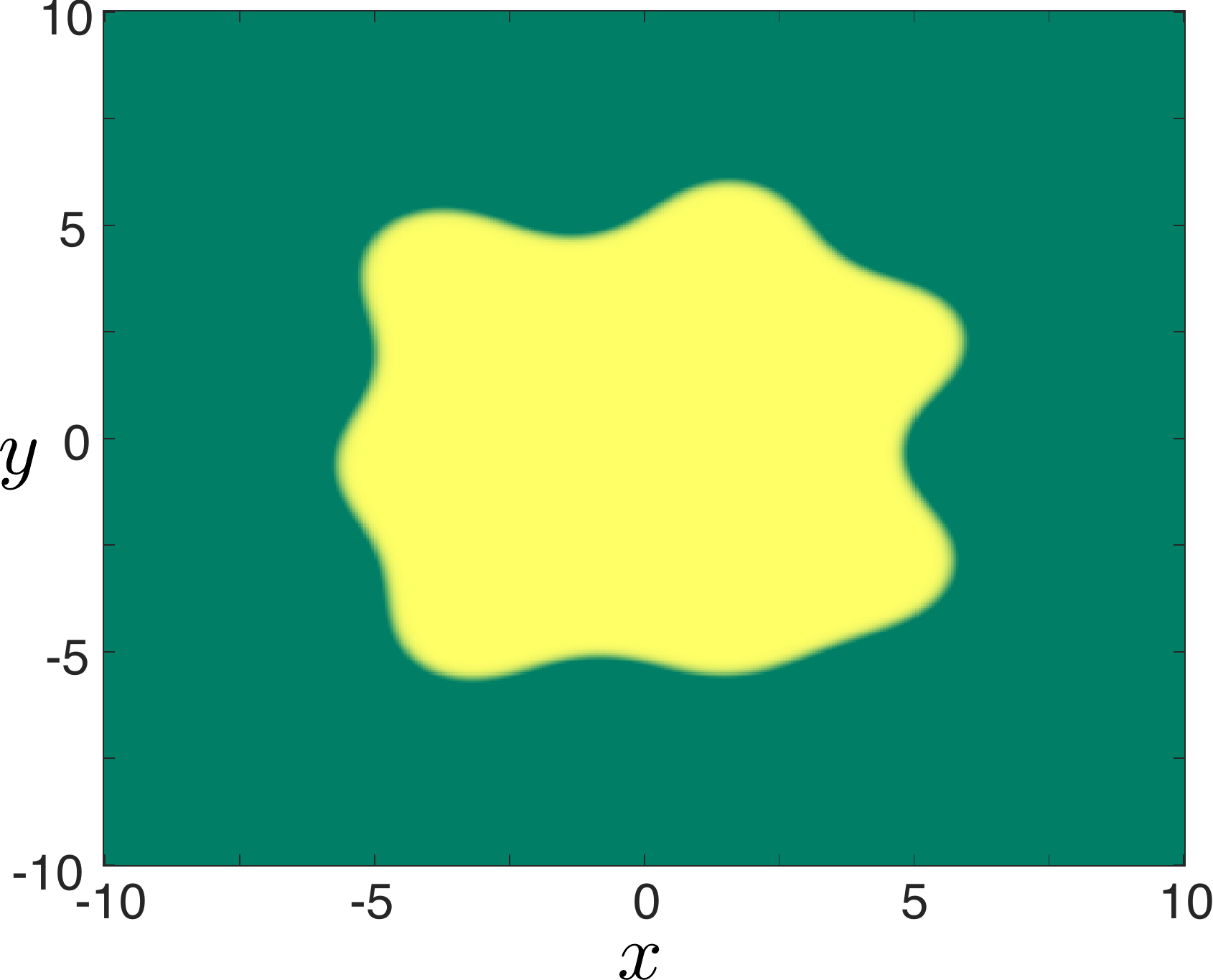}
\end{subfigure}
\begin{subfigure}{.23 \textwidth}
\centering
\includegraphics[width=1\linewidth]{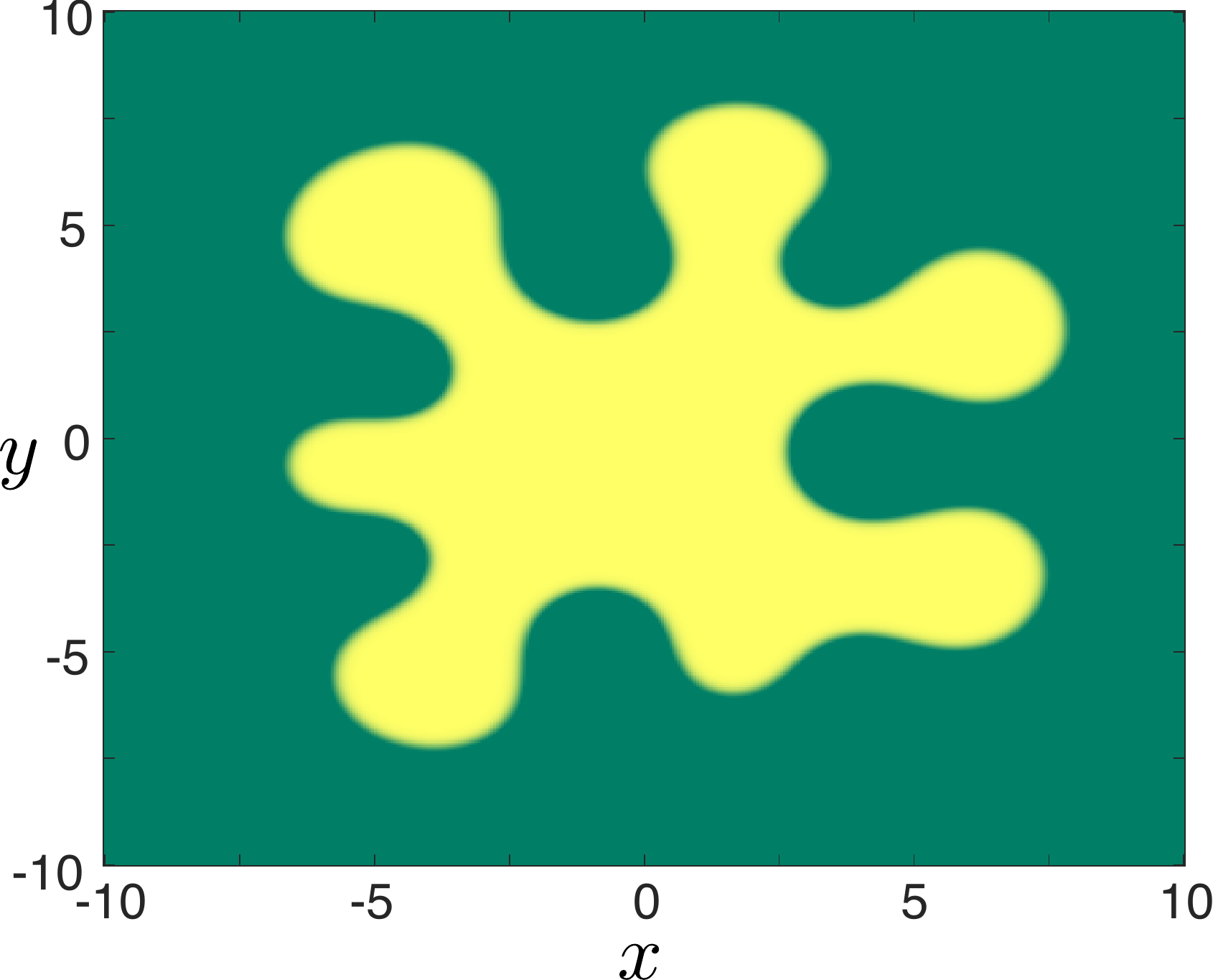}
\end{subfigure}
\begin{subfigure}{.23 \textwidth}
\centering
\includegraphics[width=1\linewidth]{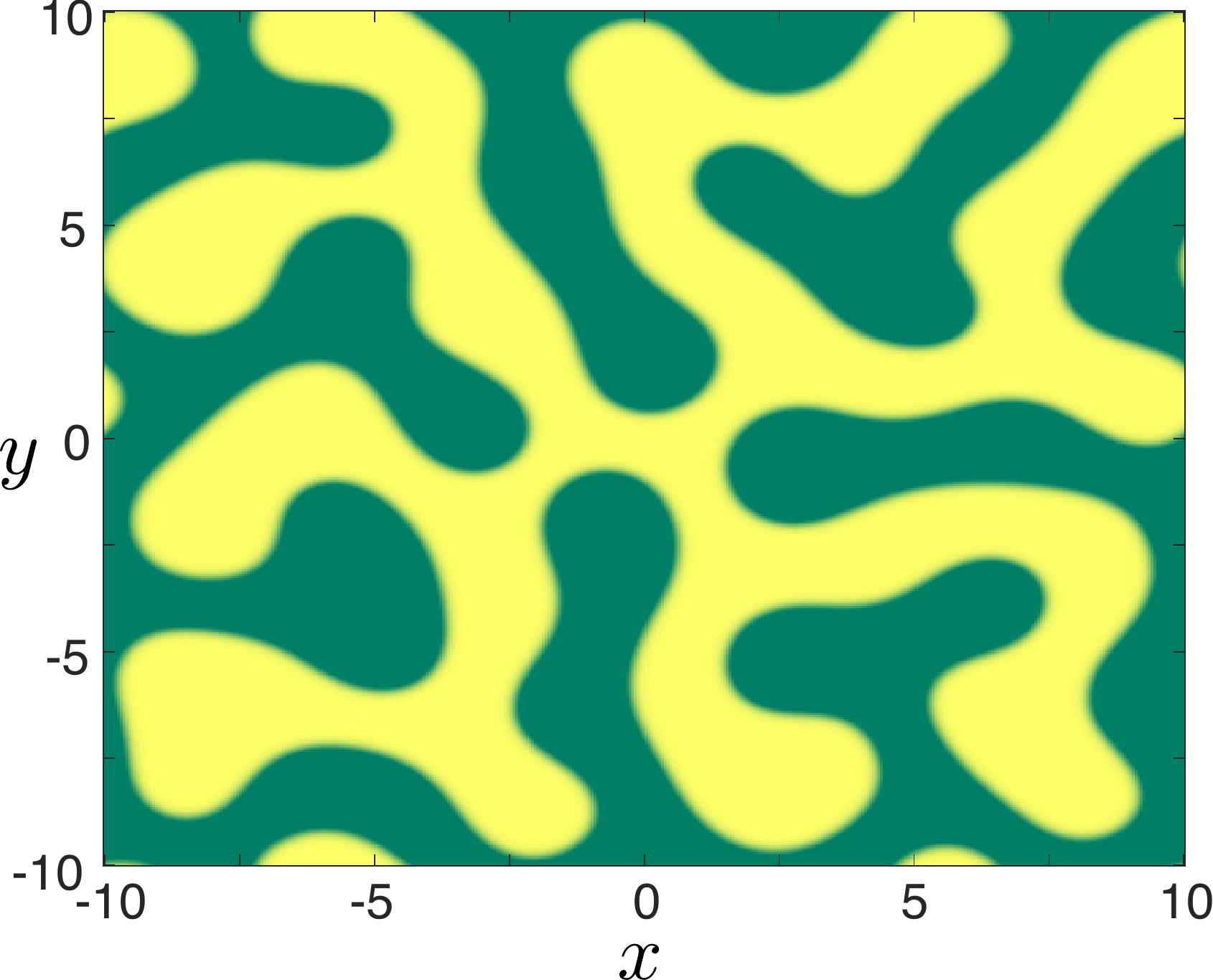}
\end{subfigure}
\hspace{.01\textwidth}
\caption{(Upper panels) Snapshots of direct numerical simulation of a spot solution for $a=2.625, b=1.0, m=0.5, \delta = 0.05$, with initial data given by the solution in Figure~\ref{fig:unstable_spot} (top left panel). The spot develops finger-type patterns along the interface which spread throughout the domain. (Lower panels) Snapshots of direct numerical simulation of gap solution for $a=2.665, b=1.0, m=0.5, \delta = 0.05$, with initial data given by the solution in Figure~\ref{fig:unstable_spot} (bottom left panel), which develops similar finger-type patterns. Simulations were performed in Matlab using finite differences for spatial discretization with periodic boundary conditions, and Matlab's ode15s routine for time integration. }
\label{fig:spot_fingering}
\end{figure}

A natural question concerns the nature of these linear instabilities in the nonlinear dynamics of the spots/gaps. In the large radius limit, we expect such solutions to inherit the sideband instability of the nearby stationary front; in~\cite{CDLOR}, it was demonstrated that this sideband instabilty can lead to the appearance of finger-like patterns along the front interface, which can in turn lead to labyrinthine patterns which expand spatially into the homogeneous states. By performing direct numerical simulations using the unstable spot and gap solutions from Figure~\ref{fig:unstable_spot} as initial data, we see a similar instability manifest along the (circular) interface; see Figure~\ref{fig:spot_fingering} for snapshots of these simulations. We leave a more detailed study of the appearance of such finger-like patterns, and the relation to the corresponding instabilities in the stationary front interface to future work.

\begin{figure}
\hspace{.02\textwidth}
\begin{subfigure}{.3 \textwidth}
\centering
\includegraphics[width=1\linewidth]{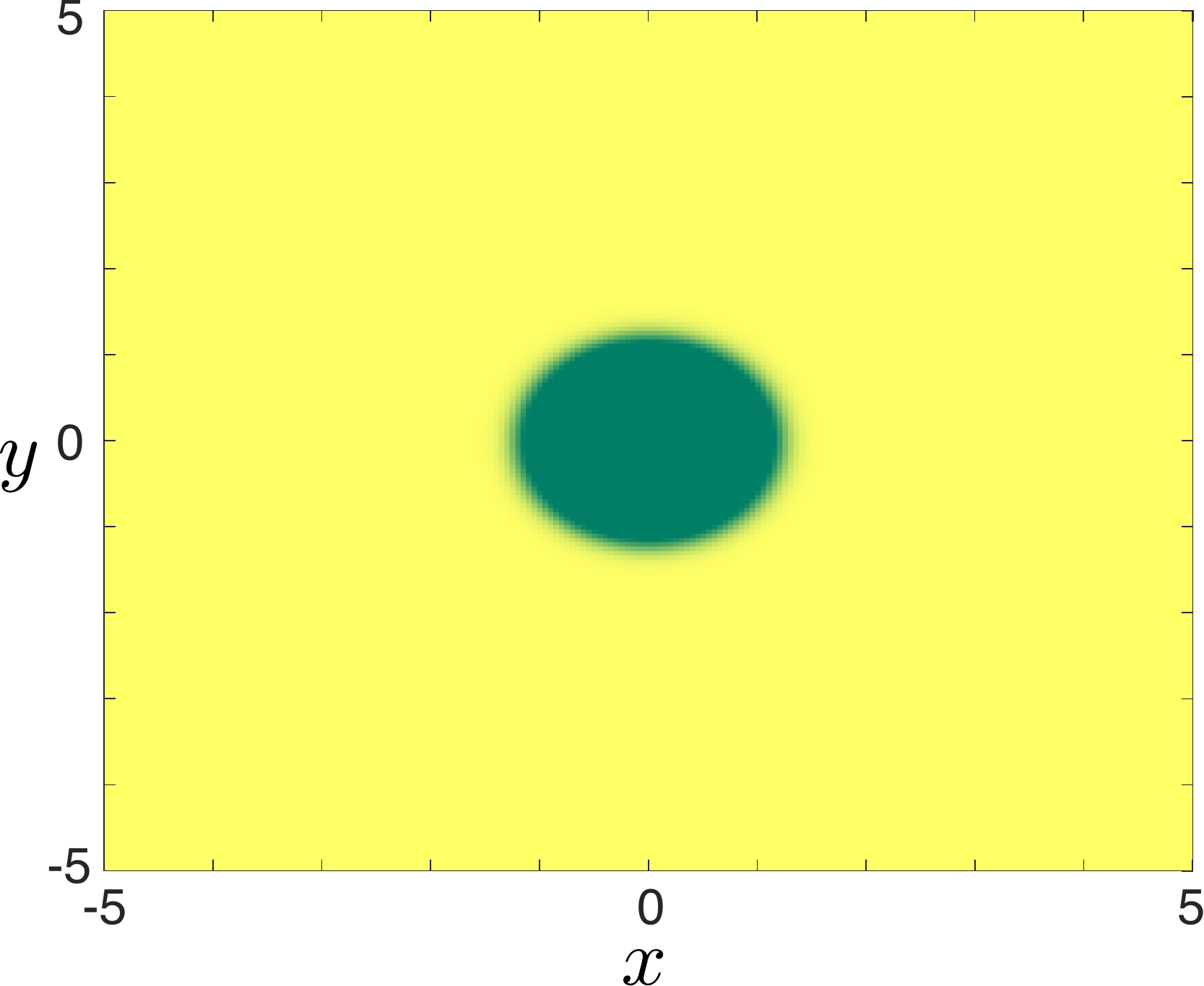}
\end{subfigure}
\hspace{.02\textwidth}
\begin{subfigure}{.3 \textwidth}
\centering
\includegraphics[width=1\linewidth]{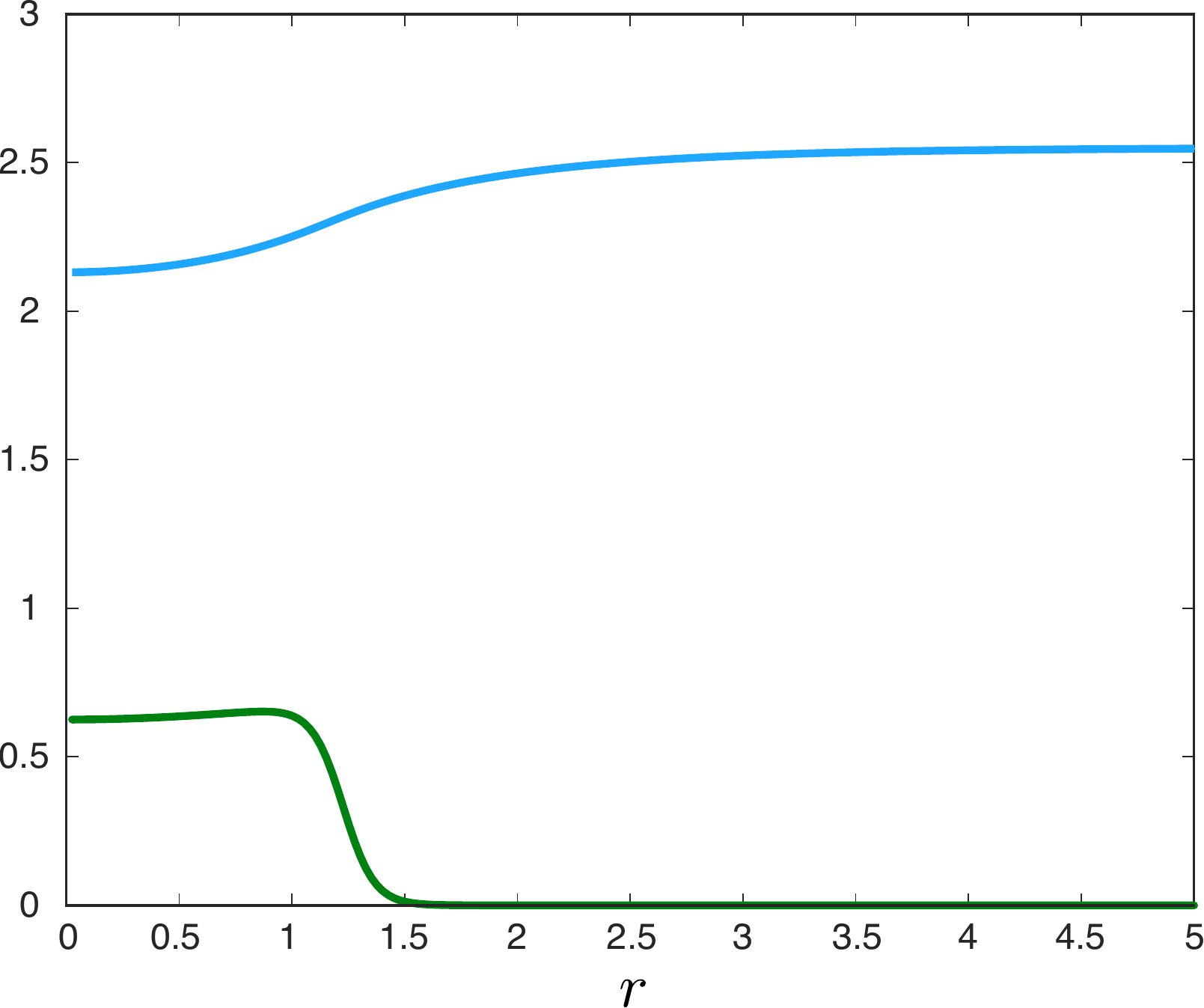}
\end{subfigure}
\hspace{.02\textwidth}
\begin{subfigure}{.3 \textwidth}
\centering
\includegraphics[width=1\linewidth]{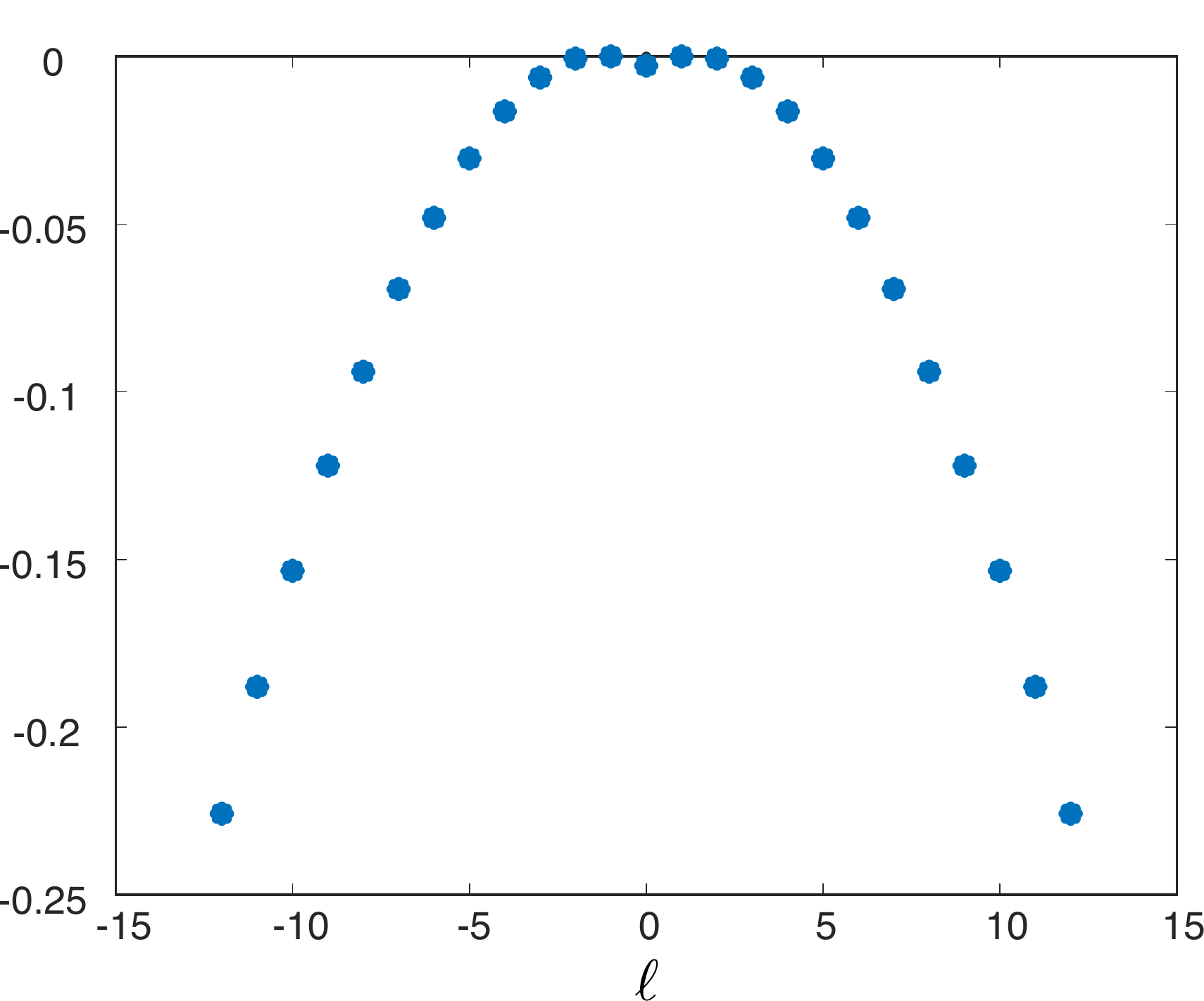}
\end{subfigure}\\

\hspace{.02\textwidth}
\begin{subfigure}{.3 \textwidth}
\centering
\includegraphics[width=1\linewidth]{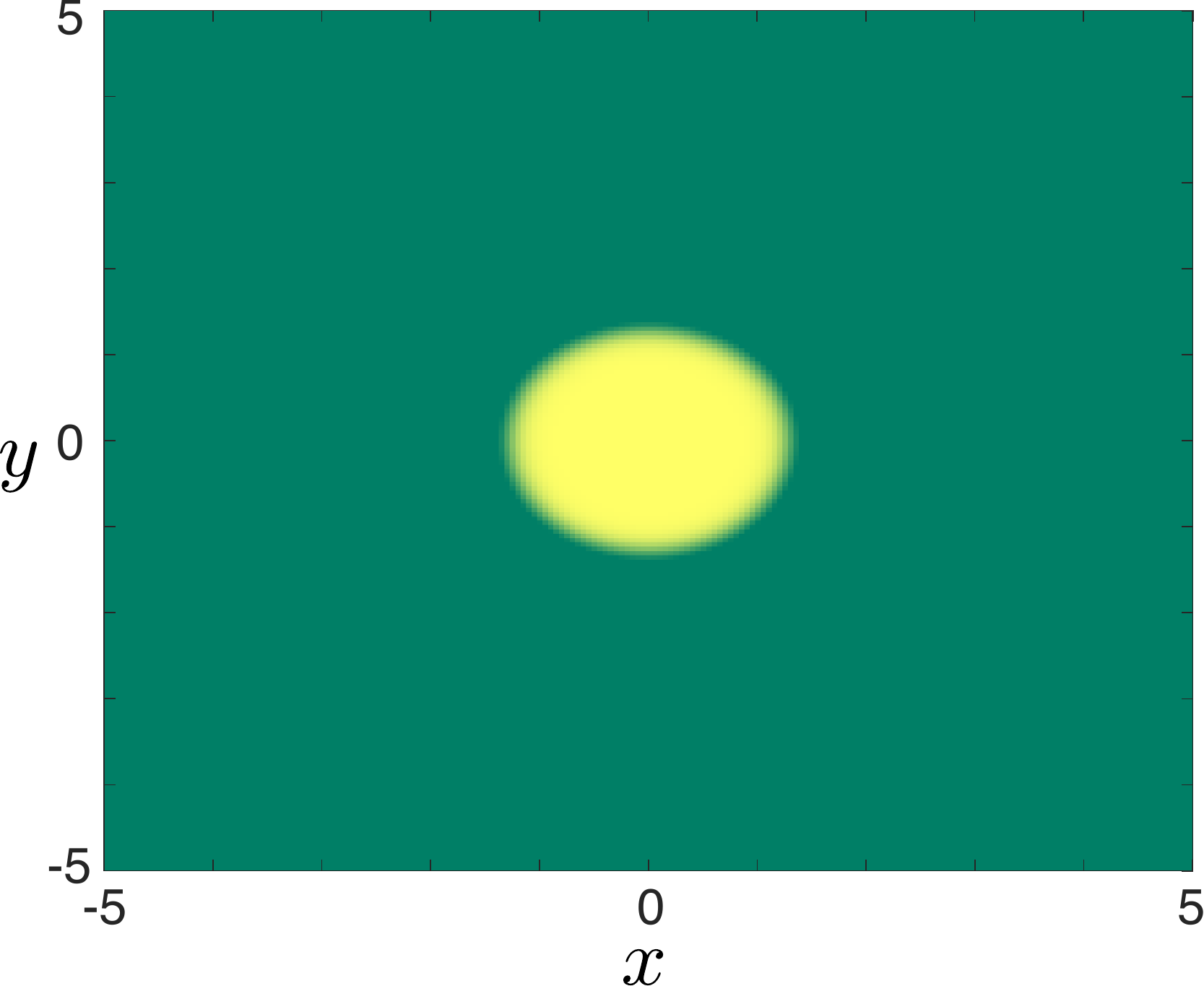}
\end{subfigure}
\hspace{.02\textwidth}
\begin{subfigure}{.3 \textwidth}
\centering
\includegraphics[width=1\linewidth]{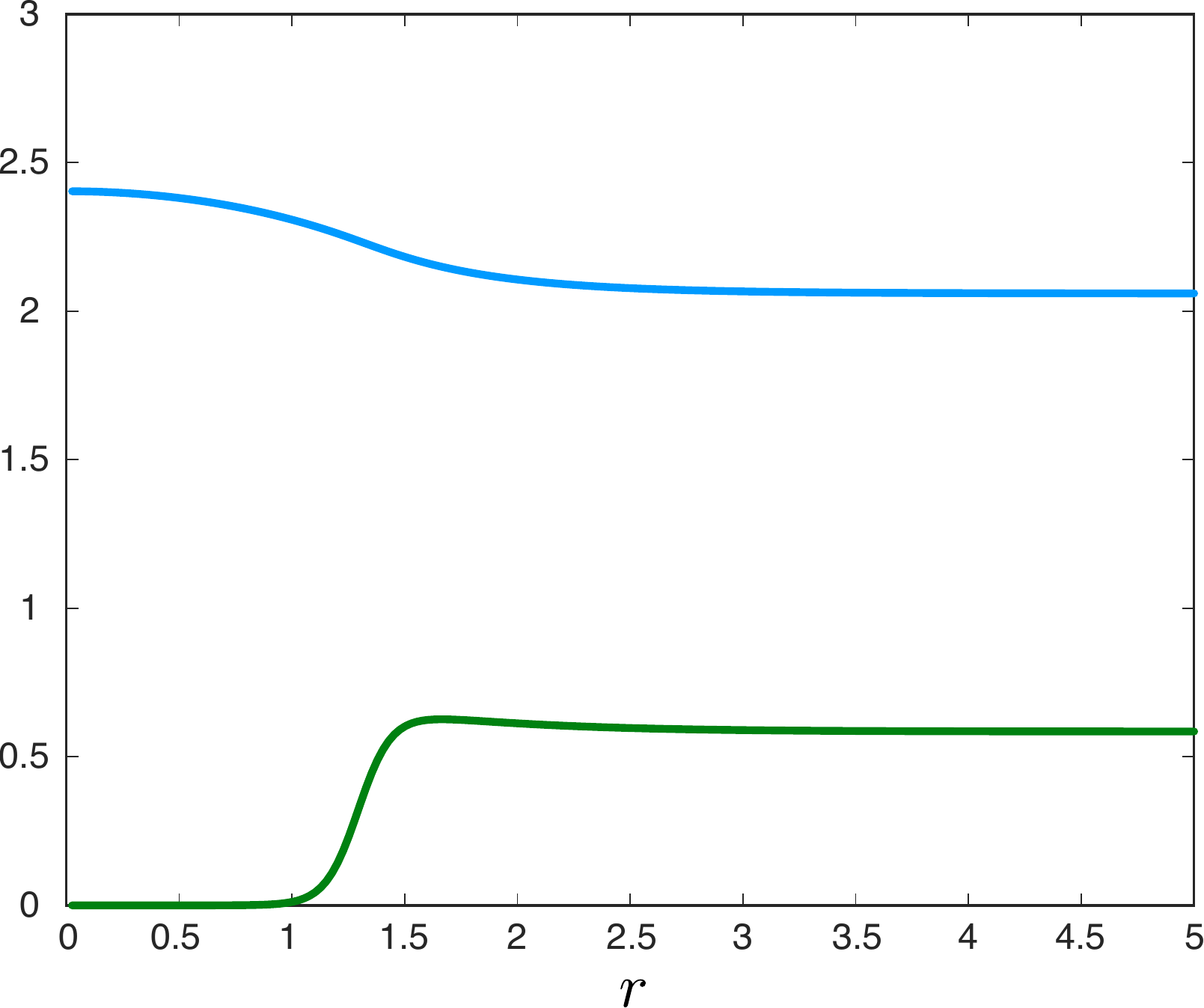}
\end{subfigure}
\hspace{.02\textwidth}
\begin{subfigure}{.3 \textwidth}
\centering
\includegraphics[width=1\linewidth]{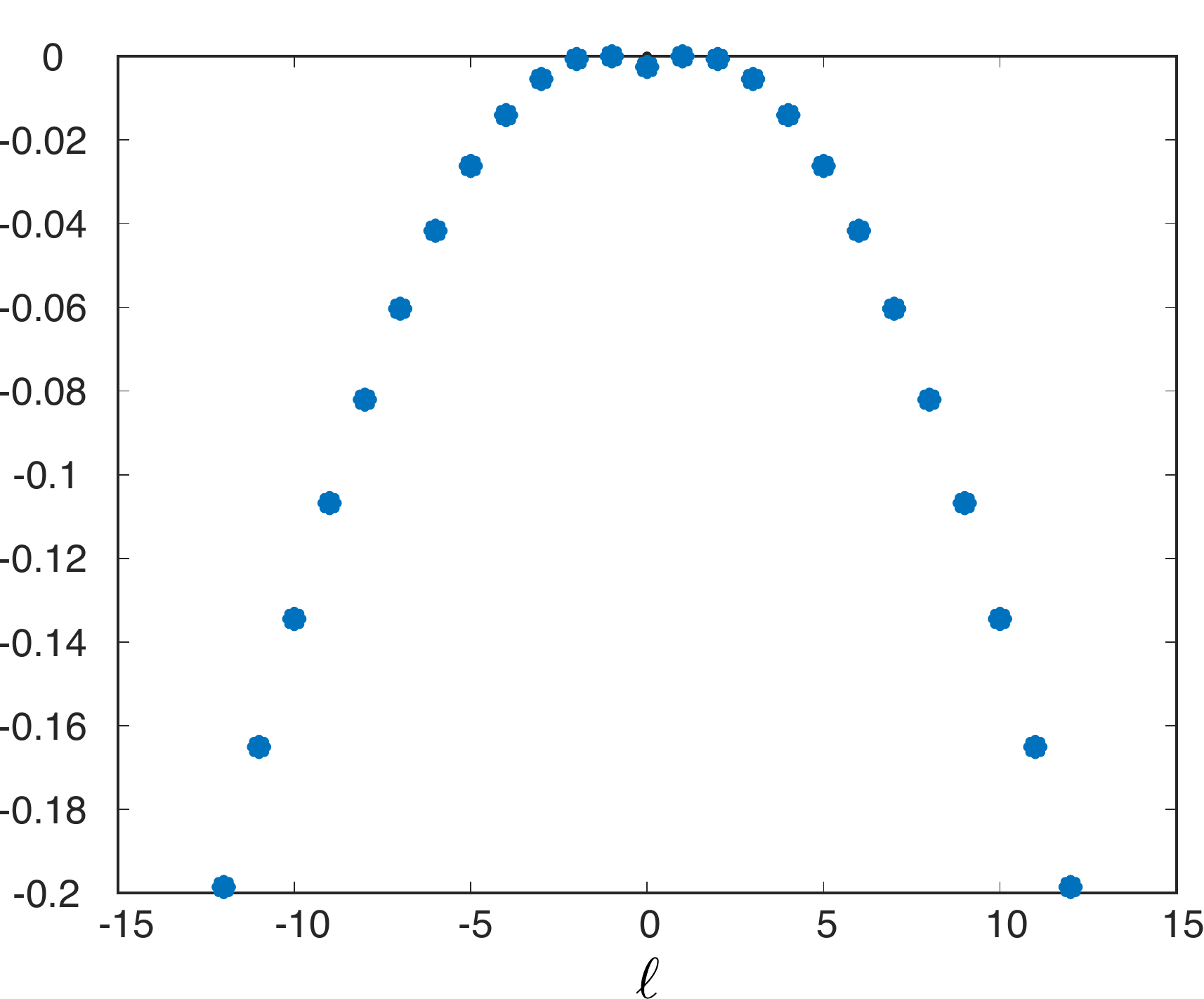}
\end{subfigure}\\

\hspace{.02\textwidth}
\begin{subfigure}{.3 \textwidth}
\centering
\includegraphics[width=1\linewidth]{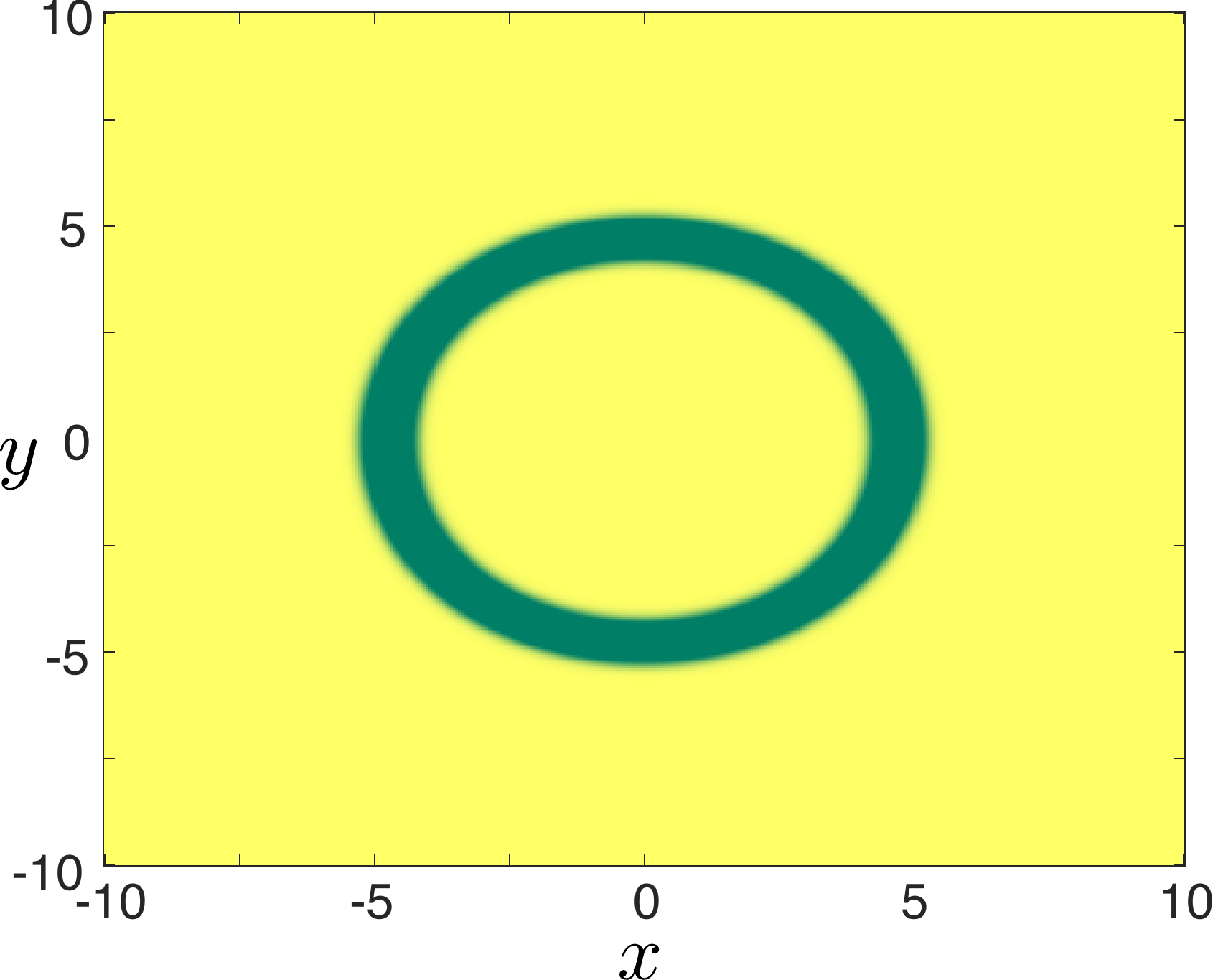}
\end{subfigure}
\hspace{.02\textwidth}
\begin{subfigure}{.3 \textwidth}
\centering
\includegraphics[width=1\linewidth]{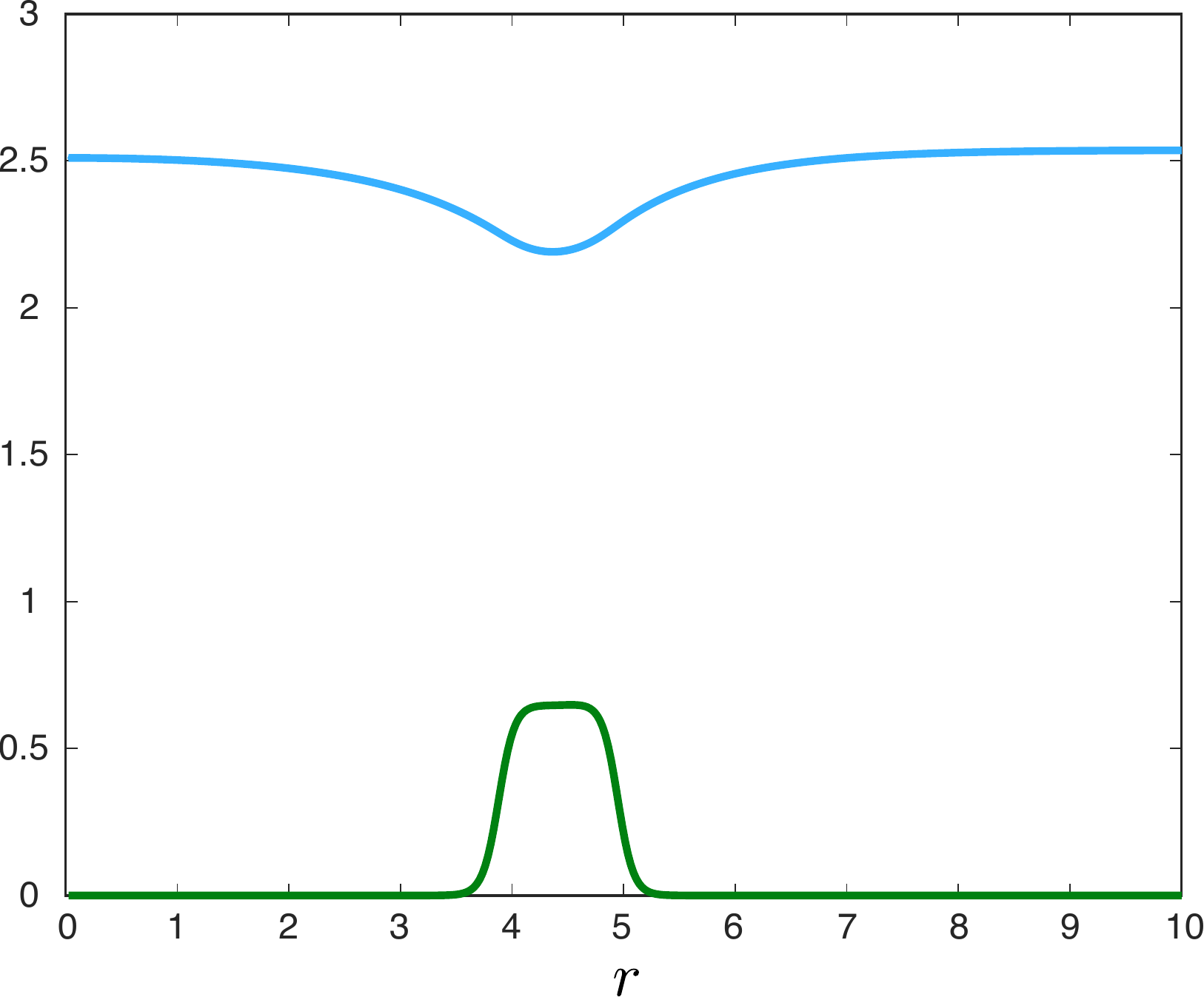}
\end{subfigure}
\hspace{.02\textwidth}
\begin{subfigure}{.3 \textwidth}
\centering
\includegraphics[width=1\linewidth]{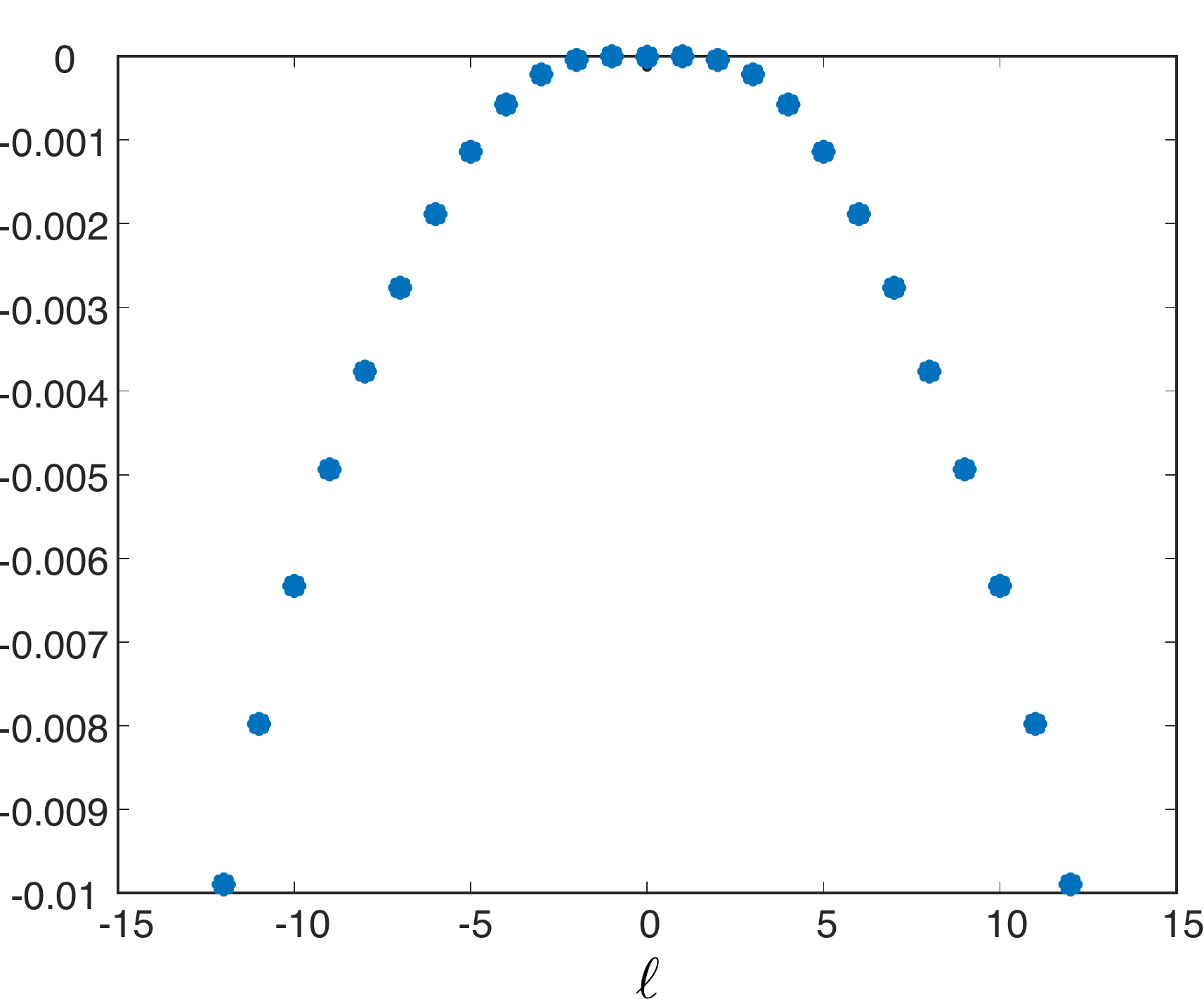}
\end{subfigure}\\

\hspace{.02\textwidth}
\begin{subfigure}{.3 \textwidth}
\centering
\includegraphics[width=1\linewidth]{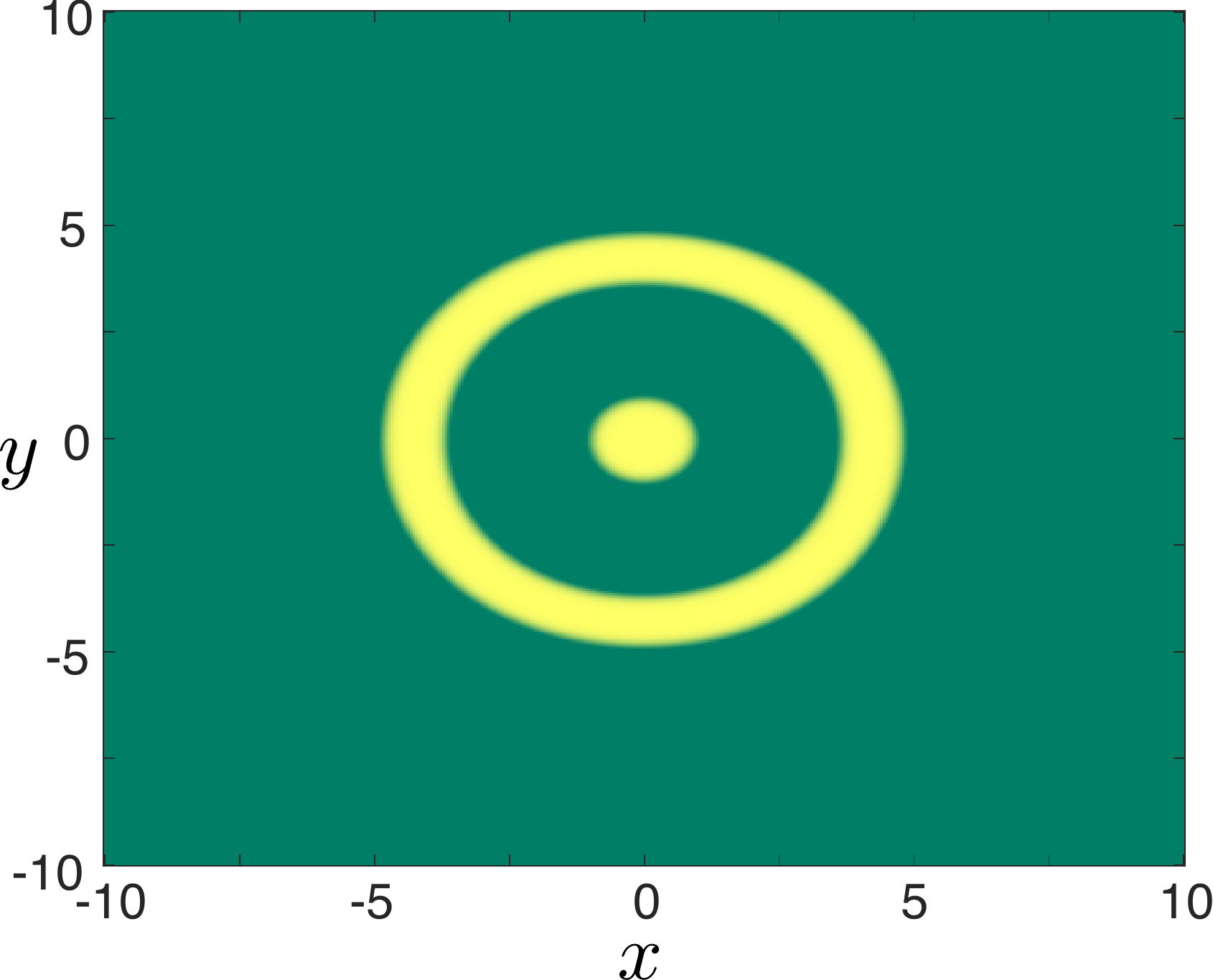}
\end{subfigure}
\hspace{.02\textwidth}
\begin{subfigure}{.3 \textwidth}
\centering
\includegraphics[width=1\linewidth]{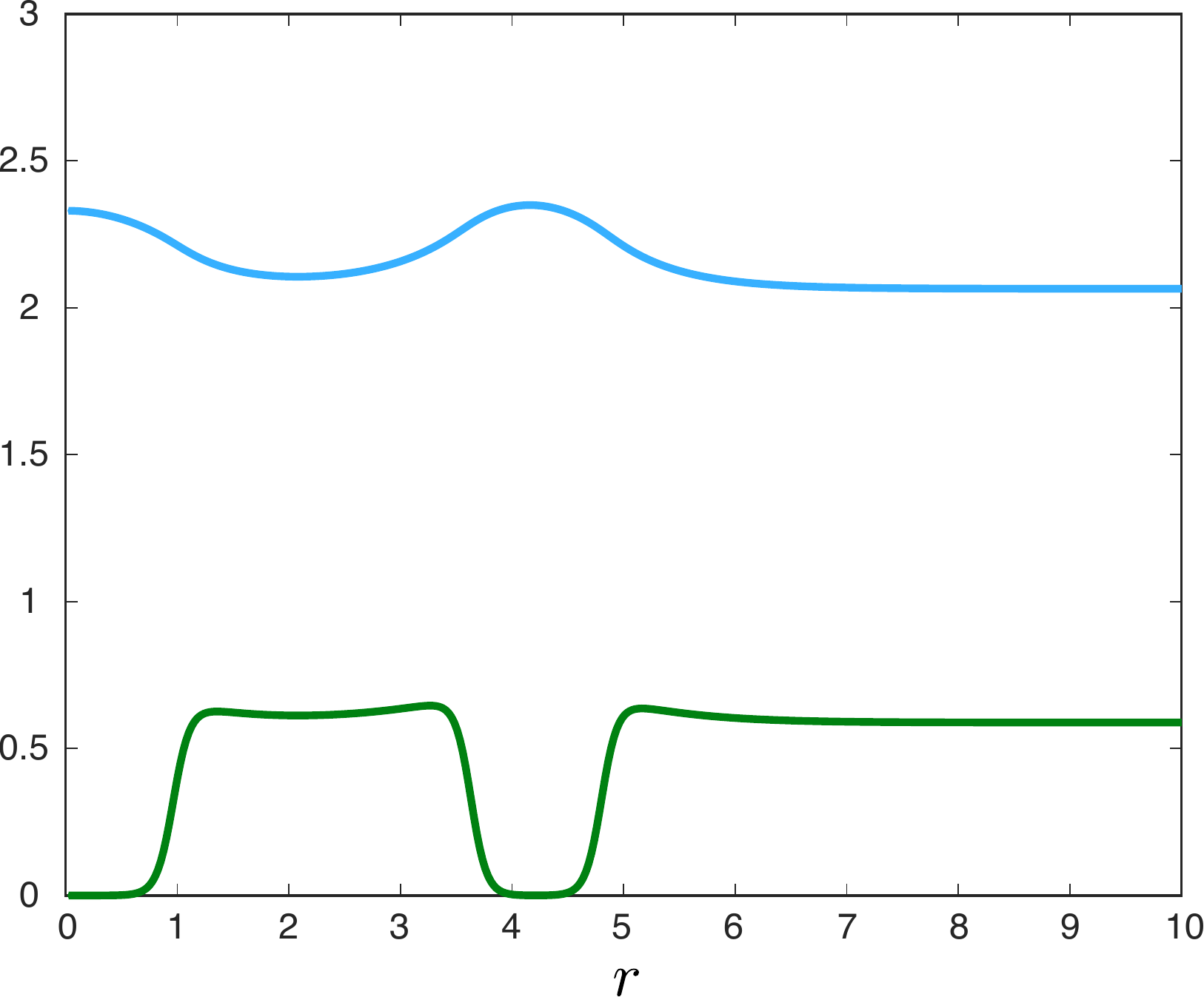}
\end{subfigure}
\hspace{.02\textwidth}
\begin{subfigure}{.3 \textwidth}
\centering
\includegraphics[width=1\linewidth]{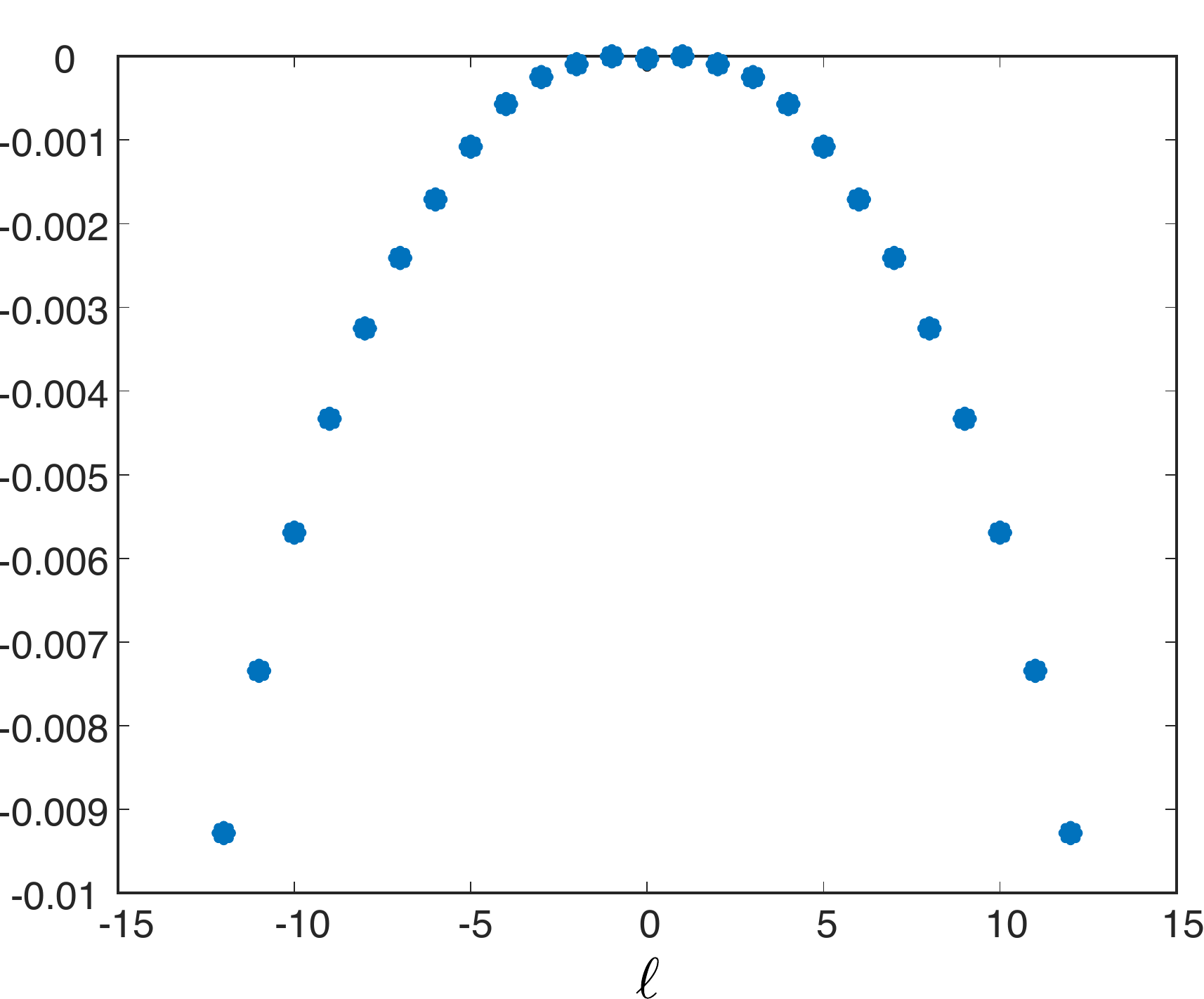}
\end{subfigure}

\caption{Radially symmetric solutions obtained for the parameter values $b=1.0, m=0.5, \delta = 0.05$ for different values of $a$. From top to bottom: spot ($a=2.55$), gap ($=2.765$), ring ($a=2.538$), and target pattern ($a=2.78$). The various solutions survive in direct numerical simulations suggesting they are indeed stable. For each row, the left panel depicts the planar profile obtained by direct numerical simulation in~\eqref{eq:modifiedKlausmeier} using finite differences for spatial discretization with periodic boundary conditions. The middle panel depicts a radial profile obtained by solving~\eqref{eq:klaus_stationary} using finite differences and Neumann boundary conditions, and the right panel depicts the critical eigenvalue $\lambda(\ell)$ for $-12\leq\ell\leq 12$, obtained by linearizing~\eqref{eq:modifiedKlausmeier} about the radial profile and using Matlab's eigs routine.}
\label{fig:stable_radial_solutions}
\end{figure}

\begin{figure}
\hspace{.01\textwidth}
\begin{subfigure}{.45 \textwidth}
\centering
\includegraphics[width=1\linewidth]{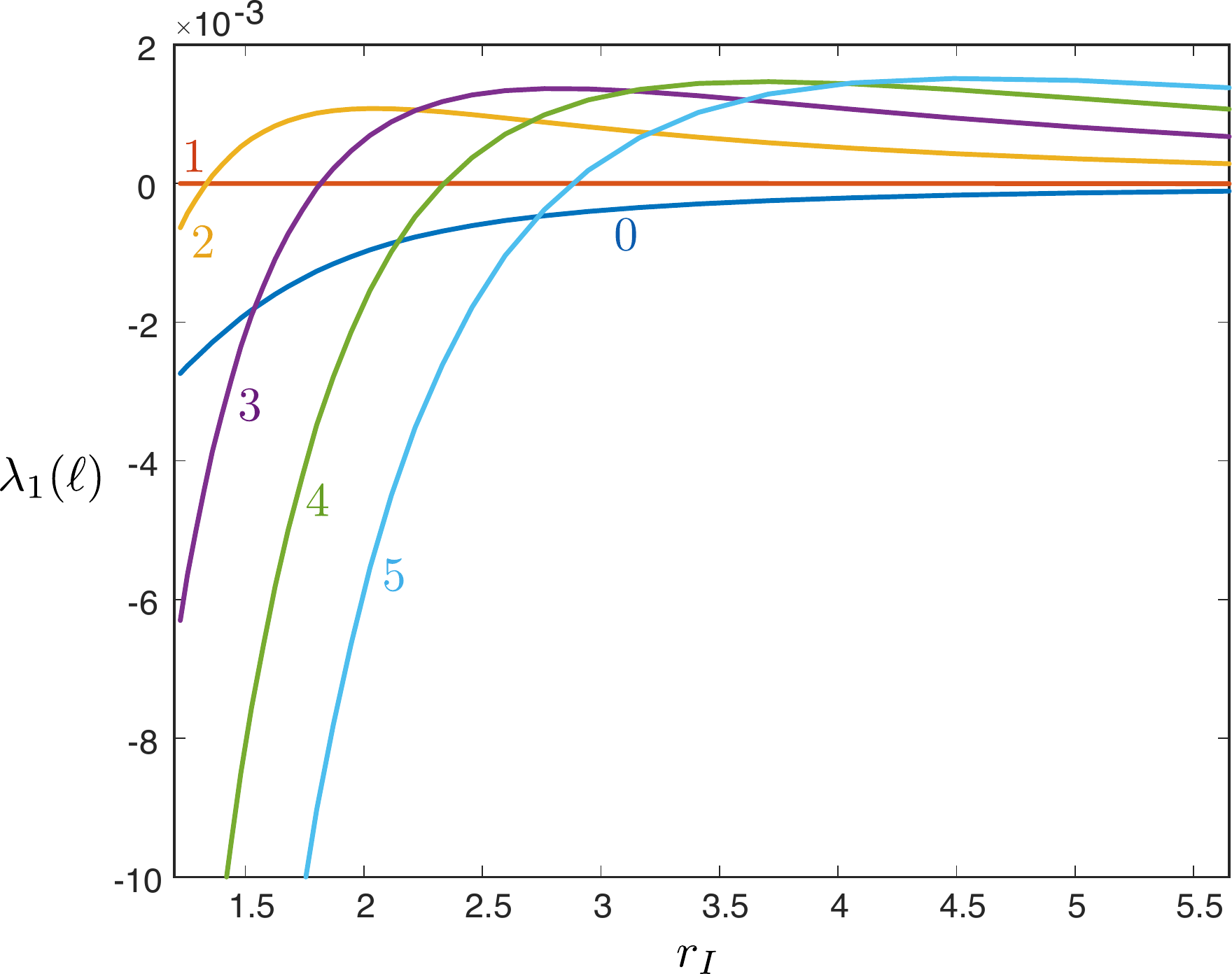}
\end{subfigure}
\hspace{.03\textwidth}
\begin{subfigure}{.45 \textwidth}
\centering
\includegraphics[width=1\linewidth]{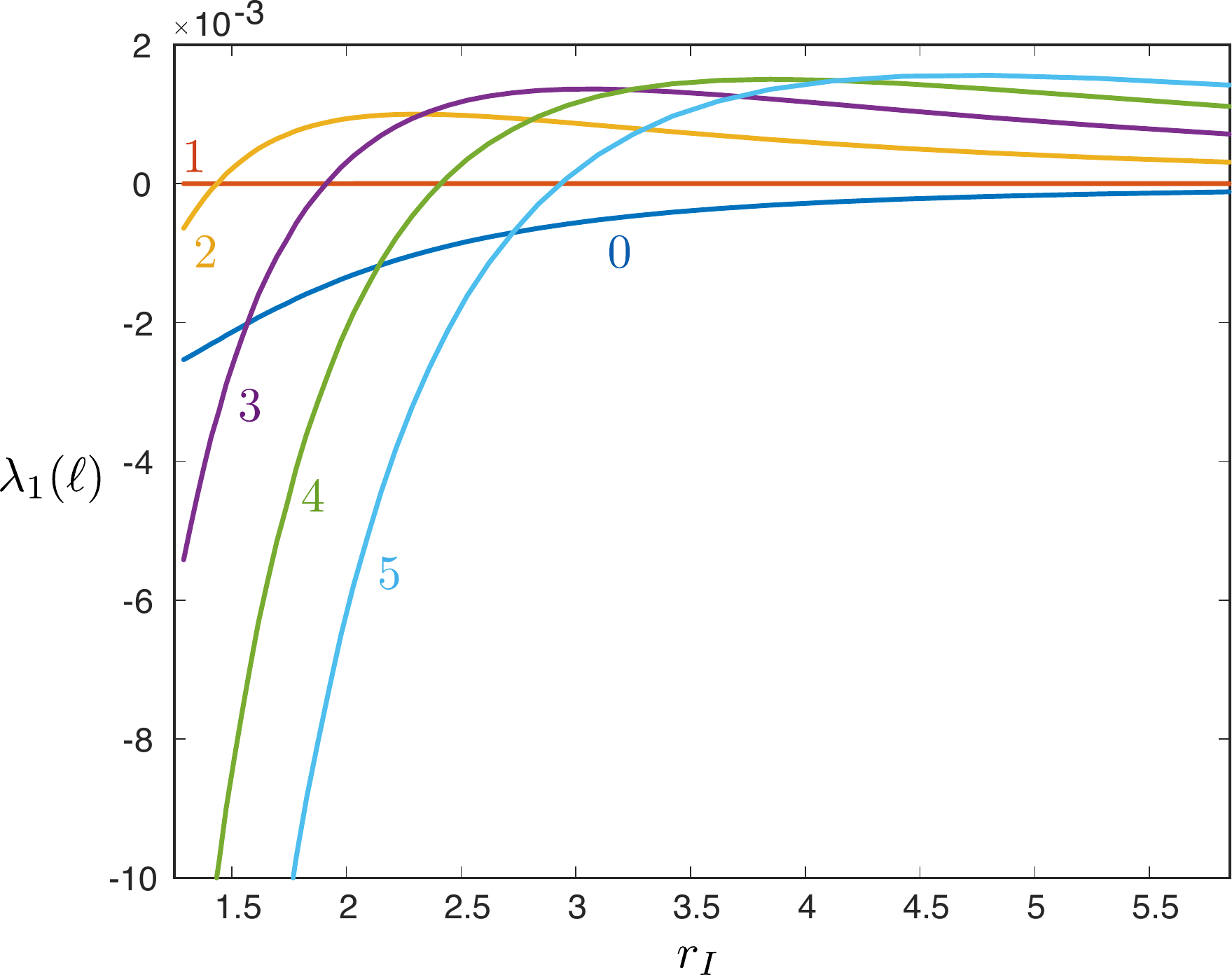}
\end{subfigure}
\caption{ {Shown is a numerical continuation of the eigenvalues $\lambda_1(\ell)$ for $\ell=0,1,2,3,4,5$ for a spot solution (left) and gap solution (right) as a function of the interface radius $r_I$. The curves were obtained by continuing the unstable spot and gap solutions from Figure~\ref{fig:unstable_spot} to the stable spot and gap solutions of Figure~\ref{fig:stable_radial_solutions} by adjusting the parameter $a$ for fixed $(b,m,\delta)$. The interface location $r_I$ was computed at each step by approximating the inflection point of the vegetation profile of the corresponding solution. From the figure, we see that the unstable eigenvalues $\lambda_1(\ell), \ell=2,3,4,5$ eventually stabilize as $r_I$ decreases. }}
\label{fig:evalcont}
\end{figure}

While the spots and gaps of Theorems~\ref{thm:spot_existence}--\ref{thm:gap_existence} are unstable for radii $r_I=\mathcal{O}(1)$ (or larger) with respect to $\delta$, the relation~\eqref{eq:lambda_1_ell_asymp} is no longer valid when $r_I$ is not large, and the analysis in~\S\ref{sec:point_spectrum} is not valid if $r_I=o(1)$ as $\delta\to0$. Hence it may be possible to find smaller spots or gaps which are stable. Figure~\ref{fig:stable_radial_solutions} depicts spot and gap solutions of smaller radii (but nearby in parameter space to those in Figure~\ref{fig:unstable_spot}), for which we see that $\lambda_1(\ell)$ is no longer well approximated by~\eqref{eq:lambda_1_ell_asymp}. We see in this case that $\lambda(\ell)$ is negative aside from the double zero eigenvalue $\lambda(\pm1)=0$ due to translation invariance.  {Figure~\ref{fig:evalcont} shows a continuation of the eigenvalues $\lambda_1(\ell), \ell=0,1,2,3,4,5$ for decreasing $r_I$ for these spot and gap solutions, where we observe that each eigenvalue eventually stabilizes as $r_I$ decreases.} Additionally, as described in~\S\ref{sec:other_radial}, we are similarly able to find (seemingly) stable radially symmetric ring and target patterns (see Figure~\ref{fig:stable_radial_solutions}), which could be obtained in a similar manner to the spots/gaps of Theorems~\ref{thm:spot_existence}--\ref{thm:gap_existence} by constructing solutions with several sharp interfaces at distinct radii $r_i$. 

Obtaining the stability of spots of smaller radii rigorously appears to be a challenging problem, and we leave this to future work; in particular our existence analysis does not immediately extend to this regime. Additionally, unlike prior works which considered the stability of radially symmetric solutions using singular perturbation methods~\cite{vHS,vHS2} in a $3$-component FitzHugh--Nagumo system, the solutions of the linearized equations~\eqref{eq:slow_+_ell} do not have explicit representations in terms of special functions, which makes it difficult to determine $\lambda_1(\ell)$ for smaller values of $r_I$. This is related to the challenges which arise in the existence analysis in~\S\ref{sec:existence} due to the nonlinear reduced flow on the slow manifold $\mathcal{M}^+_\delta$.

Another natural line of further research concerns investigating whether the present insights obtained on specific model (\ref{eq:modifiedKlausmeier}) can be lifted to the general setting of the 2-component singularly perturbed reaction diffusion systems considered in the companion paper \cite{CDLOR}. In that paper, we study the (in)stability of planar fronts with respect to longitudinal perturbations, as we did for (\ref{eq:modifiedKlausmeier}) in section \S\ref{sec:fronts}, and derive two general, and relatively simple, criteria on the emergence of the sideband instability mechanism (that typically leads to finger-like patterns). The present analysis indicates that the same mechanism may drive the (in)stability of large spots and gaps in the general setting of \cite{CDLOR}, although one should not underestimate the technicalities involved in establishing the counterparts of Theorems \ref{thm:spot_existence} and \ref{thm:gap_existence} for the general model. It is less clear from the combination of the present insights and those of \cite{CDLOR}, under which conditions spots and gaps with a radius $r_I$ of $\mathcal{O}(1)$ will be unstable (as is the case for (\ref{eq:modifiedKlausmeier})). The nature of the analysis in section \S\ref{sec:largel} suggests that it is possible to derive a general (in)stability result for spots and gaps with radius $r_I = \mathcal{O}(1)$ against perturbations with $|\ell| \gg 1$. This suggests that also in the general setting, spots and gaps with radius $r_I$ `sufficiently small' are potentially the `most stable' (radially symmetric) localized patterns. Thus, the issue of the existence and stability of spots and gaps of sufficiently small radius is a central question and resolving that question may explain the abundance of `spikes' in the literature on localized patterns in singularly perturbed reaction-diffusion systems -- see~\cite{chen2011stability,kolokolnikov2009spot,wei2013mathematical} and the references therein and Remark \ref{r:spikesandmore}. These spikes have a fully homoclinic nature, in the sense that they are not close to a concatenation of almost heteroclinic orbits: away from the slow manifold $\mathcal{M}^0_\delta$ they only follow the fast (spatial) dynamics, they do not follow the slow flow on a second slow manifold $\mathcal{M}^+_\delta$ -- as is the case for the patterns constructed here (with $r_I = \mathcal{O}(1)$). Thus, by studying spot and gap patterns with radius $r_I$ decreasing from being $\mathcal{O}(1)$ to asymptotically small, one needs to zoom in on the subtle process through which a localized pattern detaches from $\mathcal{M}^+_\delta$ during its jump away from and back to $\mathcal{M}^0_\delta$ -- see also \cite{KBD22}. 

\begin{figure}
\hspace{.05\textwidth}
\begin{subfigure}{.35 \textwidth}
\centering
\includegraphics[width=1\linewidth]{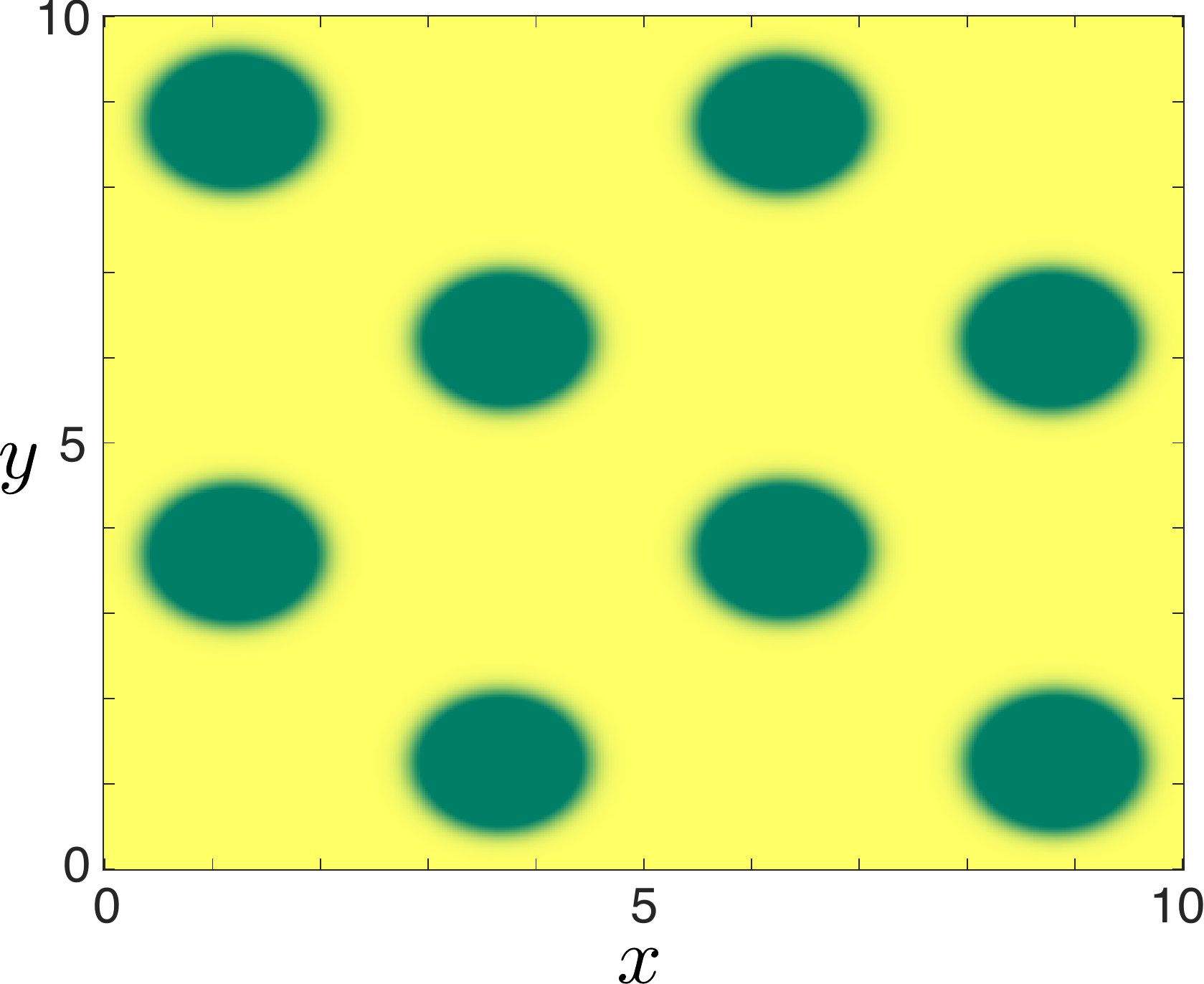}
\end{subfigure}
\hspace{.1\textwidth}
\begin{subfigure}{.35 \textwidth}
\centering
\includegraphics[width=1\linewidth]{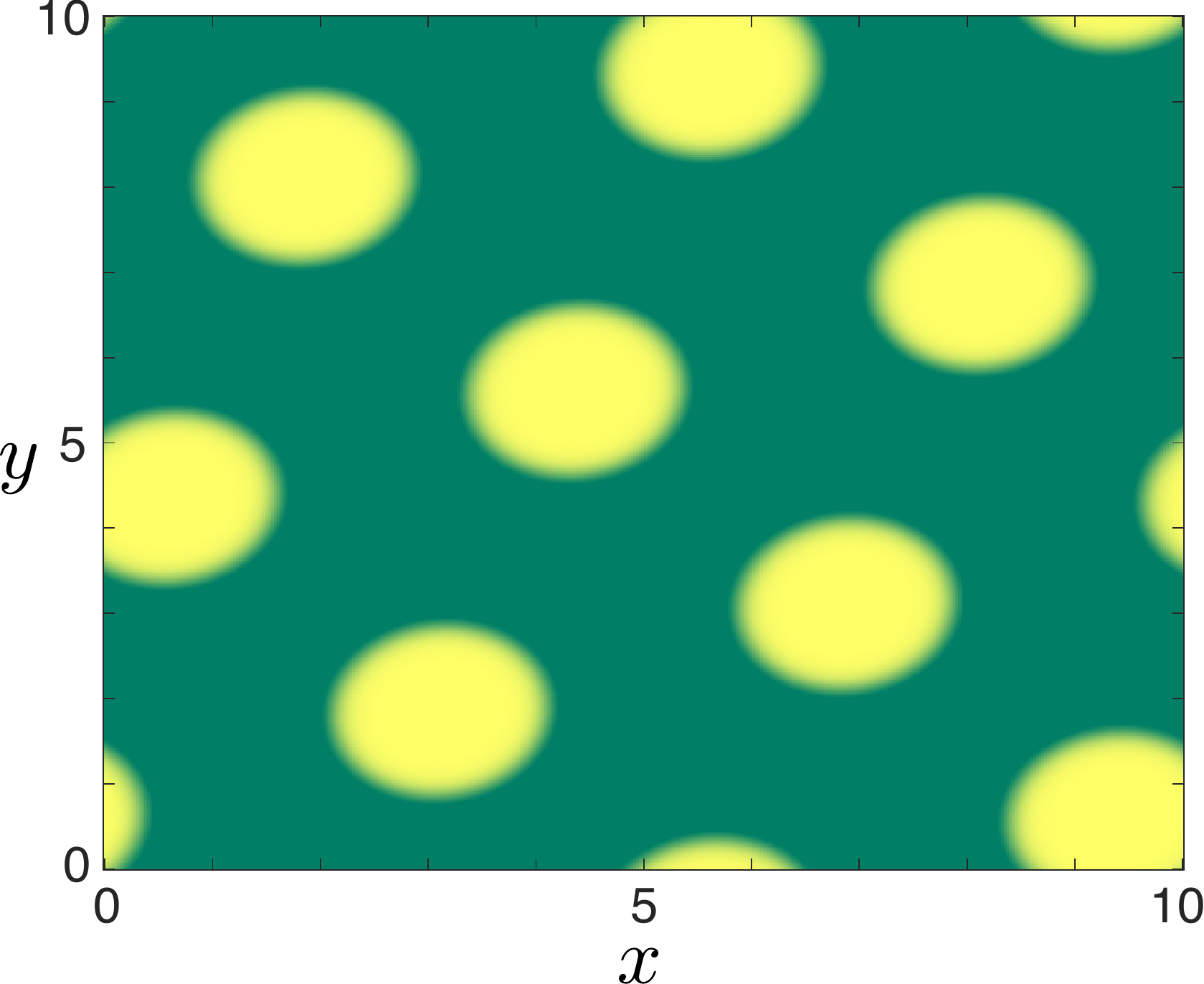}
\end{subfigure}
\hspace{.1\textwidth}
\caption{(Left) Spot pattern solution obtained for $a=2.55, b=1.0, m=0.5, \delta = 0.05$. (Right) Gap pattern solution obtained for $a=2.765, b=1.0, m=0.5, \delta = 0.05$. }
\label{fig:patterns}
\end{figure}

Lastly, we briefly describe the appearance of far-from-onset spot \emph{patterns} in~\eqref{eq:modifiedKlausmeier}. While the analysis of Theorems~\ref{thm:spot_existence}--\ref{thm:gap_existence} only applies to the construction of a single spot or gap solution, we anticipate that one could construct periodic spot or gap patterns by tiling the plane with well-separated copies of the primary spot or gap solution. See Figure~\ref{fig:patterns} for results of direct numerical simulations in~\eqref{eq:modifiedKlausmeier} which result in the appearance of spatially periodic spot and gap lattice patterns. While the construction of such patterns is beyond the scope of this work, we note that a spatial dynamics approach, such as that described in~\cite{scheel2003radially} could be used to construct such large amplitude spot patterns. However, the question of \emph{stability} of the resulting patterns is likely very challenging. 

{\bf Acknowledgment.} AD acknowledges the hospitality
of Arnd Scheel and the School of Mathematics during his stay at the University of Minnesota as Ordway
Visiting Professor.

{\bf Funding.} EB and LL were supported by the NSF REU program through the grant DMS-2016216. PC was supported by the NSF through the grants DMS-2016216 and DMS-2105816.

{\bf Conflict of interest.} The authors have no relevant financial or non-financial interests to disclose.

{\bf Data availability.} The datasets generated during and/or analysed during the current study are available from the corresponding author on reasonable request.

\bibliographystyle{abbrv}
\bibliography{my_bib}

\begin{thebibliography}{10}

\bibitem{barbier2014case}
N.~Barbier, P.~Couteron, and V.~Deblauwe.
\newblock Case study of self-organized vegetation patterning in dryland regions
  of central {A}frica.
\newblock In {\em Patterns of land degradation in drylands}, pages 347--356.
  Springer, 2014.

\bibitem{BCD}
R.~Bastiaansen, P.~Carter, and A.~Doelman.
\newblock Stable planar vegetation stripe patterns on sloped terrain in dryland
  ecosystems.
\newblock {\em Nonlinearity}, 32(8):2759, 2019.

\bibitem{CD}
P.~Carter and A.~Doelman.
\newblock Traveling stripes in the klausmeier model of vegetation pattern
  formation.
\newblock {\em SIAM Journal on Applied Mathematics}, 78(6):3213--3237, 2018.

\bibitem{CDLOR}
P.~Carter, A.~Doelman, K.~Lilly, E.~Obermayer, and S.~Rao.
\newblock Criteria for the (in) stability of planar interfaces in singularly
  perturbed 2-component reaction-diffusion equations.
\newblock {\em arXiv preprint arXiv:2207.05128}, 2022.

\bibitem{chen2011stability}
W.~Chen and M.~J. Ward.
\newblock The stability and dynamics of localized spot patterns in the
  two-dimensional {G}ray--{S}cott model.
\newblock {\em SIAM Journal on Applied Dynamical Systems}, 10(2):582--666,
  2011.

\bibitem{deblauwe2012determinants}
V.~Deblauwe, P.~Couteron, J.~Bogaert, and N.~Barbier.
\newblock Determinants and dynamics of banded vegetation pattern migration in
  arid climates.
\newblock {\em Ecological monographs}, 82(1):3--21, 2012.

\bibitem{deblauwe2011environmental}
V.~Deblauwe, P.~Couteron, O.~Lejeune, J.~Bogaert, and N.~Barbier.
\newblock Environmental modulation of self-organized periodic vegetation
  patterns in {S}udan.
\newblock {\em Ecography}, 34(6):990--1001, 2011.

\bibitem{handbook}
{\it NIST Digital Library of Mathematical Functions}.
\newblock http://dlmf.nist.gov/, Release 1.1.1 of 2021-03-15.
\newblock F.~W.~J. Olver, A.~B. {Olde Daalhuis}, D.~W. Lozier, B.~I. Schneider,
  R.~F. Boisvert, C.~W. Clark, B.~R. Miller, B.~V. Saunders, H.~S. Cohl, and
  M.~A. McClain, eds.

\bibitem{doelman2022slow}
A.~Doelman.
\newblock Slow localized patterns in singularly perturbed two-component
  reaction--diffusion equations.
\newblock {\em Nonlinearity}, 35(7):3487, 2022.

\bibitem{doelman2002homoclinic}
A.~Doelman and H.~van~der Ploeg.
\newblock Homoclinic stripe patterns.
\newblock {\em SIAM J. Applied Dynamical Systems}, 1(1):65--104, 2002.

\bibitem{eigentler2021species}
L.~Eigentler.
\newblock Species coexistence in resource-limited patterned ecosystems is
  facilitated by the interplay of spatial self-organisation and intraspecific
  competition.
\newblock {\em Oikos}, 130(4):609--623, 2021.

\bibitem{fernandez2019front}
C.~Fernandez-Oto, O.~Tzuk, and E.~Meron.
\newblock Front instabilities can reverse desertification.
\newblock {\em Physical review letters}, 122(4):048101, 2019.

\bibitem{gandhi2018influence}
P.~Gandhi, L.~Werner, S.~Iams, K.~Gowda, and M.~Silber.
\newblock A topographic mechanism for arcing of dryland vegetation bands.
\newblock {\em Journal of The Royal Society Interface}, 15(147), 2018.

\bibitem{getzin2016discovery}
S.~Getzin, H.~Yizhaq, B.~Bell, T.~E. Erickson, A.~C. Postle, I.~Katra, O.~Tzuk,
  Y.~R. Zelnik, K.~Wiegand, T.~Wiegand, et~al.
\newblock Discovery of fairy circles in australia supports self-organization
  theory.
\newblock {\em Proceedings of the National Academy of Sciences},
  113(13):3551--3556, 2016.

\bibitem{gowda2016assessing}
K.~Gowda, Y.~Chen, S.~Iams, and M.~Silber.
\newblock Assessing the robustness of spatial pattern sequences in a dryland
  vegetation model.
\newblock {\em Proceedings of the Royal Society A: Mathematical, Physical and
  Engineering Sciences}, 472(2187):20150893, 2016.

\bibitem{gowda2018signatures}
K.~Gowda, S.~Iams, and M.~Silber.
\newblock Signatures of human impact on self-organized vegetation in the {H}orn
  of {A}frica.
\newblock {\em Scientific reports}, 8(1):3622, 2018.

\bibitem{gowda2014transitions}
K.~Gowda, H.~Riecke, and M.~Silber.
\newblock Transitions between patterned states in vegetation models for
  semiarid ecosystems.
\newblock {\em Physical Review E}, 89(2):022701, 2014.

\bibitem{hill2021existence}
D.~J. Hill.
\newblock Existence of localised radial patterns in a model for dryland
  vegetation.
\newblock {\em arXiv preprint arXiv:2107.12678}, 2021.

\bibitem{jaibi2020existence}
O.~Jaibi, A.~Doelman, M.~Chirilus-Bruckner, and E.~Meron.
\newblock The existence of localized vegetation patterns in a systematically
  reduced model for dryland vegetation.
\newblock {\em Physica D: Nonlinear Phenomena}, 412:132637, 2020.

\bibitem{klausmeier1999regular}
C.~A. Klausmeier.
\newblock Regular and irregular patterns in semiarid vegetation.
\newblock {\em Science}, 284(5421):1826--1828, 1999.

\bibitem{KBD22}
D.~Kok, R.~Bastiaansen, and A.~Doelman.
\newblock From 1-pulse to 2-front stripe patterns. part i: away from the
  transition point.
\newblock {\em in preparation}, 2022.

\bibitem{kolokolnikov2007self}
T.~Kolokolnikov, M.~Ward, and J.~Wei.
\newblock Self-replication of mesa patterns in reaction--diffusion systems.
\newblock {\em Physica D: Nonlinear Phenomena}, 236(2):104--122, 2007.

\bibitem{kolokolnikov2009spot}
T.~Kolokolnikov, M.~J. Ward, and J.~Wei.
\newblock Spot self-replication and dynamics for the schnakenburg model in a
  two-dimensional domain.
\newblock {\em Journal of nonlinear science}, 19(1):1--56, 2009.

\bibitem{ludwig2005vegetation}
J.~Ludwig, B.~Wilcox, D.~Breshears, D.~Tongway, and A.~Imeson.
\newblock Vegetation patches and runoff-erosion as interacting ecohydrological
  processes in semiarid landscapes.
\newblock {\em Ecology}, 86(2):288--297, 2005.

\bibitem{macfadyen1950vegetation}
W.~Macfadyen.
\newblock Vegetation patterns in the semi-desert plains of {B}ritish
  {S}omaliland.
\newblock {\em Geographical J.}, 116(4/6):199--211, 1950.

\bibitem{may1977thresholds}
R.~M. May.
\newblock Thresholds and breakpoints in ecosystems with a multiplicity of
  stable states.
\newblock {\em Nature}, 269(5628):471, 1977.

\bibitem{meron2018patterns}
E.~Meron.
\newblock From patterns to function in living systems: Dryland ecosystems as a
  case study.
\newblock {\em Annual Review of Condensed Matter Physics}, 9:79--103, 2018.

\bibitem{noy1975stability}
I.~Noy-Meir.
\newblock Stability of grazing systems: an application of predator-prey graphs.
\newblock {\em The Journal of Ecology}, pages 459--481, 1975.

\bibitem{ravi2017ecohydrological}
S.~Ravi, L.~Wang, K.~F. Kaseke, I.~V. Buynevich, and E.~Marais.
\newblock Ecohydrological interactions within ``fairy circles'' in the {N}amib
  {D}esert: Revisiting the self-organization hypothesis.
\newblock {\em Journal of Geophysical Research: Biogeosciences},
  122(2):405--414, 2017.

\bibitem{rietkerk2021evasion}
M.~Rietkerk, R.~Bastiaansen, S.~Banerjee, J.~van~de Koppel, M.~Baudena, and
  A.~Doelman.
\newblock Evasion of tipping in complex systems through spatial pattern
  formation.
\newblock {\em Science}, 374(6564):eabj0359, 2021.

\bibitem{rietkerk2002self}
M.~Rietkerk, M.~Boerlijst, F.~van Langevelde, R.~HilleRisLambers, J.~van~de
  Koppel, L.~Kumar, H.~Prins, and A.~de~Roos.
\newblock Self-organization of vegetation in arid ecosystems.
\newblock {\em The American Naturalist}, 160(4):524--530, 2002.

\bibitem{rietkerk2004self}
M.~Rietkerk, S.~Dekker, P.~de~Ruiter, and J.~van~de Koppel.
\newblock Self-organized patchiness and catastrophic shifts in ecosystems.
\newblock {\em Science}, 305(5692):1926--1929, 2004.

\bibitem{rietkerk2008regular}
M.~Rietkerk and J.~van~de Koppel.
\newblock Regular pattern formation in real ecosystems.
\newblock {\em Trends in Ecology and Evolution}, 23(3):169--175, 2008.

\bibitem{rietkerk1997site}
M.~Rietkerk, F.~van~den Bosch, and J.~van~de Koppel.
\newblock Site-specific properties and irreversible vegetation changes in
  semi-arid grazing systems.
\newblock {\em Oikos}, pages 241--252, 1997.

\bibitem{ruiz2017fairy}
D.~Ruiz-Reyn{\'e}s, D.~Gomila, T.~Sintes, E.~Hern{\'a}ndez-Garc{\'\i}a,
  N.~Marb{\`a}, and C.~M. Duarte.
\newblock Fairy circle landscapes under the sea.
\newblock {\em Science advances}, 3(8):e1603262, 2017.

\bibitem{scheel2003radially}
A.~Scheel.
\newblock {\em Radially symmetric patterns of reaction-diffusion systems},
  volume 165.
\newblock American Mathematical Soc., 2003.

\bibitem{schlesinger1990biological}
W.~Schlesinger, J.~Reynolds, G.~Cunningham, L.~Huenneke, W.~Jarrell,
  R.~Virginia, and W.~Whitford.
\newblock Biological feedbacks in global desertification.
\newblock {\em Science}, 247(4946):1043--1048, 1990.

\bibitem{sewalt2017spatially}
L.~Sewalt and A.~Doelman.
\newblock Spatially periodic multipulse patterns in a generalized
  {K}lausmeier--{G}ray--{S}cott model.
\newblock {\em SIAM Journal on Applied Dynamical Systems}, 16(2):1113--1163,
  2017.

\bibitem{siero2015striped}
E.~Siero, A.~Doelman, M.~Eppinga, J.~Rademacher, M.~Rietkerk, and K.~Siteur.
\newblock Striped pattern selection by advective reaction-diffusion systems:
  Resilience of banded vegetation on slopes.
\newblock {\em Chaos}, 25(3):036411, 2015.

\bibitem{valentin1999soil}
C.~Valentin, J.~d'Herb{\`e}s, and J.~Poesen.
\newblock Soil and water components of banded vegetation patterns.
\newblock {\em Catena}, 37(1):1--24, 1999.

\bibitem{vHS}
P.~Van~Heijster and B.~Sandstede.
\newblock Planar radial spots in a three-component fitzhugh--nagumo system.
\newblock {\em Journal of Nonlinear Science}, 21(5):705--745, 2011.

\bibitem{vHS2}
P.~Van~Heijster and B.~Sandstede.
\newblock Bifurcations to travelling planar spots in a three-component
  {F}itz{H}ugh--{N}agumo system.
\newblock {\em Physica D: Nonlinear Phenomena}, 275:19--34, 2014.

\bibitem{von2001diversity}
J.~von Hardenberg, E.~Meron, M.~Shachak, and Y.~Zarmi.
\newblock Diversity of vegetation patterns and desertification.
\newblock {\em Physical Review Letters}, 87(19):198101, 2001.

\bibitem{wei2013mathematical}
J.~Wei and M.~Winter.
\newblock {\em Mathematical aspects of pattern formation in biological
  systems}, volume 189.
\newblock Springer Science \& Business Media, 2013.

\bibitem{wilcox2003ecohydrology}
B.~Wilcox, D.~Breshears, and C.~Allen.
\newblock Ecohydrology of a resource-conserving semiarid woodland: effects of
  scale and disturbance.
\newblock {\em Ecological Monographs}, 73(2):223--239, 2003.

\bibitem{zelnik2015gradual}
Y.~R. Zelnik, E.~Meron, and G.~Bel.
\newblock Gradual regime shifts in fairy circles.
\newblock {\em Proceedings of the National Academy of Sciences},
  112(40):12327--12331, 2015.

\end{thebibliography}

\end{document}